\documentclass[11pt]{article}

\usepackage{fancyhdr, amsmath, amssymb, amsfonts, url, mathrsfs, graphicx, epsfig,  amsthm}
\usepackage[colorlinks=true, linktocpage=true, linkcolor=red, citecolor=red]{hyperref}
\usepackage{cleveref}
\usepackage{comment}
\usepackage{tikz}
\usepackage{pst-node}
\usetikzlibrary{hobby}
\usepackage{sectsty}
\partfont{\Large}
\sectionfont{\large}
\subsectionfont{\normalsize}
\usepackage[ragged]{sidecap} 
\newtheorem{Definition}{Definition}[subsection]
\newtheorem{Theorem}{Theorem}

\newtheorem{Proposition}[Definition]{Proposition}
\newtheorem{Lemma}[Definition]{Lemma}
\newtheorem{Remark}{Remark}
\newtheorem{Corollary}[Definition]{Corollary}
\newtheorem*{Claim}{Claim}
\newtheorem*{Example}{Example}
\numberwithin{equation}{subsection}

\begin{document}
\title{\textbf{Rapid mixing for compact group extensions of hyperbolic flows}}
\date{\today}
\author{Mark Pollicott and Daofei Zhang\thanks{The first author is partly supported by ERC-Advanced Grant 833802-Resonances. The second author is supported by the Warwick Mathematics Institute Centre for Doctoral Training, and gratefully acknowledges funding from ERC-Advanced Grant 833802-Resonances.}}

\newcommand{\Addresses}{{
  \bigskip
  \footnotesize

  M.~Pollicott, \text{Department of Mathematics, University of Warwick, Coventry, CV4 7AL, UK.}\par\nopagebreak
  \textit{E-mail}: \texttt{masdbl@warwick.ac.uk}

  \medskip

  D.~Zhang, \text{Department of Mathematics, University of Warwick, Coventry, CV4 7AL, UK.}\par\nopagebreak
  \textit{E-mail}: \texttt{Daofei.Zhang@warwick.ac.uk}
}}

\maketitle

\abstract{In this article, we  give explicit  conditions for compact group extensions of hyperbolic flows 
(including  geodesic flows on negatively curved manifolds)
to exhibit quantifiable rates of mixing (or decay of correlations) with respect to  the natural probability measures,  which are locally the product of a Gibbs measure for a H\"older potential and the Haar measure.  More precisely, we show that the mixing rate with respect to H\"older functions will be faster than any given polynomial (i.e., rapid mixing). 
We also give error estimates on the equidistribution of the holonomy around closed orbits.
 In particular, these results   apply to some frame flows for  manifolds with negative sectional curvatures.  
}
\bigskip

\thispagestyle{empty}
\tableofcontents
\thispagestyle{empty}

\bigskip

\section{Introduction}\label{introduction}
\bigskip

Hyperbolic flows have been the subject of  considerable attention  over the past six  decades and the modern theory dates back to the seminal works of Smale and Anosov in the 1960s where they appeared as the natural generalizations of geodesic flows on negatively curved manifolds.  The ergodic properties of such flows developed soon afterwards  via  the work of Sinai, Ruelle and Bowen.  One of the most significant aspects of their study is the speed of mixing (or the rate of decay of correlations) with respect to Gibbs measures for H\"older continuous functions.   There are examples for which the mixing may be arbitrarily slow \cite{Pol85}.  But in  contrast to this there are positive results that show 
rapid   (i.e., super polynomial) mixing  under some mild Diophantine condition \cite{Dol98b}.

On the other hand, 
one of the classical constructions in smooth ergodic theory is that of  compact group extensions of hyperbolic diffeormorphisms. 
In this setting, one can consider the  invariant probability measure given locally by the product of a Gibbs measure for a H\"older continuous potential and the Haar measure for the group.
 For such discrete maps it is known that one  can again get  rapid mixing under some suitable Diophantine condition \cite{Dol02}.  Therefore, it is very natural to try to extend these  ideas to the context of compact group extensions of  hyperbolic flows, as we shall in this text.
A classical example  of a hyperbolic flow is the geodesic flow on a $n$-dimensional compact manifold ($n\geq 3$) with  negative sectional curvatures and  the associated frame flow, which is a $SO(n-1)$ extension of the geodesic flow, is a natural  compact group extension. 


To formulate our results more precisely, we now introduce some notation. 
Let $M$ be a compact smooth manifold, and let $g_t: \Lambda \to \Lambda\subset M$ be a $C^\infty$ hyperbolic flow. Let  $\mu_{\Phi}$ be the Gibbs measure of a H\"older continuous potential $\Phi: \Lambda \to \mathbb R$.
Let $G$ be a compact connected Lie group, and let $\widehat M$ be a smooth $G$-bundle over $M$.   
We can denote the natural  projection by $\varrho: \widehat M \to M$ then $ G = \varrho^{-1}(x)$ for all $x\in M$.  The dynamical system of interest is a smooth $G$-extension, denoted as $f_t: \widehat \Lambda \to \widehat \Lambda$, of $g_{t}:\Lambda\to\Lambda$, where $\widehat{\Lambda}:=\varrho^{-1}(\Lambda)$.  Let $\widehat \mu_{\Phi}$ be the $f_{t}$-invariant measure which is the local product measure of $\mu_{\Phi}$ and the normalized Haar measure $\hbox{Haar}_G$ on $G$.

   \begin{figure}[h!]
          \centerline{
    \begin{tikzpicture}[thick,scale=0.70, every node/.style={scale=1.5}]
\draw[->, black, thick] (1.5,0)-- (3.5,0);
\draw[->, black, thick] (1.5,4)-- (3.5,4);
 \node at (0, 0) {$(\Lambda, \mu_\Phi)$};
  \node at (5, 0) {$(\Lambda, \mu_\Phi)$};
   \node at (0, 4) {$(\widehat\Lambda, \widehat \mu_\Phi)$};
  \node at (5, 4) {$(\widehat\Lambda,\widehat \mu_\Phi)$};
  \draw[->, black, thick] (0,3)-- (0,1);
    \draw[->, black, thick] (5,3)-- (5,1);
      \node at (2.5, 5.1) {$f_t$};
            \node at (2.5, -0.6) {$g_t$};
            \node at (-0.5, 2.0) {$\rho$};
              \node at (5.5, 2.0) {$\rho$};
\end{tikzpicture}
}
\caption{The commutative diagram for the flows $f_t$ and $g_t$}
\end{figure}

\begin{Example}
A simple   class of examples are skew product flows. In this case, $\widehat \Lambda = \Lambda \times G$ and 
$f_t(x,u) = (g_t(x), \omega(w,t) u)$, where $\omega: \Lambda \times \mathbb R \to G$ is a smooth cocycle
satisfying $\omega(x, t + s) = \omega(g_s(x),t) \omega(x, s)$ (with $x \in \Lambda$ and $s,t \in \mathbb R$)
and $\widehat \mu_{\Phi} = \mu_{\Phi} \times \hbox{\rm Haar}_G$.
\end{Example}

We will be interested in the asymptotic behaviour of the following quantity.

\begin{Definition}
For $E,F : \widehat M \to \mathbb C$, their \emph{correlation function} is defined by 
$$
\rho_{E,F}(t): = \int E\circ f_t  \cdot \overline{F} d\widehat \mu_{\Phi}  - \int E d\widehat \mu_{\Phi} \int \overline{F} d\widehat \mu_{\Phi}, \hbox{ for } t \in \mathbb R.
$$
\end{Definition}

The flow $f_t : \widehat \Lambda \to \widehat \Lambda$ is  said to be  \emph{strong mixing} with respect to the measure $\widehat \mu_{\Phi}$  when  $\rho_{E,F}(t) \to 0$ as $t \to +\infty$ for any $E$ and $F\in L^{2}(\widehat{\mu}_{\Phi})$.
Fundamental progress on when such flows $f_t: \widehat \Lambda \to \widehat \Lambda$ are ergodic  and Bernoulli (and thus strong mixing) was made  by Brin and others when $G$ is a Lie group  using  an   accessibility hypothesis which can be defined via the strong stable and strong unstable manifolds for the flow 
$f_t$  \cite{Bri82}.  In particular, this is expressed in terms of the density of the associated Brin group $H \subset G$ whose elements are given by holonomies induced by $f_t$ along stable and unstable manifolds of the underlying hyperbolic flow $g_t$.
\footnote{The Brin group is defined up to conjugacy, but the property we use is conjugacy invariant.}
In this article, we  will show that the same hypotheses  always  lead to stronger  effective estimates for $\rho_{E,F}(t)$ of the following form.

\begin{Definition}\label{Def 1.2}
We say that the flow $f_{t}$ is \emph{rapidly mixing} (with respect to the measure $\widehat \mu_{\Phi}$) if 
for any $n\in\mathbb{N}^{+}$ there exist $C_1>0$ and $k\in\mathbb{N}^{+}$ such that
for any
$E, F\in C^{k}(\widehat{M})$ we have
\begin{equation}\label{1.0.1}
|\rho_{E,F}(t)| \leq C_1||E||_{C^{k}}||F||_{C^{k}}\frac{1}{t^{n}}, \hbox{ for all  } t > 0. 
\end{equation}
\end{Definition}

We  obviously need  to impose some  additional hypotheses on the underlying hyperbolic flow  $g_t$ to ensure it is rapidly mixing as well with respect to $\mu_{\Phi}$\footnote{The definition is an analogous manner to Definition \ref{Def 1.2}.}
(particularly since 
 clearly for $f_t$ to be rapidly mixing we necessarily require that $g_t$ is rapidly mixing).
The technical condition  we actually  use  in the proofs of our results is a Diophantine hypothesis (defined in \S \ref{subsec 5.1}).  However, for clarity of exposition for the present we can consider  the better known, but stronger,    condition of density of the Brin group  when we additionally assume $G$ is semi-simple.

\begin{Theorem}\label{thm1}
Assume that  $G$ is semi-simple  and  the Brin group for the flow $f_t$ is dense in $G$
(i.e., $\overline H = G$) then $f_{t}$ is rapidly mixing with respect to $\widehat \mu_{\Phi}$.
\end{Theorem}

A special case of Theorem \ref{thm1} was previously established in \cite{Fie05}
for typical cocycles and a special class of functions.
We will present a 
more general result
than Theorem \ref{thm1}
 under a milder (but more technical) Diophantine hypothesis appears in Theorem \ref{Rapid mixing for compact group extension of hyperbolic flow}.

Under the same hypothesis as in Theorem \ref{thm1}, we also obtain the superpolynomial equidistribution, namely, a superpolynomial error term in the equidistribution of the holonomy around closed orbits. Let $\tau$ represent a closed orbit of $g_{t}$, and denote its least period as $\ell_{\tau}$. Each $\tau$ induces a conjugacy class $[\tau]$ in $G$ which is called the holonomy class of $\tau$. For any $T>0$, let $\pi(T)$ be the collection of prime closed orbits $\tau$ with $\ell_{\tau}\le T$. 

\begin{Theorem}\label{thm2}
Assume that $G$ is semi-simple and the Brin group for $f_t$ is dense in $G$ (i.e., $\overline H = G$).
For any $n\in\mathbb{N}^{+}$  there exist $C_2>0$ and $k\in\mathbb{N}^{+}$ such that for any class function $F\in C^{k}(G)$,
\begin{equation}\label{1.0.2}
\bigg|\dfrac{1}{\#\pi(T)}\sum_{
\ell_\tau \leq T
}F([\tau])-\int_{G}F(g)d(\hbox{\rm Haar}_G)(g)\bigg|\le C_2||F||_{C^{k}}\frac{1}{{T^{n}}}, 
\hbox{  for  all $T>1$, }
\end{equation}
where $\hbox{\rm Haar}_G$ is the normalized Haar measure on $G$.
\end{Theorem}
\vspace{5pt}

This problem was previously studied in \cite{SarWak99} in the particular case of constant negative curvature geodesic flows.  Theorem \ref{thm2}  improves on results in \cite{Pol08} where only a restricted class of  test functions $F$ were considered.

We now  recall  one of the principal  examples.

\begin{Example}
\footnote{This motivates our prior use of the notation $g_t$ for the hyperbolic flow and $f_t$ for the compact group extension flow as a useful mnemonic even when we are working in greater generality}
Let  $V$ be a compact manifold of negative
curvature with dimension $d =\dim(V) \geq 3$, and let $M$ be the unit tangent bundle of $V$.
The geodesic flow $g_t: M \to M$ is Anosov (and thus hyperbolic).
Let $\widehat M$ be the frame bundle over $V$ which is a $SO(d-1)$-bundle over $M$, and let $f_t: \widehat M \to \widehat M$ be the frame flow.  
It is {\color{black} known} that the Brin group $H \subset SO(d-1)$ is  dense in the following cases:
\begin{itemize}
{\color{black}
\item $d$ is odd and $d \neq 7$ \cite{Bri75b} \cite{Bri80}; 
\item $d$ is even or $d=7$, and $V$ has  negative sectional  curvatures satisfying a suitable pinching condition,  \cite{Bri75b} \cite{Bri80}, \cite{Cek21}.
}
\end{itemize}
{\color{black} When $d\geq 4$ then $SO(d-1)$ is a semi-simple Lie group.}
Thus, by Theorems \ref{thm1} and \ref{thm2}, in {\color{black} both  of the cases above, provided  $d \geq 4$},  the frame flow $f_t$ is rapidly mixing (i.e., \eqref{1.0.1} holds)  and enjoys superpolynomial equidistribution as in \eqref{1.0.2}.
\end{Example}

Even stronger  results exist in the  very special case  of  frame flows for constant negative curvature manifolds 
 with respect to the normalized volume \cite{SarWak99}.  
 {\color{black} In both  of the cases above (when $d\geq 3$) rapid mixing with 
 respect to the normalized volume has been established by Lefeuvre and Ceki\'c \cite{CF} used a very different method.}
 As a statistical property, the mixing rate of frame flows not only has its own importance but can also be used to solve problems in other fields, such as the surface subgroup conjecture of Waldhausen proved in \cite{Kah12}.

The proof of Theorems \ref{thm1} and \ref{thm2} requires a non-trivial extension of  a   classical  approach involving  a symbolic model for both 
the flows $g_t$ and $f_t$ and the use of transfer operators.  
Although the technical details are challenging, the strategy is fairly straightforward. 
We briefly summarize the main steps in the approach:

\begin{enumerate}
\item
Using symbolic dynamics, it is easy to see that one can replace the correlation function $\rho(t)$  by  that  for an associated symbolic model for $f_t$ (in \S \ref{subsec 6.5});   

\item Then, we decompose the correlation function $\rho(t)=\sum_{\pi}\rho_{\pi}(t)$ in terms of irreducible representations $\pi$ (in \S \ref{subsec 7.1}); We replace  $\rho_{\pi}(t)$ by a function $\chi_{\pi}(t)$ whose Laplace transform $\widehat \chi_{\pi} (s)$ is easier  to analyze, but which leads to the same conclusion (see \S \ref{subsec 7.1}).
 
\item
The Laplace transform of each term associated to an irreducible representation can be expressed in terms of  families of transfer operators (adapting a traditional approach for suspension flows) (in \S \ref{subsec 7.2});

\item
The assumption that the Brin group $H \subset   G$  is dense 
and $G$ is semi-simple implies that a technical Diophantine property for the flow $f_t$ (which is in turn implies that  a corresponding Diophantine property for the the symbolic flow) (see \S \ref{subsec 6.6});

\item 
The symbolic Diophantine property implies estimates on families of transfer operators (in \S \ref{subsec 7.3}). Thus, the Laplace transform can be analytically extended to a specific region in the left half-plane. This implies a decay rate of the correlation function, which depends on the shape of the region (see \S \ref{subsec 7.4}).
\end{enumerate}

A key idea that makes this analysis work is the particular choice of norm used on each of the components associated to the irreducible representations of $G$.

\S \ref{sec 2} is about preliminaries. In \S \ref{subsec 2.1} we present some basic, but useful, results in the representation theory of compact groups, and in \S \ref{subsec 2.2} we introduce the concept of Diophantine subsets that will be used in the proof of Theorems \ref{thm1} and \ref{thm2}. In \S\S\S \ref{subsec 2.3}, \ref{subsec 2.4} and \ref{subsec 2.5}, we present some general constructions  in ergodic theory and dynamical systems that will serve us well in later sections.  

In \S \ref{sec 3} we describe some of the useful  properties of hyperbolic flows, and in \S \ref{sec 4} we briefly discuss the basic properties of compact group extensions of hyperbolic flows.

In \S \ref{sec 5}, we present the stronger version of our results where Theorem \ref{thm1} follows from Theorem \ref{Rapid mixing for compact group extension of hyperbolic flow} and Theorem \ref{thm2} follows from Theorem \ref{Superpolynomial equidistribution for compact group extension of hyperbolic flow}.

In \S \ref{sec 6}, we introduce the appropriate symbolic models used in the subsequent proofs. In particular, we can deduce rapid mixing of the extension flow (Theorem \ref{Rapid mixing for compact group extension of hyperbolic flow}) to that of a symbolic flow (Theorem \ref{Rapid mixing for symbolic flow of compact group extension}). 

In \S \ref{sec 7}, we present the proof of Theorem \ref{Rapid mixing for symbolic flow of compact group extension}. Finally, in  \S \ref{sec 8} we prove Theorem \ref{Superpolynomial equidistribution for compact group extension of hyperbolic flow}.

\bigskip
While this work was being completed we learned from Thibault Lefeuvre that he and 
Mihajlo Ceki\'c have  independently  established related  results in the context of  compact group extensions 
of volume preserving  Anosov flows \cite{CF}.   Their approach is completely different to ours, and we are grateful to them for both  sharing their elegant notes and for clarifying for us the issues  in the case of frame flows when $d=3$.

\section{Preliminaries}\label{sec 2}

\subsection{Representation theory of compact Lie groups}\label{subsec 2.1}

Before embarking on the dynamical aspects, we will consider some basic 
results  from the representation theory  of compact Lie groups. Our primary references are  \cite{App14} and \cite{Sug71}. 

Let $G$ be a compact connected Lie group. A representation of $G$ on $H_{\pi}$ is a strongly continuous homomorphism $\pi$ from $G$ to the group of bounded invertible operators on a normed linear space $H_{\pi}$. The dimension of $H_{\pi}$ is referred to as the dimension of $\pi$, denoted by $\dim_{\pi}$. A representation $\pi$ is termed unitary if $H_{\pi}$ is a complex separable Hilbert space, and the image of $\pi$ is contained in $U(H_{\pi})$, where $U(H_{\pi})$ denotes the set of unitary operators on $H_{\pi}$. Throughout, when we mention a representation, it implies a unitary representation. A special case is the left regular representation defined by
$$L:G\to U(L^{2}(G)),\quad L(g)f(g^{\prime})=f(g^{-1}g^{\prime})$$
where the inner product on $L^{2}(G)$ is with respect to the Haar measure. A representation $\pi$ of $G$ on $H_{\pi}$ is termed irreducible if there is no non-trivial invariant subspace of $\pi$. Two irreducible representations, $\pi_{1}$ and $\pi_{2}$ of $G$ on $H_{1}$ and $H_{2}$ respectively, are considered equivalent if there exists a unitary operator $T: H_{1} \to H_{2}$ such that $T\pi_{1}(g) = \pi_{2}(g)T$ for any $g \in G$. The set of all equivalent irreducible representations of $G$ is denoted by $\widehat{G}$ and is commonly referred to as the unitary dual of $G$. The notation $\pi \neq 1$ signifies that $\pi$ is not the trivial representation. As $G$ is compact, every irreducible representation is finite-dimensional, and $\widehat{G}$ is countable. Furthermore, for each element $\pi \in \widehat{G}$, $H_{\pi}$ is an invariant subspace of the left regular representation $L$, and $\pi$ is the restriction of the left regular representation $L$ to $H_{\pi}$, namely $\pi = L|_{H_{\pi}}$.

Fix an inner product on the Lie algebra $\mathfrak{g}$, and let $|\cdot|$ be the associated norm. Let $\lambda_{\pi}$ be the highest weight of $\pi$, which is a purely imaginary-valued linear form on $\mathfrak{g}$ and can thus be considered an element in $\mathfrak{g}$. It is worth mentioning that $\inf_{\pi\not=1}|\lambda_{\pi}|>0$, and we will use this fact to ensure the quantity $b_{\pi}$ considered in \S \ref{sec 7} is large enough for any $b\in\mathbb{R}$ and any $1\not=\pi\in\widehat{G}$. Given $\pi \in \widehat{G}$, since $\dim_{\pi} < \infty$, we can regard $\pi$ as a mapping from $G$ to the unitary matrices $U(\dim_{\pi})$. For a matrix $A \in GL(n)$, we use $\text{Tr}(A)$ to denote its trace. The Hilbert-Schmidt norm of $A$ is denoted by $|A|_{HS}$, defined as $|A|_{HS} = \text{Tr}(AA^{*})^{1/2}$. We use $|A|_{2}$ to denote the $2$-norm of $A$, so $|A|_{2} \leq |A|_{HS}$. For the representation $\pi$, for future convenience,  instead of using $||\cdot||_{L^{2}}$ we use $||\cdot||$ to represent the $L^{2}$ norm on $H_{\pi}$. We have $||\pi(g)||= |\pi(g)|_{2}$, where $||\pi(g)||$ is the operator norm of $\pi(g)$ acting on $H_{\pi}$. In particular, we have $||\pi(g)||\leq |\pi(g)|_{HS}$.

Let $\dim_{G}$ be the dimension of $G$, and let $r_{G} \in \mathbb{N}$ be the rank of $G$. Denote by $m_{G} := \frac{\dim_{G} - r_{G}}{2} \in \mathbb{N}$. We begin with a corollary of the famous Weyl’s dimensional formula, also found in \cite[Corollary 2.5.2]{App14}.

\begin{Lemma}\label{Lemma 2.7.1}
There exists $C_{3}\ge1$ such that for any $\pi$ we have $\dim_{\pi}\le C_{3}|\lambda_{\pi}|^{m_{G}}$.
\end{Lemma}

For a $\pi\in\widehat{G}$, we define its differential or derived representation by
$$d\pi(X)=\dfrac{d}{dt}\pi(\exp(tX))\bigg|_{t=0},\quad X\in\mathfrak{g}.$$
Clearly, $d\pi(X)$ is an operator on $H_{\pi}$. We then define $|d\pi|_{\infty}:=\sup_{|X|=1}||d\pi(X)||$. By definition and the Differential Mean Value theorem, we obtain the following result.

\begin{Lemma}\label{Lemma 2.7.2}
For any $g,g^{\prime}\in G$, we have $||\pi(g)-\pi(g^{\prime})||\le|d\pi|_{\infty}d_{G}(g,g^{\prime})$.
\end{Lemma}

We can bound $|d\pi(X)|_{HS}$ by the norm of the highest weight $\lambda_{\pi}$ of $\pi$ as follows (see also \cite[Theorem 3.4.1]{App14}).

\begin{Lemma}\label{Lemma 2.7.3}
For any $\pi$ and any $X\in\mathfrak{g}$, we have $|d\pi(X)|_{HS}\le C_{3}|\lambda_{\pi}|^{1+m_{G}/2}|X|$.
\end{Lemma}

Now, by Lemmas \ref{Lemma 2.7.2}, \ref{Lemma 2.7.3} and the inequality $||d\pi(X)||\le |d\pi(X)|_{HS}$, we have the following.

\begin{Corollary}\label{Corollary 2.7.4}
For any $1\not=\pi\in\widehat{G}$, any $g$ and $g^{\prime}\in G$, we have $||\pi(g)-\pi(g^{\prime})||\le C_{3}|\lambda_{\pi}|^{1+m_{G}/2} d_{G}(g,g^{\prime})$.
\end{Corollary}

We will utilize the following estimate multiple times, whose proof can be found in \cite[Theorem 3.2.1]{App14}. Recall that $r_{G}$ is the rank of $G$.

\begin{Lemma}\label{Lemma 2.7.5}
For any $x>r_{G}$, we have $\sum_{\pi\not=1}|\lambda_{\pi}|^{-x}$ converges.
\end{Lemma}

We now shift our focus to the decomposition of $L^{2}(G)$, specifically the Peter-Weyl theorem. It is well-known that the Hilbert spaces $H_{\pi}\subset L^{2}(G)$ of $\pi\in\widehat{G}$ are pairwise orthogonal. Additionally, $\{H_{\pi}:\pi\in\widehat{G}\}$ forms the decomposition of $L^{2}(G)$ into invariant subspaces concerning the left regular representation $L$. Given $F\in L^{2}(G)$ and $\pi\in\widehat{G}$, we define the Fourier coefficient as
$$\widehat{F}(\pi):=\int_{G}\pi(g^{-1})F(g)dg$$
which is a matrix-valued integral. The function $\text{Tr}(\widehat{F}(\pi)\pi)$ is $C^{\infty}$ and belongs to $H_{\pi}$. The projection of $F$ onto each $H_{\pi}$ is equal to $\dim_{\pi}\text{Tr}(\widehat{F}(\pi)\pi)$. Thus, we have the following \cite[Theorem 2.3.1]{App14}.

\begin{Lemma}\label{Lemma 2.7.6}
For any $F\in L^{2}(G)$, we can write $F=\sum_{\pi\in\widehat{G}}F_{\pi}$ where $F_{\pi}=\dim_{\pi}\text{Tr}(\widehat{F}(\pi)\pi)$ and the convergence is in the $L^{2}$ sense.
\end{Lemma}

We are interested in highly regular $F$. In this case, we will observe that the convergence in Lemma \ref{Lemma 2.7.6} is uniform and absolute. To achieve this, we require the following decay of Fourier coefficients \cite[Equation (3.3.14)]{App14}.

\begin{Lemma}\label{Lemma 2.7.7}
For any $F\in C^{2n}(G)$ and any $\pi\not=1$, we have $|\widehat{F}(\pi)|_{HS}\le |\widehat{\Delta^{n}F}(\pi)|_{HS}|\lambda_{\pi}|^{-2n}$, where $\Delta$ is the Laplace-Beltrami operator.
\end{Lemma}

By the Parseval equality of $\Delta^{n}F$, an elementary argument gives $|\widehat{\Delta^{n}F}(\pi)|_{HS}\le \dim_{G}||F||_{C^{2n}}$. Now, by Lemmas \ref{Lemma 2.7.1} and \ref{Lemma 2.7.7}, for any $g\in G$, we can bound
$$
\begin{aligned}
|F_{\pi}(g)|\le\dim_{\pi}|\widehat{F}(\pi)|_{HS}|\pi(g)|_{HS}=\dim_{\pi}^{3/2}|\widehat{F}(\pi)|_{HS}\le C_{3}^{3/2}|\lambda_{\pi}|^{3m_{G}/2-2n}\dim_{G}||F||_{C^{2n}}.
\end{aligned}
$$
In particular, if we choose $2n>\dim_{G}$, then $2n-3m_{G}/2>r_{G}$ and thus the convergence in Lemma \ref{Lemma 2.7.6} is uniform and absolute by Lemma \ref{Lemma 2.7.5}. By the above argument, we have the following.

\begin{Lemma}\label{Lemma 2.7.8}
For any integer $n\ge1$, any $F\in C^{2n}(G)$ and any $\pi\not=1$, we have $|F_{\pi}|_{\infty}\le C_{3}||F||_{C^{2n}}|\lambda_{\pi}|^{-(2n-3m_{G}/2)}$. In particular, if $2n>\dim_{G}$, we have $\sum_{\pi}F_{\pi}$ converges to $F$ uniformly and absolutely.
\end{Lemma}

Recall that class functions $F$ are functions in $L^{2}(G)$ that satisfy $F(hgh^{-1})=F(g)$. For any $\pi\in\widehat{G}$, we can associate $\pi$ with a $C^{\infty}$ class function $\xi_{\pi}$ called the character of $\pi$, defined by $\xi_{\pi}(g)=\text{Tr}(\pi(g))$. The decomposition of a class function is more concise.

\begin{Lemma}\label{Lemma 2.7.9}
For any class function $F\in L^{2}(G)$ and any $\pi$, we have $F_{\pi}=f_{\pi}\xi_{\pi}$, i.e., $F=\sum_{\pi}f_{\pi}\xi_{\pi}$ where $f_{\pi}=\int_{G}F(g)\xi_{\pi}(g^{-1})dg$. Furthermore, for any integer $n\ge1$, any class function $F\in C^{2n}(G)$ and any $\pi\not=1$, we have $|f_{\pi}|\le C_{3}||F||_{C^{2n}}|\lambda_{\pi}|^{-(2n-3m_{G}/2)}$.
\end{Lemma}
\begin{proof}
For the first part, we refer the reader to \cite[Section 2.4]{App14}. We are interested in the decay of the coefficient $f_{\pi}$. It is well-known that ${\xi_{\pi}:\pi\in\widehat{G}}$ forms an orthogonal basis for class functions. In particular, $||\xi_{\pi}||_{L^{2}}=1$ for any $\pi$. Thus, by Lemma \ref{Lemma 2.7.8}, for any class function $F\in C^{2n}(G)$ and any $\pi\not=1$, we have $|f_{\pi}|=||f_{\pi}\xi_{\pi}||_{L^{2}}\le|F_{\pi}|_{\infty}\le C_{1}||F||_{C^{2n}}|\lambda_{\pi}|^{-(2n-3m_{G}/2)}$, which completes the proof.
\end{proof}

We will need the following estimate of $\xi_{\pi}$ in the proof of the superpolynomial equidistribution theorem.

\begin{Lemma}\label{Lemma 2.7.10}
For any $\pi\in\widehat{G}$, we have $|\xi_{\pi}|_{\infty}\le C_{3}|\lambda_{\pi}|^{m_{G}}$.
\end{Lemma}
\begin{proof}
Denote by $E_{n}$ the identity $n\times n$ matrix. By definition, 
$$|\xi_{\pi}(g)|=|\text{Tr}(\pi(g))|\le|E_{\dim_{\pi}}|_{HS}|\pi(g)|_{HS}=\dim_{\pi}$$
which completes the proof by Lemma \ref{Lemma 2.7.1}.
\end{proof}

\subsection{Diophantine subsets}\label{subsec 2.2}

A dynamical version of the following Diophantine condition 
for a subset $\Gamma \subset G$ will be the key to the proof of Theorems \ref{thm1} and \ref{thm2}.

\begin{Definition}\label{Definition of Diophantine condition on Brin group}
We say a subset $\Gamma\subset G$ is Diophantine if there exist $C>0$  and $\delta>0$ such that for any $1\not=\pi\in\widehat{G}$ and any $h\in H_{\pi}$ there exists $g\in\Gamma$ such that 
$$||h-\pi(g)h||_{L^{2}}\ge ||h||_{L^{2}}\dfrac{\delta}{|\lambda_{\pi}|^{C}}.$$
\end{Definition}

Note that the above Diophantine property of subsets is invariant under a conjugacy, namely if $\Gamma^{\prime}$ is another subgroup of $G$ and $\Gamma^{\prime}$ is conjugate to $\Gamma$ then $\Gamma^{\prime}$ is also Diophantine. It was observed in \cite{Dol02} that  if $\Gamma$ is Diophantine then any group containing it will be dense, and it also shown that generic two-point set is Diophantine. We will relate this to the dynamics of compact group extensions of hyperbolic flows in a later section. 

\begin{Remark}\label{Remark 2}
	We should  comment on how this general Diophantine condition relates to the classical Diophantine approximation problem.
	If we assume $G=\mathbb{T}^{1}$, then each $\pi$ associated a $m\in\mathbb{Z}$, and $\pi:G\to U(\mathbb{C})$, $\pi(e^{i\theta})=e^{im\theta}$. Each $H_{\pi}$ is one-dimensional, and its basis is $h_{m}:S^{1}\to\mathbb{C}$, $h_{m}(z)=z^{m}$. Assume $\Gamma=\{e^{im\theta}\}$ is a single point set. Then the Diophantine condition in Definition \ref{Definition of Diophantine condition on Brin group} means that $|1-e^{i2\pi m\theta}|\ge \delta/|m|^{C}$ which says $\theta$ is a bad Diophantine number since in this case we have $|m\theta-q|\ge\delta/|m|^{C}$. 
\end{Remark}

The following result gives us a precise example of Diophantine subsets, see also \cite[Theorem A.3]{Dol02}.

\begin{Lemma}\label{Lemma 2.2.2}
Assume $G$ is semi-simple. Then any $\varepsilon$-net of $G$ with small sufficiently $\varepsilon>0$ is Diophantine.
\end{Lemma}




\subsection{Pressure and equilibrium states}\label{subsec 2.3}

In the remainder  of this section, we collect together some background material from ergodic theory. We formulate this in a general context. In particular, in this subsection, we introduce pressure and Gibbs measure
(or equilibrium states). A standard reference for this material is \cite{Wal82}.
 
We assume $X$ is a compact metric space, and $T:X\to X$ is continuous. Let $\mathscr{B}$ be the Borel $\sigma$-algebra.
 Denote by $M(X)$ the collection of probability measures with respect to $\mathscr{B}$, and denote by $M(X,T)$ the elements of $M(X)$ which are $T$-invariant. Thus, $T$ is a measure-preserving map on the probability space $(X, \mathscr{B}, \mu)$ for $\mu\in M(X,T)$. Let $h_{\mu}(T)\in[0,\infty]$ denote  the entropy of the measure $\mu$. 

\begin{Definition}
For a continuous function $\Phi\in C(X,\mathbb{R})$, its 
\emph{pressure} $P(\Phi)$ is the supremum of $h_{\mu}(T)+\int_{X}\Phi d\mu$ taken over all $\mu\in M(X,T)$.\end{Definition}

 In the case where $\Phi\equiv0$, we call $h_{\text{top}}(T):=P(0)$
the topological entropy of $T$.

\begin{Definition}
If $\mu\in M(X,T)$ such that $h_{\mu}(T)+\int_{X}\Phi d\mu=P(\Phi)$, then we call $\mu$
an 
\emph{equilibrium state} 
or \emph{Gibbs measure}
for $\Phi$. In the case where $\Phi\equiv0$, we call $\mu$ a measure of maximal entropy of $T$.
\end{Definition}

For a continuous semi-flow\footnote{This means $\{\phi_{t}\}_{t\ge0}$ is a collection of continuous maps, and $\phi_{t+s}=\phi_{t}\circ\phi_{s}$ for $s,t \geq 0$.} $\phi_{t}:X\to X$ on a compact metric space $X$, the definitions of pressures and equilibrium states are as follows.

\begin{Definition}
For a continuous function $\Phi\in C(X,\mathbb{R})$, we define its 
\emph{pressure} $P(\Phi)$ with respect to $\phi_{t}$ to be the pressure of $\Phi$ with respect to the time-1 map $\phi_{1}$. We call $\mu\in M(X,T)$ an equilibrium state of $\phi_{t}$ if it is an 
\emph{equilibrium state} for  $\phi_{1}$. In the case where $\Phi\equiv0$, we call $h_{\text{top}}(\phi_{1})=P(0)$ the topological entropy of $\phi_{t}$ and call $\mu$ a measure of maximal entropy for $\phi_{t}$.
\end{Definition}

There are two particularly   important equilibrium state measures. The first is the measure of maximal entropy (MME)  associated with constant potentials. This measure holds a strong connection with the system's periodic orbit distribution. The second is the Sinai-Ruelle-Bowen (SRB) measure, representing the Gibbs measure of the geometric potential when  the system is an attractor. This measure 
is closely  related to the  statistical properties of the system.

\subsection{Skew products}\label{subsec 2.4}

 Let $T:X\to X$ be a continuous map on a compact metric space $X$, and let $G$ be a compact Lie group. Given a continuous function $\Theta:X\to G$, we can define the skew product $\widehat{T}$ as follows:
$$
\widehat{T}:\widehat{X}\to\widehat{X},\quad \widehat{T}(x,g)=(T(x),\Theta(x)g)
$$
where $\widehat{X}=X\times G$. Obviously, $\widehat{T}:\widehat{X}\to\widehat{X}$ is still a continuous map on the compact metric space $\widehat{X}$ with respect to the product metric. Let Haar$_{G}$ be the normalized Haar measure on $G$. If $\mu\in M(X,T)$, then it is easy to see that $\widehat{\mu}:=\mu\times\text{Haar}_{G}\in M(\widehat{X},\widehat{T})$.

\subsection{Suspension semi-flows}\label{subsec 2.5}

Let $T:X\to X$ be a continuous map on a compact metric space $X$, and let $r:X\to\mathbb{R}$ be a continuous function with $\inf_{x\in X}r(x)>0$. We define the suspension space as $X_{r}:=\{(x,u):0\le u\le r(x)\}/\sim$, where $(x,r(x))\sim(T(x),0)$. We also define the suspension semi-flow $\phi_{t}$ by 
$$
\phi_{t}:X_{r}\to X_{r},\quad \phi_{t}(x,u)=(x,u+t),\quad t\ge0,
$$
with respect to $\sim$ on $X_{r}$. If $T$ is a bijection, then with the same definition, $\phi_{t}$ becomes a flow. In the above definition, we call $\phi_{t}:X_{r}\to X_{r}$ a suspension semi-flow (or flow) of $r$ and $T$. Obviously, $\phi_{t}:X_{r}\to X_{r}$ is a continuous semi-flow on the compact metric space $\widehat{X}$ with respect to the product metric. Each $\mu\in M(X,T)$ induces a natural probability measure $\mu\times\text{Leb}$ on $\widehat{X}_{r}$ by
$$
\int_{X_{r}}Fd(\mu\times\text{Leb})=\dfrac{1}{\int_{X}rd\mu}\int_{X}\int_{0}^{r(x)}F(x,u)dud\mu(x),\quad F\in C(X_{r}).
$$ 
It can be shown that $\mu\times\text{Leb}$ is $\phi_{t}$-invariant \cite{Par90}.

\section{Hyperbolic flows and their properties}\label{sec 3}

In this section, we recall the definition of a hyperbolic flow and some of its properties. Appropriate references are \cite{Par90} and \cite{Fis19}. 

\subsection{Hyperbolic flows, Gibbs measures and u-s paths}\label{subset 3.1}

 Let $M$ be a smooth compact Riemannian manifold, and let $TM$ be its tangent space. Let $g_{t}:M\to M$ be a $C^{1}$ flow, and let $\Lambda\subset M$ be a compact $g_{t}$-invariant set. 

\begin{Definition}\label{Definition of hyperbolic flow}
We say $g_{t}:\Lambda\to \Lambda$ is a hyperbolic flow if 
\begin{enumerate}
\item There is a $Dg_{t}$-invariant splitting $T_{\Lambda}M=E^{s}\oplus E^{c}\oplus E^{u}$ over $\Lambda$, and exist constant $C>0$ and $\delta>0$ such that $E^{c}$ is the line bundle tangent to $g_{t}$, and $||Dg_{t}|_{E^{s}}||,\  ||Dg_{-t}|_{E^{u}}||\le Ce^{-\delta t}$ for all $t\ge0$.
\item $\Lambda$ consists of more than a single closed orbit, and closed orbits in $\Lambda$ are dense.
\item $g_{t}$ is transitive on $\Lambda$,\footnote{It means for any open sets $U,V$  and sufficiently  large $t>0$  we have $U\cap g_{t}V\not=\emptyset$.} and there is an open set $\mathcal{U}\supset\Lambda$ such that $\Lambda=\bigcap_{t=-\infty}^{\infty}g_{t}\mathcal{U}$.
\end{enumerate}
\end{Definition}

In the case where $\Lambda=M$, $g_t$ is referred to as an Anosov flow. 

The subbundles $E^{s}$ and $E^{u}$ are called the 
\emph{stable and unstable subbundles}, respectively, and they are always integrable. The integral manifolds of $E^{s}$ and $E^{u}$ are called the \emph{stable and unstable manifolds}, denoted by $W^{s}$ and $W^{u}$, respectively. They are $g_{t}$-invariant foliations, and can be written as  
$$W^{s}(x)=\{y\in M:d(g_{t}(y),g_{t}(x))\to0\ as\ t\to\infty\}$$
and 
$$W^{u}(x)=\{y\in M:d(g_{-t}(y),g_{-t}(x))\to0\ as\ t\to\infty\}.$$
The local versions of these objects are also useful. Given $\varepsilon>0$ small, the local stable manifold of $x$ of size $\varepsilon$ is defined by
$$W^{s}_{\varepsilon}(x)=\{y\in W^{s}(x):\sup_{t\ge0}d(g_{t}(y),g_{t}(x))\le\varepsilon\}.$$
Similarly, the local unstable manifold of $x$ of size $\varepsilon$ is defined by
$$W^{u}_{\varepsilon}(x)=\{y\in W^{u}(x):\sup_{t\ge0}d(g_{-t}(y),g_{-t}(x))\le\varepsilon\}.$$
When we don't need to emphasize their size, we will write them as $W^{s}_{loc}$ and $W^{u}_{loc}$. They are exponentially convergent under the action of $g_{t}$ or $g_{-t}$. Usually, we assume that $g_{t}$ is $C^{2}$ and topologically mixing which means for any two non-empty open sets $U,V$ of $\Lambda$ and large enough $t>0$ we have $U\cap g_{t}V\not=\emptyset$. It is well-known that for any H\"older potential $\Phi$ on $\Lambda$, i.e. it is a H\"older continuous real-valued function, there exists a unique equilibrium state $\mu_{\Phi}$ on $\Lambda$ which is called the 
\emph{Gibbs measure} of $\Phi$ \cite{Bow75}. We will use the following concept of u-s paths to define the Brin group in \S \ref{subsec 4.2}.

\begin{Definition}\label{Definition of closed chain and u-s path}
A u-s path of $g_{t}$ is a set of finitely many points $\{x_{i}\}_{i=0}^{p}\subset\Lambda$ so that $x_{i+1}\in W^{s}(x_{i})$ or $W^{u}(x_{i})$ for each $0\le i\le p-1$. A closed chain $W$ at $x$ is a u-s path $\{x_{i}\}_{i=0}^{p}$ so that $x_{0}=x_{p}=x$.
\end{Definition}

It is worth mentioning that one can also use u-s paths to describe the non-integrability of an Anosov flow \cite{Pla72} \cite{Fis19}, and we will do this in \S \ref{subsec 3.4}. 

    \begin{figure}[h!]
          \centerline{
    \begin{tikzpicture}[thick,scale=0.60, every node/.style={scale=0.85}]
\draw[<->, black] (0,0)-- (1,1);
\draw[-, black] (1,1)-- (2,2);
\draw[-, black] (-1,-1)-- (0,0);
\draw[<-, black] (1,3)-- (2,2);
\draw[-, black] (1,3)-- (0,4);
\draw[->, black] (-1,5)-- (0,4);
\draw[<->, black] (0,6)-- (1,7);
\draw[-, black] (1,7)-- (2,8);
\draw[-, black] (-1,5)-- (0,6);
\draw[<-, black] (1,9)-- (2,8);
\draw[-, black] (1,9)-- (0,10);
\draw[->, black] (-1,11)-- (0,10);
\draw[<->, black] (-3,9)-- (-2,10);
\draw[-, black] (-2,10)-- (-1,11);
\draw[-, black] (-4,8)-- (-3,9);
\draw[<-, black] (-5,9)-- (-4,8);
\draw[-, black] (-6,10)-- (-5,9);
\draw[->, black] (-7,11)-- (-6,10);
\draw[<->, black] (-9,9)-- (-8,10);
\draw[-, black] (-8,10)-- (-7,11);
\draw[-, black] (-9,9)-- (-10,8);
\draw[->, black] (-10,8)-- (-7,5);
\draw[-, black] (1,3)-- (0,4);
\draw[->, black] (-1,-1)-- (-4,2);
\draw[-, black] (-7,5)-- (-4,2);
 \node at (-1, -1.6) {$x_0 = x_p$};
 \node at (1.5, 0) {$W^u(x_0)$};
  \node at (2.6, 2) {$x_1$};
   \node at (1.5, 4) {$W^s(x_1)$};
     \node at (2.7, 8) {$x_{i-1}$};
          \node at (-0.9, 11.5) {$x_{i}$};
           \node at (1.7, 10) {$W^s(x_{i-1})$};
                      \node at (-3.3, 10) {$W^u(x_{i})$};
                      \node at (-3.8, 7.7) {$x_{i+1}$};
                        \node at (-6.6, 8.9) {$W^s(x_{i+1})$};
                                           \node at (-10.6, 7.7) {$x_{p-1}$};
                                           \node at (-8, 4) {$W^s(x_{p-1})$};
\end{tikzpicture}
}
\caption{A closed chain $W=\{x_{i}\}_{i=0}^{p}$ on $M$ }
\end{figure}
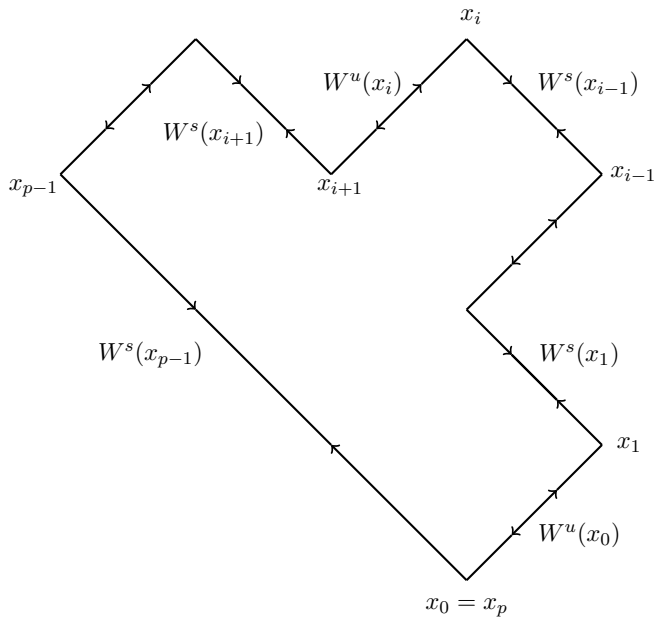

\subsection{Rate of mixing for hyperbolic flows} \label{subsec 3.2}

Recall that for a $g_{t}$-invariant measure $\mu$ on $\Lambda$ and two test functions $E,F\in L^{2}(\mu)$, their correlation function is defined by 
$$\rho_{E,F}(t)=\int_{\Lambda}E\circ g_{t}\cdot\overline{F}d\mu-\int_{\Lambda}Ed\mu\int_{\Lambda}\overline{F}d\mu
.$$
It is well-known that $g_{t}$ is always mixing with respect to Gibbs measures, namely $\rho_{E,F}(t)\to0$ as $t\to\infty$ for any $E,F\in L^{2}(\mu)$ \cite{Bow75}. A central question in the theory of dynamical systems is to characterize the rate of mixing of $g_{t}$ with respect to Gibbs measures, which is also called the decay of correlations. 

\begin{Definition}\label{Definition of rapid mixing}
Let $\mu_{\Phi}$ be the Gibbs measure of a H\"older potential $\Phi$ on $\Lambda$.
We say $g_{t}$ is rapidly mixing with respect to $\mu_{\Phi}$ if for each $n\in\mathbb{N}^{+}$ there exist $C_4>0$ and $k\in\mathbb{N}^{+}$ such that for any $t>0$ and any $E,F\in C^{k}(M)$, we have $|\rho_{E,F}(t)|\le C_4||E||_{C^{k}}||F||_{C^{k}}t^{-n}$.
\end{Definition}

It's essential to recognize that not every hyperbolic flow is mixing with respect to Gibbs measures. For example, constant height suspensions of hyperbolic diffeomorphisms. Although a hyperbolic flow under this condition shows mixing with respect to Gibbs measures \cite{Bow75}, there are instances where mixing occurs at an arbitrarily slow rate \cite{Pol85}. We are interested in hyperbolic flows $g_{t}$ that are at least rapidly mixing. In particular, this leads us to consider two types of hyperbolic flows. One type is hyperbolic flows that satisfying a Diophantine condition on closed orbits, and we will introduce them in the next subsection. The other type is jointly non-integrable Anosov flows which will be introduced in \S \ref{subsec 3.4}.

\subsection{A Diophantine condition on closed orbits}\label{subsec 3.3}

We consider compact group extensions of a class of hyperbolic flows which are  rapidly mixing. In \cite{Dol98b}, it is proven that if one can choose two closed orbits with a specific period ratio known as a bad Diophantine number, then the hyperbolic flow is rapidly mixing with respect to Gibbs measures. Using the method in \cite{Dol98b}, it is proven in \cite{Pol01} that one can also choose three closed orbits which satisfy the following similar Diophantine condition to achieve the same effect and the proof is simpler. 

\begin{Definition}\label{Definition of Diophantine for hyperbolic flow}
We say a hyperbolic flow $g_{t}$ is Diophantine if we can choose three closed orbits with periods $l_{1},l_{2},l_{3}$ such that $\alpha:=\frac{l_{1}-l_{2}}{l_{2}-l_{3}}$ is a bad Diophantine number, which means there exists $C_5>0$ and $\delta>0$ such that $|q\alpha-p|\ge \delta|q|^{-C_5}$ for any $q\in\mathbb{Z}$ and any $p\in\mathbb{Z}$.
\end{Definition}

Note that we may assume the Diophantine exponent $C_{5}>1$. Then, it is known that the set of bad Diophantine numbers with exponent $C_{5}>1$ has full Lebesgue measure \cite{Bug05}.

\subsection{Jointly non-integrable Anosov flows}\label{subsec 3.4}

Another type of rapidly mixing hyperbolic flows are  jointly non-integrable Anosov flows. Let $g_{t}:M\to M$ be an Anosov flow. In general, the codimension one subbundle $E^{s}\oplus E^{u}$ is not integrable, and we say that $g_t$ is jointly non-integrable if this is the case. It is convenient to characterize joint non-integrability using u-s paths of $g_{t}$ as shown in \cite{Pla72} and \cite{Fis19}. For a point $x\in M$, consider the collection $\mathcal{W}_{\delta}(x)$ of all u-s paths $W=\{x_{i}\}_{i=0}^{4}$ with $x_{0}=x$,
$$x_{1}\in W^{u}_{\delta}(x_{0}),\quad x_{2}\in W^{s}_{\delta}(x_{1}),\quad x_{3}\in W^{u}_{\delta}(x_{2}),\quad x_{4}\in W^{s}_{\delta}(x_{3})\quad\text{and}\quad x_{4}\in g_{t_{W}}(x_{0})$$
for some $t_{W}\in\mathbb{R}$. Then, $g_{t}$ is jointly non-integrable if and only if there exist $x\in M$ and small $\delta>0$ such that $t_{W}\not=0$ for any u-s path $W\in \mathcal{W}_{\delta}(x)$.

Joint non-integrability always implies topological mixing, and they are equivalent if $g_{t}$ is codimension one \cite[Theorem 3.7]{Pla72}. Dolgopyat proved in \cite{Dol98b} that if $g_{t}$ is a jointly non-integrable Anosov flow, then it is rapidly mixing with respect to Gibbs measures. 

\section{Compact group extensions}\label{sec 4}

After introducing the necessary concepts about hyperbolic flows, we now shift our focus to the research subject in this article: compact group extensions of hyperbolic flows. These flows are no longer hyperbolic and represent one of the simplest examples of partially hyperbolic flows.

\subsection{Compact group extensions of hyperbolic flows}\label{subsec 4.1}

 Let $M$ be a compact manifold, and let $g_{t}:\Lambda\to\Lambda\subset M$ be a hyperbolic flow. Let $G$ be a compact connected Lie group, and let $\widehat{M}$ be a smooth $G$-bundle over $M$ with the projection $\pi:\widehat{M}\to M$. We consider a $C^{2}$ flow $f_{t}$ on $\widehat{M}$ and denote $\widehat{\Lambda}:=\pi^{-1}(\Lambda)$.

\begin{Definition}\label{Definition of compact group extension}
We say $f_{t}:\widehat{\Lambda}\to\widehat{\Lambda}$ is a $G$-extension of $g_{t}$ if 
\begin{enumerate}
\item For any $t\in\mathbb{R}$, we have $\pi\circ f_{t}=g_{t}\circ\pi$ on $\widehat{\Lambda}$.
\item For any $g\in G$ and any $t\in\mathbb{R}$, we have $f_{t}\circ R_{g}=R_{g}\circ f_{t}$ on $\widehat{\Lambda}$, where $R_{g}$ is the right action of $g$.
\end{enumerate}
\end{Definition}

It is known \cite{Bri75a} that $f_{t}$ has two foliations called the 
\emph{stable foliation} and \emph{unstable foliation} of $f_t$, denoted by $W^{s}_{f_{t}}$ and $W^{u}_{f_{t}}$, respectively, which correspond to the stable foliation $W^{s}$ and unstable foliation $W^{u}$ of $g_t$.  They are $f_{t}$-invariant and are also exponentially convergent under the action of $f_{t}$ or $f_{-t}$. For any point $\widehat{x}\in\pi^{-1}(x)$, we have $\pi(W^{j}_{f_{t}}(\widehat{x}))=W^{j}(x)$ and $W^{j}_{f_{t}}(R_{g}(\widehat{x}))=R_{g}W^{j}_{f_{t}}(\widehat{x})$ for $j=s,u$. This means $W^{j}_{f_{t}}(\widehat{x})$ is the graph of a map from $W^{j}(x)$ to its fibre for $j=s,u$. 

The Gibbs measure $\mu_{\Phi}$ on $\Lambda$ of a H\"older potential $\Phi$ naturally induces a probability measure $\widehat{\mu}_{\Phi}$ on $\widehat{\Lambda}$ which is the local product of $\mu_{\Phi}$ and the normalized Haar measure Haar$_{G}$ on $G$. By definition, we have $\widehat{\mu}_{\Phi}=\int_{\Lambda}\text{Haar}_{x}d\mu_{\Phi}(x)$, where $\text{Haar}_{x}$ is the push-forward of Haar$_{G}$ on $G$ to the fibre $\pi^{-1}(x)$. 

The following result is easily proved.

\begin{Lemma}\label{Invariance of local product measure}
$f_{t}$ preserves the measure $\widehat{\mu}_{\Phi}$.
\end{Lemma}


\subsection{Rapid mixing and superpolynomial equidistribution}\label{subsec 4.2}

We are interested in the rate of mixing of $f_{t}$ with respect to $\widehat{\mu}_{\Phi}$. By analogy with the case of hyperbolic flows, for $E,F\in L^{2}(\widehat{\mu}_{\Phi})$ their correlation function is defined as 
$$\rho_{E,F}(t)=\int_{\widehat{\Lambda}}E\circ f_{t}\cdot\overline{F}d\widehat{\mu}_{\Phi}-\int_{\widehat{\Lambda}}Ed\widehat{\mu}_{\Phi}\int_{\widehat{\Lambda}}\overline{F}d\widehat{\mu}_{\Phi}.$$
Then the statement of $f_{t}$ is  rapidly mixing with respect to $\widehat{\mu}_{\Phi}$ is an analogous manner to Definition
\ref{Definition of  rapid mixing}.

We can also look at error terms of the equidistribution theorem. Let $\tau$ represent a closed orbit of $g_{t}$, and denote its period as $\ell_{\tau}$. Via the extension flow $f_{t}$, each $\tau$ induces a conjugacy class $[\tau]$ in $G$ which is called the holonomy class of $\tau$. For any $T>0$, let $\pi(T)$ be the collection of prime closed orbits $\tau$ with $\ell_{\tau}\le T$.

\begin{Definition}\label{Definition of rapid equidistribution}
We say $f_{t}$ satisfies the superpolynomial equidistribution if for any $n\in\mathbb{N}^{+}$ there exist $k\in\mathbb{N}^{+}$ and $C_6>0$ such that for any class function $F\in C^{k}(G)$ and any $T>0$,
$$\bigg|\dfrac{1}{\#\pi(T)}\sum_{\tau\in\pi(T)}F([\tau])-\int_{G}Fd\text{\rm (Haar)}_{G}(g)\bigg|\le C_6||F||_{C^{k}}{T^{-n}}.$$
\end{Definition}

\subsection{The Brin group}\label{subsec 4.3}

In Theorems \ref{thm1} and \ref{thm2}, we gave assumptions on the Brin group of $f_{t}$ so that $f_{t}$ enjoys rapid  mixing with respect to $\widehat{\mu}_{\Phi}$. 
In this subsection, we provide a precise definition of the Brin group. We beginning with the following definitions of stable and unstable twists which are used to define the Brin group.

\begin{Definition}\label{Definition of stable twist}
For $x,y\in \Lambda$ with $y\in W^{s}(x)$, let $\varphi_{x}$ and $\varphi_{y}$ be the trivializations at $x$ and $y$ respectively. Consider the point $(x,e)$ in the fibre $\pi^{-1}(x)$. Then the point in the graph $W^{s}_{f_{t}}(x,e)$ that passing through the fibre $\pi^{-1}(y)$ is equal to $(y,\mathcal{T}^{s}(x,y))$. We call $\mathcal{T}^{s}(x,y)$ the stable twist of $x$ and $y$.
\end{Definition}

The unstable twist $\mathcal{T}^{u}(x,y)$ of $y\in W^{u}(x)$ can be defined in an analogous manner. Obviously, $\mathcal{T}^{s}(x,y)$ depends on the trivializations at $x$ and $y$. Since  $W^{s}_{f_{t}}(R_{g}(\widehat{x}))=R_{g}W^{s}_{f_{t}}(\widehat{x})$, for any point $(x,g)$ in $\pi^{-1}(x)$, we have the point in the graph $W^{s}_{f_{t}}(x,g)$ that passing through the fibre $\pi^{-1}(y)$ is equal to $(y,\mathcal{T}^{s}(x,y)g)$. We will express the stable twist $\mathcal{T}^{s}$ and unstable twist $\mathcal{T}^{u}$ precisely in the symbolic model constructed in \S \ref{subsec 6.6}.

Given a closed chain $W=\{x_{i}\}_{i=1}^{p}$ of $g_{t}$ at $x\in\Lambda$ and trivializations $\varphi_{x_{i}}$ at $x_{i}$, the twist induced by $W$ is defined as $g_{W}:=\mathcal{T}^{\delta_{x_{p-1},x_{p}}}(x_{p-1},x_{p})\cdots\mathcal{T}^{\delta_{x_{0},x_{1}}}(x_{0},x_{1})$ where $\delta_{x,y}=s$ if $y\in W^{s}(x)$ and $\delta_{x,y}=u$ if $y\in W^{u}(x)$. Although $\mathcal{T}^{s}$ and $\mathcal{T}^{u}$ depend on the trivializations, it is not the case for $g_{W}$.

\begin{Lemma}\label{Independence of closed chain for local coordiante}
For any closed chain $W=\{x_{i}\}_{i=1}^{p}$ of $g_{t}$ at $x$, we have $g_{W}$ is independent of the choice of the trivialization $\varphi_{x_{i}}$ at $x_{i}$ for $1\le i\le p-1$.
\end{Lemma}
\begin{proof}
Suppose, for each $1\le i\le p-1$, $\varphi^{\prime}_{x_{i}}$ is another trivialization at $x_{i}$. Let $g_{i}$ be the difference of the coordinate transformation from $\varphi_{x_{i}}$ to $\varphi^{\prime}_{x_{i}}$, i.e. $\varphi^{\prime}_{x_{i}}\varphi_{x_{i}}^{-1}(x_{i},e)=(x_{i},g_{i})$. Then by definition, the stable or unstable twist with respect to $\varphi^{\prime}_{x_{i}}$ is $g_{i+1}^{-1}\mathcal{T}^{j_{x_{i+1},x_{i}}}(x_{i},x_{i+1})g_{i}$. Note that $g_{0}=g_{p}=e$. Therefore the twist induced by $W$ with respect to  $\varphi^{\prime}_{x_{i}}$ is 
$$
\begin{aligned}
g_{p}^{-1}\mathcal{T}^{j_{x_{p-1},x_{p}}}(x_{p-1},x_{p})g_{p-1}\cdots g_{1}^{-1}\mathcal{T}^{j_{x_{0},x_{1}}}(x_{0},x_{1})g_{0}=\mathcal{T}^{j_{x_{p-1},x_{p}}}(x_{p-1},x_{p})\cdots\mathcal{T}^{j_{x_{0},x_{1}}}(x_{0},x_{1})
\end{aligned}
$$
which completes the proof.
\end{proof}

It is straightforward to verify that all twists induced by closed chains \(W\) at \(x\) form a subgroup of \(G\). In a series of papers by Brin (e.g., \cite{Bri75a}, \cite{Bri75b}, \cite{Bri82}, and \cite{Bri84}), this subgroup plays an irreplaceable role in proving the ergodicity of frame flows and more general partially hyperbolic systems. Therefore, the subgroup is commonly referred to as the Brin group.

\begin{Definition}\label{Definition of Brin group}
	For a point \(x \in \Lambda\) and a trivialization \(\varphi_{x}\) at \(x\), the Brin group \(H(x) \subset G\) at \(x\) is the set of all twists \(\{g_{W}\}\) of closed chains \(W\) at \(x\).
\end{Definition}

Although the definition of the Brin group involves trivializations, \(H(x)\) is independent of the choice of trivializations \(\varphi_{x}\) at \(x\). More precisely, changing the trivialization will result in a new Brin group, which is conjugate to the old one as follows.

\begin{Lemma}\label{Independence of Brin group for local coordinate}
Let $x\in\Lambda$, and let $H(x)$ and $H^{\prime}(x)$ be the Brin groups at $x$ with respect to the trivializations $\varphi_{x}$ and $\varphi^{\prime}_{x}$ respectively. Denote by $\tilde{g}$ the difference of the coordinate transformation from $\varphi_{x}$ to $\varphi^{\prime}_{x}$, namely $(x,\tilde{g})=\varphi^{\prime}_{x}\varphi_{x}^{-1}(x,e)$, then $H(x)$ and $H^{\prime}(x)$ are conjugate by $\tilde{g}$ which means $H(x^{\prime})=\tilde{g}^{-1}H(x)\tilde{g}$.
\end{Lemma}
\begin{proof}
This can be easily obtained from the proof of Lemma \ref{Independence of closed chain for local coordiante}.
\end{proof}

\section{Statements of  the main results}\label{sec 5}

In this section we present statements of stronger versions of Theorems \ref{thm1} and \ref{thm2}
and, because of their importance as examples, revisit frame flows as an application.

\subsection{Theorems under a Diophantine condition}\label{subsec 5.1}

Let $M$ be a compact manifold, and let $g_{t}:\Lambda\to\Lambda\subset M$ be a hyperbolic flow. Let $G$ be a compact connected Lie group, and let $\widehat{M}$ be a smooth $G$-bundle over $M$ with the projection $\pi:\widehat{M}\to M$. The dynamical system of interest is a $C^{\infty}$ $G$-extension, denoted as $f_{t}:\widehat{\Lambda}\to \widehat{\Lambda}$, of $g_{t}:\Lambda\to\Lambda$, where $\widehat{\Lambda}:=\pi^{-1}(\Lambda)$. Let $\Phi$ be a H\"older potential on $\Lambda$, and let $\mu_{\Phi}$ be its Gibbs measure. Then, $\mu_{\Phi}$ lifts to a $f_t$-invariant measure $\widehat{\mu}_{\Phi}$ on $\widehat{\Lambda}$ which is a local product  measure  combining $\mu_{\Phi}$ with the normalized Haar measure Haar$_{G}$ on $G$. By Lemma \ref{Invariance of local product measure}, the extension flow $f_{t}$ preserves $\widehat{\mu}_{\Phi}$. 

Note that the mixing rate of $f_{t}$ with respect to $\widehat{\mu}_{\Phi}$ is not superior to the mixing rate of $g_{t}$ with respect to $\mu_{\Phi}$. Therefore, the underlying hyperbolic flow $g_{t}$ we should consider must have a corresponding  mixing rate. Specifically, The hyperbolic flow $g_{t}$ considered is those defined in \cite{Dol98b} and \cite{Pol01}.
Throughout this article, we make the assumption that $g_{t}$ is either Diophantine (see Definition \ref{Definition of Diophantine for hyperbolic flow})  or a jointly non-integrable Anosov flow. As established in \cite{Dol98b} and \cite{Pol01}, hyperbolic flows of this type demonstrate rapid mixing with respect to $\mu_{\Phi}$.

Theorems \ref{thm1} and \ref{thm2} are consequences of the following two results with a Diophantine  hypothesis. Recalling the definition of a Diophantine subset in Definition \ref{Definition of Diophantine condition on Brin group}.

\begin{Theorem} \label{Rapid mixing for compact group extension of hyperbolic flow} 
If there exists a finite subset $\Gamma\subset H$ such that $\Gamma$ is Diophantine, then $f_{t}$ is rapidly mixing with respect to $\widehat{\mu}_{\Phi}$.
\end{Theorem}

\begin{proof}[\textbf{Proof of Theorem \ref{thm1} assuming Theorem \ref{Rapid mixing for compact group extension of hyperbolic flow}}]
If the Brin group $H$ is dense in $G$, and considering the compactness of $G$, we can identify a finite $\varepsilon$-net $\Gamma$ with $\varepsilon>0$ is suitably small. Thus, by Lemma \ref{Lemma 2.2.2}, $\Gamma$ is Diophantine. Then, the result follows from Theorem \ref{Rapid mixing for compact group extension of hyperbolic flow}.
\end{proof}


Under the same hypothesis, we also obtain the superpolynomial equidistribution. Let $\tau$ represent a closed orbit of $g_{t}$, and denote its least period as $\ell_{\tau}$. Each $\tau$ induces a conjugacy class $[\tau]$ in $G$ which is called the holonomy class of $\tau$. For any $T>0$, let $\pi(T)$ be the collection of prime closed orbits $\tau$ with $\ell_{\tau}\le T$. 

\begin{Theorem}\label{Superpolynomial equidistribution for compact group extension of hyperbolic flow}
If there exists a finite subset $\Gamma\subset H$ such that $\Gamma$ is Diophantine, then $f_{t}$ enjoys the superpolynomial equidistribution.
\end{Theorem}

\begin{proof}[\textbf{Proof of Theorem \ref{thm2} assuming Theorem \ref{Superpolynomial equidistribution for compact group extension of hyperbolic flow}}] This follows the same lines as the 
proof of Theorem \ref{thm1} assuming Theorem \ref{Rapid mixing for compact group extension of hyperbolic flow}
\end{proof}

We can   compare  our theorems  with prior results. In \cite{Fie05}, the authors studied rapid mixing for skew products of hyperbolic flows. They proved that, for a typical cocycle, the skew product flow is rapidly mixing 
for  special case of regular $G$-equivariant test functions associated with a fixed unitary representation $\pi:G\to U(d)$.  Theorem \ref{Rapid mixing for compact group extension of hyperbolic flow} is evidently stronger since it applies to all regular test functions. 
In \cite{Pol08} the authors investigated the error term of equidistribution for $\mathbb{T}^{1}$-extensions of hyperbolic flows. The authors proved that each non-trivial character satisfies the equidistribution theorem with a polynomial error term. Therefore, our Theorem \ref{Superpolynomial equidistribution for compact group extension of hyperbolic flow} generalizes their result by obtaining a stronger superpolynomial error term for any class function and any compact Lie group.



\subsection{Applications to frame flows}\label{sec 1.3}

The main examples of compact group extensions of hyperbolic flows are frame flows of compact manifolds of negative curvature. Applying Theorems \ref{thm1}, \ref{thm2}, \ref{Rapid mixing for compact group extension of hyperbolic flow}, and \ref{Superpolynomial equidistribution for compact group extension of hyperbolic flow}, we can prove that a substantial class of frame flows is rapidly mixing and satisfies the superpolynomial equidistribution. As a statistical property, the mixing rate of frame flows not only has its own importance but can also be used to solve problems in other fields, such as the surface subgroup conjecture of Waldhausen proved by Kahn and Markovic in \cite{Kah12}. 

Consider a smooth compact $d$-dimensional oriented manifold $V$ of negative curvature with $d \geq 3$. Let $M = T^{1}V$ denote the unit tangent bundle over $V$. The geodesic flow $g_{t}$ acts on $M$, translating every unit vector $(x,v)$ in $M$ by the parallel transportation  along the geodesic of $(x,v)$ at a distance $t$. It is well-known that $g_{t}$ is a smooth jointly non-integrable Anosov flow. A frame on $V$ is represented by a point $(x,v_{1},\cdots,v_{d})$ with $x \in V$ and $v_{i} \in T^{1}{x}V$, satisfying $(v_{i},v_{j}) = \delta_{ij}$ for each $1 \leq i,j \leq d$. Let $\widehat{M}$ be the set of positively oriented orthonormal frames on $V$. Notably, $\widehat{M}$ forms a smooth $SO(d-1)$-bundle over $M$, where the bundle projection is defined as the projection taking a frame to its first vector. The frame flow $f_{t}$ acts on $\widehat{M}$, translating every frame in $\widehat{M}$ along the geodesic of $(x,v_{1})$ at a distance $t$. It is known that $f_{t}$ is a smooth $SO(d-1)$-extension of $g_{t}$ \cite{Bri82}.

The geodesic flow $g_{t}$ exhibits exponential mixing with respect to the Liouville measure, which serves as the Gibbs measure derived from the geometric potential. Additionally, the frame flow $f_{t}$ preserves the local product measure of the Liouville measure and the normalized Haar measure on $SO(d-1)$. The ergodic properties of the frame flow have undergone extensive investigation \cite{Bri82} \cite{Bur03} \cite{Cek21}. 
{\color{black} As mentioned in the Introduction},  to date it has been established that $f_{t}$ is ergodic and mixing with respect to the local product measure in the following cases:
\begin{itemize}
{\color{black}
\item $d$ is odd and $d \neq 7$ \cite{Bri75b} \cite{Bri80}; 
\item $d$ is even or $d=7$, and $V$ has  negative sectional  curvatures satisfying a suitable pinching condition,  \cite{Bri75b} \cite{Bri80}, \cite{Cek21}.
}
\end{itemize}

Let $\tau$ represent a closed geodesic of $g_{t}$, and denote its length as $\ell_{\tau}$. The parallel translation along the closed geodesic induces a conjugacy class $[\tau]$ in $G$ which is called the holonomy class of $\tau$. For any $T>0$, let $\pi(T)$ denote the collection of prime closed geodesics $\tau$ with $\ell_{\tau}\le T$. We have the following result for frame flows.

\begin{Theorem}
{\color{black} Let $d \geq 4$ and} let $f_{t}$ be a frame flow in one of the above cases. Then:
\begin{enumerate}
\item it is rapidly mixing with respect to the local product measures of Gibbs measures and the normalized Haar measure Haar$_{G}$ on $G$; and 
\item it satisfies the superpolynomial equidistribution, namely for any $n\ge1$ there exist an integer $k\ge1$ and a constant $C_8>0$ such that for any class function $F\in C^{k}(G)$ and any $T>0$,
$$\bigg|\dfrac{1}{\#\pi(T)}\sum_{\tau\in\pi(T)}F([\tau])-\int_{G}Fd(\text{\rm Haar})_{G}(g)\bigg|\le C_8||F||_{C^{k}}{T^{-n}}.$$
\end{enumerate}
\end{Theorem}\begin{proof}

For $d\ge4$, the group $SO(d-1)$ is semi-simple. In this scenario, if $d$ is odd and $d\neq7$, it has been established in \cite{Bri80} that $\overline{H(x)}=SO(n-1)$ for any $x\in M$. Moreover, it is proven in \cite{Cek21} that if $d$ is even or $d=7$ and the negative curvature satisfies a specific pinching condition, then $\overline{H(x)}=SO(d-1)$ for any $x\in M$. Consequently, the result follows from Theorems \ref{thm1} and \ref{thm2}.
\end{proof}

  \begin{figure}[h!]
          \centerline{
    \begin{tikzpicture}[thick,scale=0.70, every node/.style={scale=1}]
\draw (0,0) ellipse (200pt and 100pt);
\draw[->, red, ultra thick] (0,-3.53) --(3,-3.53);
\draw (0,-3.53) ellipse (25pt and 60pt);
\draw[->, blue, ultra thick] (0,-3.53) --(0,-1.4);
\draw[->, blue, ultra thick] (0,-3.53) --(0.8,-4.4);
   \node at (3, -4.1) {$v_1$};
      \node at (0, -1) {$v_2$};
           \node at (1.2, -4.6) {$v_3$};
\draw (0,3.53) ellipse (25pt and 60pt);
\draw[->, blue, ultra thick] (0,3.53) --(0,1.4);
\draw[->, blue, ultra thick] (0,3.53) --(0.8,4.4);
   \node at (-3, 4.1) {$\phi_t(v_1)$};
   \node at (0,4) {$x'$};
      \node at (0, 1) {$v_2'$};
           \node at (1.2, 4.6) {$v_3'$};
                 \node at (0.0, -4.0) {$x$};
           \draw[->, red, ultra thick] (0,3.53) --(-3,3.53);
\end{tikzpicture}
}
\caption{The frame flow when $n=3$ where $f_t (x,v_1,v_2, v_3) = (x',\phi_t(v_1), v_2',v_3')$}
\end{figure}
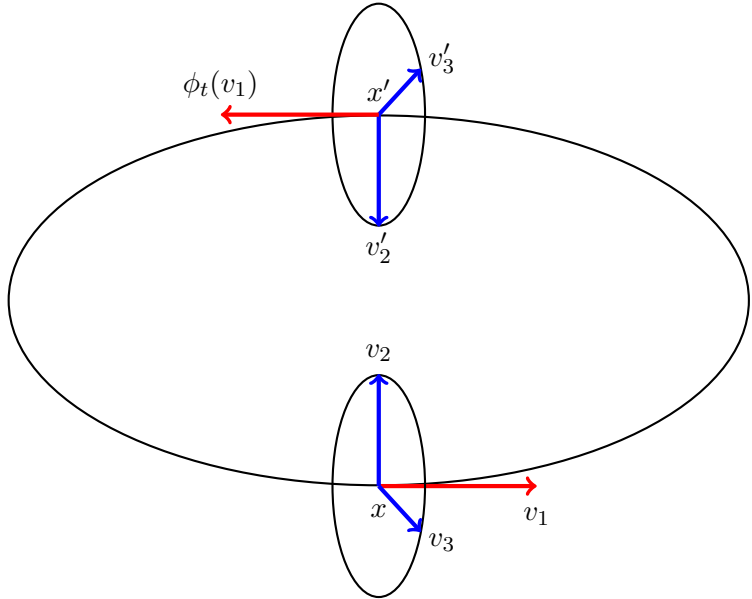

\section{Markov partitions and symbolic models}\label{sec 6}

In this section, we review classical subshifts of finite type and their associated symbolic models, with \cite{Par90} serving as a helpful  reference. 
Markov partitions of the underlying hyperbolic flows
play a crucial role in establishing that suspensions of subshifts serve as symbolic models for hyperbolic flows. Similarly, suspensions of skew products over subshifts emerge as symbolic models for compact group extensions of hyperbolic flows. This section aims to provide estimates and useful  insights into these symbolic models.

\subsection{Subshifts of finite type}\label{subsec 6.1}

Let $A$ be an $N\times N$ matrix of zeros and ones. We assume $A$ is aperiodic, meaning that some power of $A$ is a positive matrix. Let $X := X_{A}$ be the two-sided symbolic space associated with $A$, defined as $X = \{x=(x_{i})_{i=-\infty}^{\infty}\in\{1,\cdots,N\}^{\mathbb{Z}}: A(x_{i}, x_{i+1}) = 1\}$, and let $\sigma:X\to X$ be the two-sided subshift, defined as $(\sigma(x))_{i} = x_{i+1}$. 

\begin{Definition}
Given $\lambda \in (0, 1)$, we can define a metric $d_{\lambda}$ on $X$ by $d_{\lambda}(x, y) := \lambda^{N(x,y)}$, where $N(x,y) = \min\{|i| : x_{i} \neq y_{i}\}$.
\end{Definition}

We say $\lambda\in(0,1)$ is a metric constant on $X$ means that $X$ is assigned the metric $d_{\lambda}$. It can be proved that the metric $d_{\lambda}$ is compatible with the Tychonov product topology on $X$ where $\{1,\cdots,N\}$ is given the discrete topology. Then, by the Tychonov theorem, $(X, d_{\lambda})$ is a compact metric space \cite{Bow08}. It is not difficult to show that $\sigma:X\to X$ is topologically mixing, which means for any two open sets $U,V\subset X$ and any large enough $n\in\mathbb{N}^{+}$ we have $U\cap\sigma^{n}V\not=\emptyset$, if and only if $A$ is aperiodic. 

\begin{Definition}
Let $\lambda\in(0,1)$ be a metric constant on $X$.
\begin{itemize}
\item Denote by $F_{\lambda}(X)$ the Banach space of all Lipschitz continuous complex-valued functions on $(X,d_{\lambda})$ with respect to the Lipschitz norm $||\cdot||_{\text{Lip}}:=|\cdot|_{\text{Lip}} + |\cdot|_{\infty}$, where $|\cdot|_{\text{Lip}}$ is the Lipschitz semi-norm and $|\cdot|_{\infty}$ is the supremum norm. 

\item For a compact Lie group $G$, we define the space $F_{\lambda}(X, G)$ to be the set of functions from $(X,d_{\lambda})$ to $(G,d_{G})$ that are Lipschitz continuous.
\end{itemize}
\end{Definition}

It is important to note that, given a Banach space $(B, ||\cdot||)$, we can similarly define the Banach space $F_{\lambda}(X, B)$ with the norm $||\cdot||_{\text{Lip}}$. For instance, in \S \ref{subsec 7.1}, we will consider $B$ to be the Hilbert space $H_{\pi}\subset L^{2}(G)$ associated to a unitary representation $\pi$ of a compact Lie group $G$.  

\begin{Definition}
For each $n\in\mathbb{N}^{+}$, the $n$-cylinders are sets of the form $[x_{0}\cdots x_{n-1}]_{n} = \{y \in X : y_{i} = x_{i}, |i| \le n-1\}$. For each $n\in\mathbb{N}^{+}$, let $F_{n}(X)$ be the set of functions in $F_{\lambda}(X)$ that are constant on $n$-cylinders. 
\end{Definition}

The functions in \(F_{n}(X)\) are referred to as locally constant functions. All of the above objects can be similarly defined for a one-sided subshift of finite type \(\sigma:X^{+}\to X^{+}\), where \(X=\{x=(x_{i})_{i\ge0}\in\{1,\ldots,N\}^{\mathbb{N}} : A(x_{i}, x_{i+1})=1\}\), using the same notations but replacing \(X\) with \(X^{+}\).

For a potential \(\varphi\in F_{\lambda}(X,\mathbb{R})\), it is well-known \cite{Bow08} that there exists a unique equilibrium state of \(\varphi\) on \(X\), which is called the Gibbs measure of \(\varphi\) on \(X\), denoted as \(\mu_{\varphi}\). By adding  a coboundary \(h\circ\sigma-h\) if necessary, we can assume \(\varphi\) depends only on future coordinates, and thus \(\varphi\in F_{\lambda}(X^{+},\mathbb{R})\). Then, there exists a unique equilibrium state of \(\varphi\) on \(X^{+}\), which is called the Gibbs measure of \(\varphi\) on \(X^{+}\), denoted as \(\mu\). Let \(\Pi_{+}:X\to X^{+}\) be the coordinate projection. It is known that \(\Pi_{+}^{*}\mu_{\varphi}=\mu\) \cite{Bow75}. For each \(n\in\mathbb{N}^{+}\), let \(\varphi_{n}=\sum_{i=0}^{n-1}\varphi\circ\sigma^{i}\). The Gibbs measure \(\mu\) satisfies the following Gibbs property, see also \cite{Bow75}.

\begin{Lemma}\label{Gibbs property}
There exists $C_{9} \geq 1$ such that 
$$C_{9}^{-1}\le \dfrac{\mu[x_{0}\cdots x_{n-1}]_{n}}{e^{\varphi_{n}(x)-nP(\varphi)}}\le C_{9},$$
for any $n\in\mathbb{N}^{+}$ and any $x\in X^{+}$ where $P(\varphi)$ is the pressure of $\varphi$.
\end{Lemma}

\subsection{Suspensions of two-sided subshifts}\label{subsec 6.2}

Let $\sigma:X\to X$ be a two-sided subshift. 
For each $n\in\mathbb{Z}$ and any $x\in X$, we define 
$$W^{s}_{n}(x):=\{y\in X:y_{i}=x_{i},\ \forall i\ge n\}\quad\text{and}\quad W^{u}_{n}(x):=\{y\in X:y_{i}=x_{i},\ \forall i\le-n\}.$$
Let $W^{s}(x)=\bigcup_{n\in\mathbb{Z}}W^{s}_{n}(x)$ and $W^{u}(x)=\bigcup_{n\in\mathbb{Z}}W^{u}_{n}(x)$. Then 
$$W^{s}(x)=\{y\in X:d_{\lambda}(\sigma^{n}(x),\sigma^{n}(y))\to 0\text{ as }n\to\infty\} \hbox{ and }$$
$$  W^{u}(x)=\{y\in X:d_{\lambda}(\sigma^{-n}(x),\sigma^{-n}(y))\to0\text{ as }n\to\infty\}.$$  Given $r\in F_{\lambda}(X,\mathbb{R}^{+})$, let $\psi_{t}:X_{r}\to X_{r}$ be the suspension flow of $\sigma$ and $r$ defined in \S \ref{subsec 2.5}. For each $n\in\mathbb{N}^{+}$, let $r_{n}=\sum_{i=0}^{n-1}r\circ\sigma^{i}$. 

\begin{Definition}\label{Definition 6.2.1}
	For any $x,y\in X$ with $y\in W^{s}(x)$, we define $\Delta^{s}(x,y):=\lim_{n\to\infty} r_{n}(y)-r_{n}(x)$.\footnote{The limit exists because $y\in W^{s}(x)$ which means $d_{\lambda}(\sigma^{n}(y),\sigma^{n}(x))=O(\lambda^{n})$.} Similarly, for any $x,y\in X$ with $y\in W^{u}(x)$, we define $\Delta^{u}(x,y):=\lim_{n\to\infty} r_{n}(\sigma^{-n}(x))-r_{n}(\sigma^{-n}(y))$. 
\end{Definition}

The quantities \(\Delta^{s}\) and \(\Delta^{u}\) will be utilized to define the Brin group at the symbolic level, as illustrated in the next subsection. In the proof of Theorem \ref{Rapid mixing for compact group extension of hyperbolic flow}, we need to estimate the convergence rate in the limit \(\Delta^{u}\) as follows.

\begin{Lemma}\label{Lemma 6.2.2}
There exists $C_{10}>0$ such that for any $k\in\mathbb{N}^{+}$, any $n\in\mathbb{N}$ and any $x,y\in X$ with $y\in W^{u}_{n}(x)$,
$$|\Delta^{u}(x,y)-[r_{k}(\sigma^{-k}(x))-r_{k}(\sigma^{-k}(y))]|\le C_{10}\lambda^{k-n}.$$
\end{Lemma}
\begin{proof}
	We compute,
	$$
	\begin{aligned}
		&|\Delta^{u}(x,y)-[r_{k}(\sigma^{-k}(x))-r_{k}(\sigma^{-k}(y))]|\\
		=&\bigg|\sum_{i=1}^{\infty}[r(\sigma^{-i}(x))-r(\sigma^{-i}(y))]-\sum_{i=1}^{k}[r(\sigma^{-i}(x))-r(\sigma^{-i}(y))]\bigg|\\
		=&\bigg|\sum_{i=k+1}^{\infty}[r(\sigma^{-i}(x))-r(\sigma^{-i}(y))]\bigg|\\
		\le&|r|_{Lip}\lambda^{k-n}\dfrac{1}{1-\lambda}.
	\end{aligned}
	$$
	Therefore, we can set $C_{10}=|r|_{Lip}\frac{1}{1-\lambda}$ to complete the proof. 
\end{proof}

Below, we would like to introduce a class of test functions on $X_{r}$ that will be used in \S \ref{subsec 7.3}.

\begin{Definition}\label{Definition of test function space of symbolic flow for hyperbolic flow}
	For a map $E$ from $X_{r}$ to a Banach space $(B, ||\cdot||)$, define 
	$$|E|_{\lambda}:=\sup_{x\not=y}\sup_{ u\in[0,\min\{r(x),r(y)\}]}\dfrac{||E(x,u)-E(y,u)||}{d_{\lambda}(x,y)}.$$
	Define $||E||_{\lambda}:=||E||_{\infty}+|E|_{\lambda}$ where $||E||_{\infty}:=\sup_{(x,u)}||E(x,u)||$. Let $F_{\lambda}(X_{r},B)$ be the set of all maps $E$ from $X_{r}$ to $B$ with $||E||_{\lambda}<\infty$.
\end{Definition}

In \S \ref{subsec 7.3}, we will consider functions belong to $F_{\lambda}(X_{r}, C^{k}(G))$ and $F_{\lambda}(X_{r}, H_{\pi})$ where $G$ is a compact Lie group and $H_{\pi}\subset L^{2}(G)$ is the Hilbert space associated to a representation $\pi\in\widehat{G}$.

\subsection{Suspensions of skew products over two-sided subshifts}\label{subsec 6.3}

Let $\sigma:X\to X$ be a two-sided subshift. Given a compact Lie group $G$ and $\Theta\in F_{\lambda}(X,G)$, let $\widehat{\sigma}:\widehat{X}\to\widehat{X}$ be the skew product defined in \S \ref{subsec 2.4}. Given $r\in F_{\lambda}(X,\mathbb{R}^{+})$, we can think of $r$ is a continuous function on $\widehat{X}$ by setting $r(x,g)=r(x)$. Then, let $\phi_{t}:\widehat{X}_{r}\to\widehat{X}_{r}$
be the suspension flow of $r$ and $\widehat{\sigma}$ defined in \S \ref{subsec 2.5}. 

For a potential $\varphi\in F_{\lambda}(X,\mathbb{R})$ with its Gibbs measure $\mu_{\varphi}$, we introduced in \S \ref{subsec 2.4} that  $\widehat{\mu}_{\varphi}:=\mu_{\varphi}\times$Haar is a $\widehat{\sigma}$-invariant probability measure on $\widehat{X}$. Then, let $\widehat{\mu}_{\varphi}\times$Leb be the natural $\phi_{t}$-invariant probability measure induced by $\widehat{\mu}_{\varphi}$ on $\widehat{X}_{r}$. Denote by $\rho_{E,F}$ the correlation function of test functions $E$ and $F\in C(\widehat{X}_{r})$.
We can also study the rate of mixing of $\phi_{t}$ with respect to $\widehat{\mu}_{\varphi}\times$Leb. However, since the symbolic space $X$ lacks differential structure, the test functions $E$ and $F$ we should consider are a little different from the case of $f_{t}$.

\begin{Definition}\label{Definition of test function space of symbolic flow for compact group}
We will use the following notation.
\begin{itemize}
	\item For a map $E$ from $\widehat{X}_{r}$ to a Banach space $(B, ||\cdot||)$, define 
	$$|E|_{\lambda}:=\sup_{g\in G}\sup_{x\not=y}\sup_{ u\in[0,\min\{r(x),r(y)\}]}\dfrac{||E(x,g,u)-E(y,g,u)||}{d_{\lambda}(x,y)}.$$
	Define $||E||_{\lambda}:=||E||_{\infty}+|E|_{\lambda}$ where $||E||_{\infty}:=\sup_{(x,g,u)}||E(x,g,u)||$.
	
	\item For each $k\ge1$, let $F^{k}_{\lambda}(\widehat{X}_{r})$ be the set of all complex-valued functions $E$ on $\widehat{X}_{r}$ which are $C^{k}$ with respect to $(g,u)$ and $||E||_{\lambda,k}<\infty$ where $||E||_{\lambda,k}:=\sup_{k_{1}+k_{2}\le k}||\frac{\partial^{k_{1}+k_{2}}E}{\partial g^{k_{1}}\partial u^{k_{2}}}||_{\lambda}$.
\end{itemize}
\end{Definition}

This leads to the following definition of rapid mixing which is appropriate at the symbolic level.

\begin{Definition}\label{Definiiton 6.3.2}
We say $\phi_{t}$ is \emph{rapidly mixing} with respect to $\widehat{\mu}_{\varphi}\times$Leb if for any $n\in\mathbb{N}^{+}$ there exist $C_{11}\ge1$ and integer $k\in\mathbb{N}^{+}$ such that for any $E,F\in F_{\lambda}^{k}(\widehat{X}_{r})$ and any $t>0$ we have $|\rho_{E,F}(t)|\le C_{11}||E||_{\lambda,k}||F||_{\lambda,k}t^{-n}$.
\end{Definition}

In \S \ref{subsec 6.5}, via a Markov partition for the hyperbolic flow $g_{t}$, we can think of $f_{t}$ as being modelled by the above suspension flow $\phi_{t}$ and we will reduce rapid mixing of $f_{t}$ with respect to $\widehat{\mu}_{\Phi}$ to rapid mixing of $\phi_{t}$ with respect to $\widehat{\mu}_{\varphi}\times$Leb. 

We now introduce
some (symbolic) Diophantine conditions ensuring that \(\phi_{t}\) exhibits rapid mixing. For each \(n\in\mathbb{N}^{+}\), let \(\Theta_{n}=\Theta\circ\sigma^{n-1}\cdots\Theta\). By 
analogy to  \(\Delta^{s}\) and \(\Delta^{u}\), we can also define the quantities \(\Theta^{s}\) and \(\Theta^{u}\) as follows.

\begin{Definition}\label{Definition 6.3.3}
For any $x,y\in X$ with $y\in W^{s}(x)$, let $\Theta^{s}(x,y):=\lim_{n\to\infty}(\Theta_{n}(y))^{-1}\Theta_{n}(x)$.
The limit exists because $y\in W^{s}(x)$ which means $d_{\lambda}(\sigma^{n}(y),\sigma^{n}(x))=O(\lambda^{n})$. Similarly, for any $x,y\in X$ with $y\in W^{u}(x)$, let $\Theta^{u}(x,y):= \lim_{n\to\infty}\Theta_{n}(\sigma^{-n}(y))(\Theta_{n}(\sigma^{-n}(x)))^{-1}$.
\end{Definition}

As mentioned earlier, an estimate of the convergence rate in the limit \(\Delta^{u}\) will be employed to prove Theorem \ref{Rapid mixing for compact group extension of hyperbolic flow}. Similarly, an estimate of the convergence rate in the limit \(\Theta^{u}\) is also required as follows.

\begin{Lemma}\label{Lemma 6.3.4}
There exists $C_{12}>0$ such that for any $k\in\mathbb{N}^{+}$, any $n\in\mathbb{N}$ and any $x,y\in X$ with $y\in W^{u}_{n}(x)$,
$$d_{G}(\Theta^{u}(x,y),\Theta_{k}(\sigma^{-k}(y))(\Theta_{k}(\sigma^{-k}(x)))^{-1})\le C_{12}\lambda^{k-n}.$$
\end{Lemma}
\begin{proof}
We compute,
$$
\begin{aligned}
	&d_{G}(\Theta^{u}(x,y),\Theta_{k}(\sigma^{-k}(y))(\Theta_{k}(\sigma^{-k}(x)))^{-1})\\
	=&\lim_{p\to\infty}d_{G}(\Theta_{p}(\sigma^{-p}(y))(\Theta_{p}(\sigma^{-p}(x)))^{-1},\Theta_{k}(\sigma^{-k}(y))(\Theta_{k}(\sigma^{-k}(x)))^{-1})\\
	=&\lim_{p\to\infty}d_{G}((\Theta_{k}(\sigma^{-k}(y)))^{-1}\Theta_{p}(\sigma^{-p}(y)),(\Theta_{k}(\sigma^{-k}(x)))^{-1}\Theta_{p}(\sigma^{-p}(x)))\\
	=&\lim_{p\to\infty}d_{G}(\Theta_{p-k}(\sigma^{-p}(y)),\Theta_{p-k}(\sigma^{-p}(x)))\le |\Theta|_{\lambda}\lambda^{k-n}\dfrac{1}{1-\lambda}.
\end{aligned}
$$
Therefore, we can set $C_{12}=|\Theta|_{Lip}\frac{1}{1-\lambda}$ to complete the proof. 
\end{proof}

In the symbolic model $\phi_{t}:\widehat{X}_{r}\to\widehat{X}_{r}$, we can define the 
analogue of the 
Brin group as follows. We say $V=\{x_{i}\}_{i=0}^{p}$ is a closed chain at $x\in X$ if $x_{p}=x_{0}=x$ and $x_{i+1}\in W^{s}(x_{i})$ or $W^{u}(x_{i})$. Given $n_{0}\in\mathbb{Z}$ and $p_{0}\in\mathbb{N}^{+}$, we say $V$ is of size $n_{0}$ and length $p_{0}$ if $x_{i+1}\in W^{s}_{n_{0}}(x_{i})$ or $W^{u}_{n_{0}}(x_{i})$ and $p\le p_{0}$. Again each closed chain $V=\{x_{i}\}_{i=0}^{p}$ induces a twist 
$$g_{V}:=\Theta^{\delta_{x_{p-1},x_{p}}}(x_{p-1},x_{p})\cdots\Theta^{\delta_{x_{0},x_{1}}}(x_{0},x_{1})$$ where $\delta_{x,y}=s$ if $y\in W^{s}(x)$ and $\delta_{x,y}=u$ if $y\in W^{u}(x)$. The following definition  plays the role of the u-s paths  at the symbolic level.

\begin{Definition}\label{Definition of Brin group in the symbolic level}
For any $x\in X$ and each $n_{0}\in\mathbb{Z},\ p_{0}\in\mathbb{N}$, we define $\Gamma(x, n_{0},p_{0})$ to be the set $\{g_{V}\}$ of all closed chains $V=\{x_{i}\}_{i=0}^{p}$ at $x$ of size $n_{0}$ and length $p_{0}$ which satisfying $\sum_{i=0}^{p-1}\Delta^{\delta_{x_{i},x_{i+1}}}(x_{i},x_{i+1})=0$.
\end{Definition}

\subsection{Markov sections of hyperbolic flows}\label{subsec 6.4}

A Markov section of a hyperbolic flow allows us to build a bridge between the suspension flow introduced in the previous subsection 
 and the hyperbolic flow.


We follow \cite{Bow73} in  introducing  Markov sections of hyperbolic flows. We start by introducing the local product coordinate. Fix a hyperbolic flow $g_{t}:\Lambda\to\Lambda$. For a sufficiently small $\delta>0$ and any two points $x,y\in\Lambda$ with $d(x,y)<\delta$, the intersection $W^{s}_{\delta}(x)\cap \bigcup_{|t|\le\delta}g_{t}W^{u}_{\delta}(y)$ consists of a single point which belongs to $\Lambda$, and we denote it by $[x,y]$ which is usually called the local product coordinate of $x$ and $y$. A subset $B\subset W^{s}_{\delta}(z)\cap\Lambda$ is called proper if it is closed and $\overline{B^{o}}=B$ with respect to the induced topology on $W^{s}_{\delta}(z)\cap\Lambda$. A proper subset in $W^{u}_{\delta}(z)\cap\Lambda$ is defined similarly.

\begin{Definition}
A \emph{parallelogram} $R$ is a set of the form $R=[U,S]$ where $U\subset W^{u}_{\delta}(z)\cap\Lambda$ and $S\subset W^{s}_{\delta}(z)\cap\Lambda$ are proper subsets.
\end{Definition}

By definition, a parallelogram $R$ is a local cross section of the flow $g_{t}$. A parallelogram $R$ has disintegration with leaves $[x,S]$, namely $R=\cup_{x\in U}[x,S]$. Obviously, $[x,S]\subset W^{s}_{\delta}(x)\cap\Lambda$. Thus, we can  denote $[x,S]$ as $W^{s}(x,R)$.
We denote by $W^{u}(x,R) = [U,x]$ where $x\in S$. Two parallelograms $R$ and $R^{\prime}$ are called parallel if there exists a continuous function $r$ on $R$ such that $R^{\prime}=\cup_{x\in R}g_{r(x)}(x)$. Obviously, the translation of $R$ along the direction of the flow is always parallel to $R$, namely $g_{t}R$ is parallel to $R$.

Let $\{R_{i}\}_{i=1}^{N}$ be finitely many parallelograms with $R_{i}\cap R_{j}=\emptyset$ for $i\not=j$. Fix a small $\varepsilon>0$, we assume $\cup_{i}\cup_{|t|\le\varepsilon}g_{t}R_{i}=\Lambda$. Let $R=\cup_{i}R_{i}$ and $P:R\to R$ correspond to the first return map of $g_{t}$ on $R$.
\footnote{The map actually is the first return map on $\hbox{\rm int}(R) \cap P^{-1}\hbox{\rm int}(R)$ which will be sufficient to define the corresponding map at the symbolic level}

\begin{Definition}
$\{R_{i}\}_{i=1}^{N}$ is called a Markov section of $g_{t}$ if whenever $x\in R_{i}^{o}\cap P^{-1}R_{j}^{o}\not=\emptyset$, 
$$
PW^{s}(x,R_{i})^{o}\subset W^{s}(P(x),R_{j})^{o}\quad\text{and}\quad P^{-1}W^{u}(P(x),R_{j})^{o}\subset W^{u}(x,R_{i})^{o}.
$$
\end{Definition}

Given a Markov section $\{R_{i}\}_{i=1}^{N}$, let $r:R^{o} \cap  P^{-1}R^{o}\to\mathbb{R}^{+}$ be the first return time function for $P: R^{o} \cap  P^{-1}R^{o} \mapsto R^{o}$. 
We call the set $R_{i}^{r}:=
\overline{\cup_{x\in R^{o} \cap  P^{-1}R^{o}}\cup_{0\le t< r(x)}g_{t}(x)}$ a parallelepiped, and we call $\{R_{i}^{r}\}_{i=1}^{N}$ a Markov partition. The maximal value of the sizes of $R_{i}^{r}$, $1\le i\le N$, is called the size of $\{R_{i}^{r}\}_{i=1}^{N}$. The following well-known result is attributed to Bowen \cite{Bow73}.\footnote{ In the case where $g_{t}$ is an Anosov flow, the same result was obtained in \cite{Rat73}.}

\begin{Lemma}\label{Markov partition of hyperbolic flows}
There exists a Markov partition $\{R_{i}^{r}\}_{i=1}^{N}$ of $g_{t}:\Lambda\to\Lambda$ of arbitrarily small size.
\end{Lemma}

Note that, for a Markov section $\{R_{i}\}_{i=1}^{N}$, some parallelograms may be parallel. For the corresponding Markov partition $\{R_{i}^{r}\}_{i=1}^{N}$, they cover $\Lambda$ and have disjoint pairwise interiors. 

We can naturally view $\Lambda$ as the suspension space of $R$ under the function $r$ and view $g_{t}$ is the corresponding suspension flow of $P:R\to R$ and $r$. We can directly study the dynamics of $P$ and the suspension flow. But, it may be more convenient to view $R$ as a two-sided symbolic space and view $P:R\to R$ as a two-sided subshift. To this end, we can define a transition $N\times N$ matrix $A$ of zeros and ones by
$$
A(i,j)=\begin{cases}
1,\quad & \text{if}\ R^{o}_{i}\cap P^{-1}R_{j}^{o}\not=0;\\
0,\quad &\text{otherwise.}
\end{cases} 
$$
We can always assume $A$ is aperiodic, that is $A^{N}>0$ for some $N\in\mathbb{N}^{+}$. Let $\sigma:X\to X$ be the two-sided subshift. By \cite{Bow73}, we have the following result.

\begin{Lemma}\label{Semi-conjugacy of poincare map and subshifts}
There exists a Lipschitz continuous surjection $\Pi:X\to R$ such that $P:R\to R$ and $\sigma:X\to X$ are semi-conjugated ,i.e. $\Pi\circ\sigma=P\circ\Pi$.
\end{Lemma}

By considering $r\circ\Pi$, we can think of $r$ is a roof function on $X$, and $r$ only depends on future coordinates, i.e. $r(x)=r(y)$ whenever $x_{i}=y_{i}$ for any $i\ge0$. By choosing the metric constant $\lambda\in(0,1)$ on $X$ close sufficiently to 1, we can assume $r$ belongs to $F_{\lambda}(X,\mathbb{R})$. Let $\psi_{t}:X_{r}\to X_{r}$ be the suspension flow of $\sigma$ and $r$. The suspension space $X_{r}$ is a symbolic coordinate system of $\Lambda$ via $\Pi_{r}:X_{r}\to\Lambda$ defined by $\Pi_{r}(x,u)=g_{u}(\Pi(x))$, and $\psi_{t}$ is semi-conjugated to $g_{t}$ by $\Pi_{r}$.

Recall that $\mu_{\Phi}$ is the Gibbs measure of a H\"older potential $\Phi$ on $\Lambda$. By choosing the metric constant $\lambda\in(0,1)$ on $X$ close sufficiently to 1, we can assume the function $\varphi$ defined by $\varphi(x)=\int_{0}^{r(x)}(\Phi\circ\Pi_{r})(x,u)du-P(\Phi)r(x)$ belongs to $F_{\lambda}(X,\mathbb{R})$. Let $\mu_{\varphi}$ be the Gibbs measure of $\varphi$, and let $\mu_{\varphi}\times$Leb be the $g_{t}$-invariant probability measure induced by $\mu_{\varphi}$ on $X_{r}$. The following result is well-known \cite{Bow75}.

\begin{Lemma}\label{Gibbs measure and symbolic measure}
We have $\Pi_{r}^{*}(\mu_{\varphi}\times\text{Leb})=\mu_{\Phi}$.
\end{Lemma}

Using the symbolic coordinate system $X_{r}$, we can express the stable manifold $W^{s}$ and unstable manifold $W^{u}$ of $g_{t}$ as follows. Recalling the definitions of $\Delta^{s}$ and $\Delta^{u}$ in Definition \ref{Definition 6.2.1}. 

\begin{Lemma}\label{Symbolic coordinate for stable and unstable manifold}
For any $z,z^{\prime}\in\Lambda$ with $z^{\prime}\in W^{s}(z)$, if they have symbolic coordinates $(x,u)$ and $(x^{\prime},u^{\prime})$ respectively, then $x^{\prime}\in W^{s}(x)$ and $u^{\prime}=u+\Delta^{s}(x,x^{\prime})$. Similarly, for any $z,z^{\prime}\in\Lambda$ with $z^{\prime}\in W^{u}(z)$, if they have symbolic coordinates $(x,u)$ and $(x^{\prime},u^{\prime})$ respectively, then $x^{\prime}\in W^{u}(x)$ and $u^{\prime}=u+\Delta^{u}(x,x^{\prime})$. 
\end{Lemma}
\begin{proof}
For any point $z\in\Lambda$, we can write $z=g_{u}(\Pi(x))$ where $x\in X$ and $0\le u<r(x)$. Therefore the symbolic coordinate of $z$ is $(x,u)\in X_{r}$. Assume $z^{\prime}\in W^{s}(z)$ and $z^{\prime}$ has symbolic coordinate  $(x^{\prime},u^{\prime})\in X_{r}$. By definition, we have
$$
\begin{aligned}
&\lim_{t\to\infty}d(g_{t+u}(\Pi(x)),g_{t+u^{\prime}}(\Pi(x^{\prime})))\\
=&\lim_{t\to\infty}d(\psi_{t}(x,u),\psi_{t}(x^{\prime},u^{\prime}))=\lim_{t\to\infty}d((x,u+t),(x^{\prime},u^{\prime}+t))\\
=&\lim_{n\to\infty}d((\sigma^{n}(x),u+t-r_{n}(x)),(\sigma^{n}(x^{\prime}),u^{\prime}+t-r_{n}(x^{\prime})))=0
\end{aligned}
$$
which implies $\lim_{n\to\infty}d_{\lambda}(\sigma^{n}(x),\sigma^{n}(x^{\prime}))=0$ and  $\lim_{n\to\infty}|u-r_{n}(x)-u^{\prime}+r_{n}(x^{\prime})|=0$. Therefore $x^{\prime}\in W^{s}(x)$ and $u^{\prime}=u+\Delta^{s}(x,x^{\prime})$. 

The proof of the case of $W^{u}$ is similar to the above. Assume $z^{\prime}\in W^{u}(z)$, and $z,z^{\prime}\in\Lambda$ have symbolic coordinate $(x,u), (x^{\prime},u^{\prime})\in X_{r}$. By definition, we have
$$
\begin{aligned}
&\lim_{t\to\infty}d(g_{-t+u}(\Pi(x)),g_{-t+u^{\prime}}(\Pi(x^{\prime})))\\
=&\lim_{t\to\infty}d(\psi_{-t}(x,u),\psi_{-t}(x^{\prime},u^{\prime}))=\lim_{t\to\infty}d((x,u-t),(x^{\prime},u^{\prime}-t))\\
=&\lim_{n\to\infty}d((\sigma^{-n}(x),u-t+r_{n}(\sigma^{-n}(x))),(\sigma^{-n}(x^{\prime}),u^{\prime}-t+r_{n}(\sigma^{-n}(x^{\prime}))))=0
\end{aligned}
$$
which implies $\lim_{n\to-\infty}d_{\lambda}(\sigma^{-n}(x),\sigma^{-n}(x^{\prime}))=0$ and  $\lim_{n\to\infty}|u+r_{n}(\sigma^{-n}(x))-u^{\prime}-r_{n}(\sigma^{-n}(x^{\prime}))|=0$. Therefore $x^{\prime}\in W^{u}(x)$ and $u^{\prime}=u+\Delta^{u}(x,x^{\prime})$. 
\end{proof}

\subsection{A symbolic model for the extension flow}\label{subsec 6.5}

In this subsection, our goal is to address the rapid mixing issue for the extension flow $f_{t}$ by connecting it to the analogous problem presented by a symbolic model. In broad terms, this symbolic model is realized as a suspension flow of a skew product over a two-sided subshift
as we outline below.

Following subsection \ref{subsec 6.4}, for a hyperbolic flow $g_{t}$
we can assume there is a  Markov section $\{R_{i}\}_{i=1}^{N}$ for $g_{t}$ with arbitrarily small size. Let $R:=\bigcup_{i}R_{i}$, and define $P:R^{o} \cap  P^{-1}R^{o}\to R^o$ using  the first return map, with $r:R^{o} \cap  P^{-1}R^{o} \to\mathbb{R}^{+}$ representing the first return time. This Markov section $\{R_{i}\}_{i=1}^{N}$ for the underlying hyperbolic flow $g_{t}$  allows us to choose  a finite number of local cross-sections $\{\widehat R_{i}\}_{i=1}^{N}$ for the extension flow $f_{t}$ with $R_{i} = \varrho(\widehat R_{i})$ ($i=1, \cdots, N$), where $\varrho:\widehat{M}\to M$  is the bundle projection. Given that the parallelepipeds $\{R_{i}^{r}\}_{i=1}^{N}$ form a partition of $\Lambda$, it follows that $\{\widehat R_{i}^{r}\}_{i=1}^{N}$ form a partition of $\widehat{\Lambda}$. By invoking Lemma \ref{Markov partition of hyperbolic flows}, we can ensure that each parallelepiped $R_{i}^{r}$ is sufficiently small.

For each $\widehat R_{i}^{r}$, we can (locally) trivialize it as $R_{i}^{r}\times G$. Consequently, we can equate the local cross-sections $\{\widehat R_{i}\}_{i=1}^{N}$ with $\{R_{i}\times G\}_{i=1}^{N}$. Let $\widehat{R}:=R\times G$. By construction, the function corresponding to the first return time function on $\widehat{R}$ is $r$. Let $\widehat{P}:\widehat{R}\to \widehat{R}$ be the map corresponding to the first return map. We define a function $\Theta:R\to G$ using the formula:
\begin{equation}\label{3.2.1}
\widehat{P}(w,e)=f_{r(w)}(w,e)=(P(w),\Theta(w)),\quad w\in R,
\end{equation}
where $e\in G$ is the identity element. Now, we can code  $P:R\to R$ by  a two-sided subshift $\sigma:X\to X$ via the semi-conjugacy $\Pi:X\to R$. By a slight abuse of notation, we can write $r$ and $\Theta$ as functions on $X$  representing  $r\circ\Pi$ and $\Theta\circ\Pi$  on $R^{o} \cap  P^{-1}R^{o}$. Since $r\in F_{\lambda}(X,\mathbb{R}^{+})$, according to \eqref{3.2.1} and Lemma \ref{Semi-conjugacy of poincare map and subshifts}, we deduce that $\Theta\in F_{\lambda}(X,G)$. Let $\phi_{t}:\widehat{X}_{r}\to\widehat{X}_{r}$ be the suspension flow of $\widehat{\sigma}:\widehat{X}\to\widehat{X}$ and $r$, as defined in \S \ref{subsec 6.3}. Define $\widehat{\Pi}_{r}:\widehat{X}_{r}\to\widehat{\Lambda}$ as $\widehat{\Pi}_{r}(x,g,u)=f_{u}(\Pi(x),g)$. 

\begin{Lemma}\label{Semi-conjugacy of extension flow and symbolic flow}
We have $r\in F_{\lambda}(X^{+},\mathbb{R}^{+})$, $\Theta\in F_{\lambda}(X^{+},G)$ and $\widehat{\Pi}_{r}$ is a semi-conjugacy between $f_{t}$ and $\phi_{t}$.
\end{Lemma}
\begin{proof}
By  the construction of the Markov section and since the stable foliation $W^{s}$ is $g_{t}$-invariant, we  can assume that  $r$ is constant on each leaf which means it depends solely on future coordinates. Thus, $r\in F_{\lambda}(X^{+},\mathbb{R}^{+})$. Furthermore, using a  $G$-coycle,    if necessary, we ensure that $\Theta$ also depends only on future coordinates. For instance, for each $x\in X$, a coordinate transformation $g\to g^{\prime}=\Theta(x)g$ suffices. Thus, we  can assume  $\Theta\in F_{\lambda}(X^{+},G)$.

Using  \eqref{3.2.1} and the definition of a compact group extension, for any $(w,g)\in\widehat{R}$, we obtain
\begin{equation}\label{3.2.2}
\widehat{P}(w,g)=f_{r(w)}(w,g)=(P(w),\Theta(w)g)
\end{equation}
indicating that $\widehat{P}$ is a skew product of $P$. Through a direct application of \eqref{3.2.2} and Lemma \ref{Semi-conjugacy of poincare map and subshifts}, $\phi_{t}$ is semi-conjugated to $f_{t}$ through $\widehat{\Pi}_{r}$.
\end{proof}

In order to relate  the rapid mixing of $f_{t}$ to that of $\phi_{t}$, it is necessary  to assess the smoothness of $\widehat{\Pi}_{r}$. 

\begin{Lemma}\label{Smoothness of semi-conjugate}
There exists $C_{13}>0$ such that for each $k\in\mathbb{N}^{+}$ and any $F\in C^{k}(\widehat{M})$ we have $F\circ\widehat{\Pi}_{r}\in F^{k-1}_{\lambda}(\widehat{X}_{r})$ and $||F\circ\widehat{\Pi}_{r}||_{\lambda,k-1}\le C_{13}||F||_{C^{k}}$.
\end{Lemma}
\begin{proof}
By definition,
$$(F\circ\widehat{\Pi}_{r})(x,g,u)=F(f_{u}(\Pi(x),g)).$$
Then the result follows from $f_{t}$ being $C^{\infty}$ and $\Pi$ being Lipschitz continuous.
\end{proof}

Recall that the potential $\Phi$ defined on $\Lambda$ induces a potential $\varphi \in F_{\lambda}(X,\mathbb{R})$. Let $\widehat{\mu}_{\varphi}\times Leb$ represent the natural $\phi_{t}$-invariant measure, as defined in \S \ref{subsec 2.5}. According to Lemma \ref{Semi-conjugacy of extension flow and symbolic flow}, the semi-conjugacy $\widehat{\Pi}_{r}$ gives rise to a $f_{t}$-invariant measure $\widehat{\Pi}_{r}^{*}(\widehat{\mu}_{\varphi}\times Leb)$ on $\widehat{\Lambda}$. 
The next lemma shows the coincidence of these measures.

\begin{Lemma}\label{Coincide of two measures}
We have $\widehat{\Pi}_{r}^{*}(\widehat{\mu}_{\varphi}\times Leb)=\widehat{\mu}_{\Phi}$.
\end{Lemma}
\begin{proof}
Recalling that each parallelepiped $R_{i}^{r}$ is sufficiently small, and $\{R_{i}^{r}\times G\}_{i=1}^{N}$ form a partition of $\widehat{\Lambda}$. We aim to demonstrate the equality: $$\int_{\widehat{\Lambda}}Fd\widehat{\mu}_{\Phi}=\int_{\widehat{\Lambda}}Fd\widehat{\Pi}_{r}^{*}(\widehat{\mu}_{\varphi}\times Leb),\quad\text{for any } F\in C(\widehat{\Lambda}).$$
Leveraging Lemma \ref{Gibbs measure and symbolic measure}, we start by expressing the left-hand side of the equation,
\begin{equation}\label{3.2.3}
\begin{aligned}
\int_{\widehat{\Lambda}}Fd\widehat{\mu}_{\Phi}=&\sum_{i=1}^{N}\int_{R_{i}^{r}\times G}Fd(\mu_{\Phi}\times\text{Haar}_{G})\\
=&\int_{X_{r}\times G}F(\Pi_{r}(x,u),g)d((\mu_{\varphi}\times Leb)\times\text{Haar}_{G})((x,u),g)\\
=&\int_{G}\int_{X}\int_{0}^{r(x)}F(\Pi_{r}(x,u),g)dud\mu_{\varphi}(x)d\text{Haar}_{G}(g)\\
=&\int_{X}\int_{0}^{r(x)}\int_{G}F(\Pi_{r}(x,u),g)d\text{Haar}_{G}(g)dud\mu_{\varphi}(x).
\end{aligned}
\end{equation}
By definition, for the right-hand side of the equation, we have
\begin{equation}\label{3.2.4}
\begin{aligned}
\int_{\widehat{\Lambda}}Fd\widehat{\Pi}_{r}^{*}(\widehat{\mu}_{\varphi}\times Leb)=&\int_{\widehat{X}_{r}}F\circ\widehat{\Pi}_{r}d(\widehat{\mu}_{\varphi}\times Leb)\\
=&\int_{X}\int_{G}\int_{0}^{r(x)}F(\widehat{\Pi}_{r}(x,g,u))dud\text{Haar}_{G}(g)d\mu_{\varphi}(x)\\
=&\int_{X}\int_{0}^{r(x)}\int_{G}F(\widehat{\Pi}_{r}(x,g,u))d\text{Haar}_{G}(g)dud\mu_{\varphi}(x).
\end{aligned}
\end{equation}
By the definitions of $\widehat{\Pi}_{r}$ and a compact group extension, we have
\begin{equation}\label{3.2.5}
\pi(\widehat{\Pi}_{r}(x,g,u))=\pi(f_{u}(\Pi(x),g))=g_{u}(\Pi(x))=\Pi_{r}(x,u)
\end{equation}
where $\pi:\widehat{M}\to M$ is the bundle projection. Then, by \eqref{3.2.5} and the definitions of a compact group extension, we can establish
\begin{equation*}
\int_{G}F(\widehat{\Pi}_{r}(x,g,u))d\text{Haar}_{G}(g)=\int_{G}F(\Pi_{r}(x,u),g)d\text{Haar}_{G}(g).
\end{equation*}
Now, we substitute this result into \eqref{3.2.3}, and then combine it with \eqref{3.2.4} to complete the proof.
\end{proof}

This allows us to identify correlation functions for $f_t$ and $\phi_t$ on suitable classes of functions.

\begin{Lemma}\label{Suspension flow to extension flow: rapid mixing}
If $\phi_{t}$ is rapidly mixing with respect to $\widehat{\mu}_{\varphi}\times Leb$, then $f_{t}$ is rapidly mixing with respect to $\widehat{\mu}_{\Phi}$.
\end{Lemma}
\begin{proof}
The result follows from Lemmas \ref{Semi-conjugacy of extension flow and symbolic flow}, \ref{Smoothness of semi-conjugate} and \ref{Coincide of two measures}.
\end{proof}

\subsection{Relating the Diophantine assumptions for $f_t$ and the symbolic model}\label{subsec 6.6}

In this subsection, we present a result regarding the rapid mixing of $\phi_{t}$ under a Diophantine assumption. Additionally, we will demonstrate that the Diophantine condition on the Brin group of $f_{t}$ implies the Diophantine assumption of $\phi_{t}$, thereby proving Theorem \ref{Rapid mixing for compact group extension of hyperbolic flow} through the aid of Lemma \ref{Suspension flow to extension flow: rapid mixing}. 
We can express $f_{t}$ as a local skew product flow. Specifically, we can represent $f_{t}(z,g)=(g_{t}(z),S(z,t)g)$ whenever $|t|$ is sufficiently small, where $S$ is a cocycle function.

For $z^{\prime}\in W^{\delta_{z,z^{\prime}}}(z)$, where $\delta_{x,y}=s$ if $y\in W^{s}(x)$ and $\delta_{x,y}=u$ if $y\in W^{u}(x)$, if $z$ has a symbolic coordinate $(x,u)$ and $z^{\prime}$ has a symbolic coordinate $(x^{\prime},u^{\prime})$, as demonstrated in Lemma \ref{Symbolic coordinate for stable and unstable manifold}, we have shown that $u^{\prime}=u+\Delta^{\delta_{x,x^{\prime}}}(x,x^{\prime})$. We can express their stable twist $\mathcal{T}^{s}(z,z^{\prime})$ and unstable twist $\mathcal{T}^{u}(z,z^{\prime})$ as follows. Recalling the definitions of $\Delta^{s}, \Delta^{u}$ in Definition \ref{Definition 6.2.1} and the definitions of $\Theta^{s},\Theta^{u}$ in Definition \ref{Definition 6.3.3}. The following result will be used to prove the Diophantine assumption on the Brin group of $f_{t}$ implies a similar Diophantine assumption of $\phi_{t}$ in Lemma \ref{Diophantine subset of extension flow to suspension flow}.

\begin{Lemma}\label{Symbolic coordinate of stable and unstable twists for compact group extension flow}
For any $(z,g)\in\widehat{\Lambda}$ and $(z^{\prime},\mathcal{T}^{\delta_{z,z^{\prime}}}(z,z^{\prime})g)\in W^{\delta_{z,z^{\prime}}}_{f_{t}}(z,g)$ with $d(z,z^{\prime})$ small enough, and assume $z,z^{\prime}$ have symbolic coordinate $(x,u)$ and $(x^{\prime},u+\Delta^{\delta_{x,x^{\prime}}}(x,x^{\prime}))$ respectively, then we have $\mathcal{T}^{\delta_{z,z^{\prime}}}(z,z^{\prime})=S(\Pi(x^{\prime}),u+\Delta^{\delta_{x,x^{\prime}}}(x,x^{\prime}))\Theta^{\delta_{x,x^{\prime}}}(x,x^{\prime})S^{-1}(\Pi(x),u)$.
\end{Lemma}

\begin{proof}
We first prove the case where $\delta_{z,z^{\prime}}=s$, i.e., $z^{\prime}\in W^{s}(z)$. If $d(z,z^{\prime})$ is sufficiently small, then $u+\Delta^{s}(x,x^{\prime})$ is also  small. Utilizing the definition of the semi-conjugacy $\widehat{\Pi}_{r}:\widehat{X}_{r}\to\widehat{\Lambda}$, we find that the symbolic coordinate of $(z,g)$ is $(x,S^{-1}(\Pi(x),u)g,u)$, and the symbolic coordinate of $(z^{\prime},\mathcal{T}^{s}(z,z^{\prime})g)$ is $(x^{\prime},S^{-1}(\Pi(x^{\prime}),u+\Delta^{s}(x,x^{\prime}))\mathcal{T}^{s}(z,z^{\prime})g,u+\Delta^{s}(x,x^{\prime}))$. Now, we have
$$
\begin{aligned}
&\lim_{t\to\infty}d\big((x,S^{-1}(\Pi(x),u)g,u+t),(x^{\prime},S^{-1}(\Pi(x^{\prime}),u+\Delta^{s}(x,x^{\prime}))\mathcal{T}^{s}(z,z^{\prime})g,u+\Delta^{s}(x,x^{\prime})+t)\big)\\
=&\lim_{t\to\infty}d(f_{t}(z,g),f_{t}(z^{\prime},\mathcal{T}^{s}(z,z^{\prime})g))=0.
\end{aligned}
$$
This now   implies that $$\mathcal{T}^{s}(z,z^{\prime})=S(\Pi(x^{\prime}),u+\Delta^{s}(x,x^{\prime}))\lim_{n\to\infty}(\Theta_{n}(x^{\prime}))^{-1}\Theta_{n}(x)S^{-1}(\Pi(x),u)$$
which proves the case where $z^{\prime}\in W^{s}(z)$.

Now, assume $z^{\prime}\in W^{u}(z)$. If $d(z,z^{\prime})$ is sufficiently small, then $u+\Delta^{u}(x,x^{\prime})$ is likewise small. Thus, the symbolic coordinate of $(z,g)$ is $(x,S^{-1}(\Pi(x),u)g,u)$, and the symbolic coordinate of $(z^{\prime},\mathcal{T}^{u}(z,z^{\prime})g)$ is $(x^{\prime},S^{-1}(\Pi(x^{\prime}),u+\Delta^{u}(x,x^{\prime}))\mathcal{T}^{u}(z,z^{\prime})g,u+\Delta^{u}(x,x^{\prime}))$. Now, we have
$$
\begin{aligned}
&\lim_{t\to\infty}d\big((x,S^{-1}(\Pi(x),u)g,u-t),(x^{\prime},S^{-1}(\Pi(x^{\prime}),u+\Delta^{u}(x,x^{\prime}))\mathcal{T}^{u}(z,z^{\prime})g,u+\Delta^{u}(x,x^{\prime})-t)\big)\\
=&\lim_{t\to\infty}d(f_{-t}(z,g),f_{-t}(z^{\prime},\mathcal{T}^{s}(z,z^{\prime})g))=0.
\end{aligned}
$$
The above implies that $$\mathcal{T}^{u}(z,z^{\prime})=S(\Pi(x^{\prime}),u+\Delta^{u}(x,x^{\prime}))\lim_{n\to\infty}\Theta_{n}(\sigma^{-n}(x^{\prime}))(\Theta_{n}(\sigma^{-n}(x)))^{-1}S^{-1}(\Pi(x),u)$$
which proves the case where $z^{\prime}\in W^{u}(z)$ and thus completing the proof.
\end{proof}

Recall  the definition of a Diophantine subset in Definition \ref{Definition of Diophantine condition on Brin group} and the definition of $\Gamma(x,n_{0},p_{0})$ in Definition \ref{Definition of Brin group in the symbolic level}. Also recall the definition of the Brin group $H(z)$ at a point $z\in\Lambda$ in \S \ref{subsec 4.3}.

\begin{Lemma}\label{Diophantine subset of extension flow to suspension flow}
If there exist $z\in\Lambda$ and a finite subset $\Gamma\subset H(z)$ such that $\Gamma$ is Diophantine, then there exist $x\in X$, $n_{0}\in\mathbb{N}$ and $p_{0}\in\mathbb{N}^{+}$ such that $\Gamma(x,n_{0},p_{0})$ is Diophantine.
\end{Lemma}
\begin{proof} 
Each $g_{W}\in\Gamma$ associated a closed chain $W=\{z_{i}\}_{i=0}^{p}$ with $z_{0}=z_{p}=z$. Since $\Gamma$ is finite, by cutting each $W$ into smaller segments if necessary, we can assume there exists $p_{0}\in\mathbb{N}^{+}$ such that for each $g_{W}\in\Gamma$ we have $p\le p_{0}$ and $d(z_{i},z_{i+1})$ is sufficiently small for any $0\le i\le p-1$. Assume $z$ has symbolic coordinate $(x,u)$ and $z^{i}$ has symbolic coordinate $(x_{i},u_{i})$ for each $0\le i\le p$. We can assume there exists $n_{0}\in\mathbb{N}$ such that for each $g_{W}\in\Gamma$ we have $x_{i+1}\in W^{s}_{n_{0}}(x_{i})$ or $W^{u}_{n_{0}}(x_{i})$ for each $0\le i\le p-1$.
	
Since $W$ is a closed chain, we have $V=\{x_{i}\}_{i=0}^{p}$ is a closed chain of size $n_{0}$ and length $p_{0}$. Therefore $\{g_{V}:g_{W}\in\Gamma\}\subset\Gamma(x,n_{0},p_{0})$.  Note that, we have $u_{i}=u+\sum_{k=0}^{i-1}\Delta^{\delta_{x_{k+1},x_{k}}}(x_{k},x_{k+1})$ for $1\le i\le p$, and $\sum_{k=0}^{p-1}\Delta^{\delta_{x_{k+1},x_{k}}}(x_{k},x_{k+1})=0$. Furthermore, by Lemma \ref{Symbolic coordinate of stable and unstable twists for compact group extension flow}, we have 
$$
\begin{aligned}
&g_{W}\\
=&\mathcal{T}^{j_{z_{p},z_{p-1}}}(z_{p-1},z_{p})\cdots\mathcal{T}^{j_{z_{1},z_{0}}}(z_{0},z_{1})\\
=&S(\Pi(x),u)\Theta^{j_{x_{p},x_{p-1}}}(x_{p-1},x_{p})\cdots \Theta^{j_{x_{1},x_{0}}}(x_{0},x_{1})S^{-1}(\Pi(x),u)\\
=&S(\Pi(x),u)g_{V}S^{-1}(\Pi(x),u).
\end{aligned}
$$ 
Thus we have $\{g_{V}:g_{W}\in\Gamma\}=S^{-1}(\Pi(x),u)\Gamma S(\Pi(x),u)$. In particular, $\Gamma$ is Diophantine implies $\{g_{V}:g_{W}\in\Gamma\}$ is also Diophantine, and so by definition $\Gamma(x,n_{0},p_{0})$ is Diophantine.
\end{proof}

We will prove the following result in \S \ref{sec 7} 

\begin{Theorem}\label{Rapid mixing for symbolic flow of compact group extension}
If there exist $x\in X$, $n_{0}\in\mathbb{N}$ and $p_{0}\in\mathbb{N}^{+}$ such that $\Gamma(x,n_{0},p_{0})$ is Diophantine, then $\phi_{t}$ is rapidly mixing with respect to $\widehat{\mu}_{\varphi}\times Leb$.
\end{Theorem}

Given Theorem \ref{Rapid mixing for symbolic flow of compact group extension} it is easy to deduce
Theorem
\ref{Rapid mixing for compact group extension of hyperbolic flow}.

\begin{proof}[\textbf{Proof of Theorem \ref{Rapid mixing for compact group extension of hyperbolic flow} assuming Theorem \ref{Rapid mixing for symbolic flow of compact group extension}}]
The proof can be established by combining Lemmas \ref{Suspension flow to extension flow: rapid mixing}, \ref{Diophantine subset of extension flow to suspension flow}, and Theorem \ref{Rapid mixing for symbolic flow of compact group extension}.
\end{proof}

\section{Rapid mixing for the symbolic flow}\label{sec 7}

It remains to prove  Theorem \ref{Rapid mixing for symbolic flow of compact group extension}, and thus we always assume $\phi_{t}$ satisfies the assumption of Theorem \ref{Rapid mixing for symbolic flow of compact group extension} in this section. We will use a standard strategy which  involves  studying  the Laplace transform of the correlation function. By demonstrating that the Laplace transform can be analytically extended to a specific region in the left half-plane, one can derive a decay rate for the correlation function, which depends on the shape of this region. The Laplace transform can be expressed as a sum of iterations of transfer operators (Lemma \ref{Lemma 3.3.3}), and the analytical extension can be established through a Dolgopyat-type estimate of these transfer operators (Proposition \ref{Dolgopyat type estimate for rapid mixing of compact group extension}). 

\subsection{Decomposing and modifying correlation functions}\label{subsec 7.1}

Recall that  the correlation function for $E, F\in F_{\lambda}^{k}(\widehat{X}_{r})$ is defined by
$$\rho_{E,F}(t)=\int_{\widehat{X}_{r}}E\circ\phi_{t}\cdot\overline{F} d({\widehat{\mu}_{\varphi}\times\text{Leb}})-\int_{\widehat{X}_{r}}E d({\widehat{\mu}_{\varphi}\times\text{Leb}})\int_{\widehat{X}_{r}}\overline{F} d({\widehat{\mu}_{\varphi}\times\text{Leb}}).$$
For simplicity, we may sometimes write $\rho_{E,F}$ as $\rho$ when there is no confusion. As usual, by replacing $E$ by  $E- \int_{\widehat{X}_{r}}E d({\widehat{\mu}_{\varphi}\times\text{Leb}})$
and 
$F$ by 
$F- \int_{\widehat{X}_{r}}F d({\widehat{\mu}_{\varphi}\times\text{Leb}})$,  we can assume $\int_{\widehat{X}_{r}}E d{\widehat{\mu}_{\varphi}\times\text{Leb}}=\int_{\widehat{X}_{r}}\overline{F} d{\widehat{\mu}_{\varphi}\times\text{Leb}}=0$. Fix $k_{1}\in\mathbb{N}$, and we write $k=2k_{1}+k_{2}$. We can  interpret $E:X_{r}\to C^{2k_{1}}(G)$ by setting $E(x,u)(\cdot)=E(x,\cdot,u)$. Since $E, F\in F_{\lambda}^{k}(\widehat{X}_{r})$, we have $E, F\in F_{\lambda}(X_{r}, C^{2k_{1}}(G))$. By Lemma \ref{Lemma 2.7.6}, we can write
$$
E(x,u)=\sum_{\pi\in\widehat{G}}E_{\pi}(x,u)\quad\text{and}\quad F(x,u)=\sum_{\pi\in\widehat{G}}F_{\pi}(x,u)
$$
where $E_{\pi}(x,u)=\dim_{\pi}\text{Tr}(\widehat{E(x,u)}(\pi)\pi)\in H_{\pi}$ and $F_{\pi}(x,u)=\dim_{\pi}\text{Tr}(\widehat{F(x,u)}(\pi)\pi)\in H_{\pi}$. Recall that $\widehat{E(x,u)}(\pi)=\int_{G}\pi(g^{-1})E(x,g,u)d\text{Haar}_{G}(g)$. In particular, from that, we obtain $E_{\pi}$ is $C^{k_{2}}$ with respect to $u$ and $\frac{\partial^{j} E}{\partial u^{j}}=\sum_{\pi\in\widehat{G}}\frac{\partial^{j} E_{\pi}}{\partial u^{j}}$ for each $0\le j\le k_{2}$. We need the following estimate for $E_{\pi}$ and $F_{\pi}$.

\begin{Lemma}\label{Lemma 3.3.1}
For any $E,F\in F_{\lambda}^{k}(\widehat{X}_{r})\subset F_{\lambda}(X_{r}, C^{2k_{1}}(G))$ and any $1\not=\pi\in\widehat{G}$, we have $E_{\pi},F_{\pi}\in F_{\lambda}(X_{r}, H_{\pi})$. Moreover,
$$||E_{\pi}||_{\lambda}\le\dfrac{C_{3}||E||_{\lambda,2k_{1}}}{|\lambda_{\pi}|^{2k_{1}-3m_{G}/2}}\quad\text{and}\quad||F_{\pi}||_{\lambda}\le\dfrac{C_{3}||F||_{\lambda,2k_{1}}}{|\lambda_{\pi}|^{2k_{1}-3m_{G}/2}}.$$
\end{Lemma}
\begin{proof}
By Lemma \ref{Lemma 2.7.8}, 
$$||E_{\pi}(x,u)|| \le |E_{\pi}(x,u)|_{\infty} \le C_{3} \dfrac{||E(x,u)||_{C^{2k_{1}}}}{|\lambda_{\pi}|^{2k_{1}-3m_{G}/2}}$$
which implies $||E_{\pi}||_{\infty} \le \frac{C_{3}||E||_{\lambda,2k_{1}}}{|\lambda_{\pi}|^{2k_{1}-3m_{G}/2}}$.
Meanwhile, we have
$$\dfrac{E(x,u)-E(y,u)}{d_{\lambda}(x,y)} = \sum_{\pi \in \widehat{G}} \dfrac{E_{\pi}(x,u)-E_{\pi}(y,u)}{d_{\lambda}(x,y)}.$$
Thus, by definition, $|E_{\pi}|_{\lambda} \le \frac{C_{3}||E||_{\lambda,2k_{1}}}{|\lambda_{\pi}|^{2k_{1}-3m_{G}/2}}$.
\end{proof}

From the above estimate, if we choose $2k_{1} > \dim_{G}$, implying $2k_{1}-3m_{G}/2 > r_{G}$, then by Lemma \ref{Lemma 2.7.5}, the convergent series  $\sum_{\pi \in \widehat G}E_{\pi}$ and $\sum_{\pi \in \widehat G}F_{\pi}$ are absolutely convergent  and uniform in $(x,g,u)$. The decomposition of $E$ and $F$ gives us the following decomposition of the correlation function.

\begin{Lemma}\label{Lemma 3.3.2}
For any $E, F \in F_{\lambda}^{k}(\widehat{X}_{r})\subset F_{\lambda}(X_{r}, C^{2k_{1}}(G))$ with $2k_{1}> \dim_{G}$, we have $\rho_{E,F} = \sum_{\pi\in\widehat{G}}\rho_{E_{\pi},F_{\pi}}$.
\end{Lemma}
\begin{proof}
Since $\sum_{\pi\in\widehat{G}}F_{\pi}$ is absolutely convergent, we have $\overline{F} = \sum_{\pi\in\widehat{G}}\overline{F_{\pi}}$. Moreover, as the convergences $\sum_{\pi \in \widehat G}E_{\pi}$ and $\sum_{\pi}F_{\pi \in \widehat G}$ are uniform in $(x,g,u)$, we can compute

$$
\begin{aligned}
&\rho_{E,F}(t)\\
=&\int_{\widehat{X}_{r}}E\circ\phi_{t}\overline{F} d(\widehat{\mu}_{\varphi}\times\text{Leb})\\
=&\int_{G}\int_{X}\int_{0}^{r(x)}E(x,g,u+t)\overline{F}(x,g,u)dud\mu_{\varphi}(x)d\text{(Haar)}_{G}(g)\\
=&\sum_{\pi\in\widehat{G}}\sum_{\pi^{\prime}\in\widehat{G}}\int_{G}\int_{X}\int_{0}^{r(x)}E_{\pi}(x,g,u+t)\overline{F_{\pi^{\prime}}}(x,g,u)dud\mu_{\varphi}(x)d\text{(Haar)}_{G}(g)\\
=&\sum_{\pi\in\widehat{G}}\sum_{\pi^{\prime}\in\widehat{G}}\int_{X}\int_{0}^{r(x)}\int_{G}E_{\pi}(x,g,u+t)\overline{F_{\pi^{\prime}}}(x,g,u)d\text{(Haar)}_{G}(g)dud\mu_{\varphi}(x)\\
=&\sum_{\pi\in\widehat{G}}\int_{X}\int_{0}^{r(x)}\int_{G}E_{\pi}(x,g,u+t)\overline{F_{\pi}}(x,g,u)d\text{(Haar)}_{G}(g)dud\mu_{\varphi}(x)\\
=&\sum_{\pi\in\widehat{G}}\rho_{E_{\pi},F_{\pi}}(t).
\end{aligned}
$$
\end{proof}

Since $E_{\pi}$ and $F_{\pi}$ are $C^{k_{2}}$ with respect to $u$, the correlation function $\rho_{E_{\pi},F_{\pi}}$ is $C^{k_{2}}$. It satisfies, for any $t > 0$ and any $0 \leq j \leq k_{2}$, the following conditions:
\begin{equation}\label{3.3.1}
	|\rho_{E_{\pi},F_{\pi}}|_{\infty} \leq ||E_{\pi}||_{\infty}||F_{\pi}||_{\infty}, \quad \rho_{E_{\pi},F_{\pi}}^{(j)}(t) = \rho_{\frac{\partial^{j}E_{\pi}}{\partial u^{j}},F_{\pi}}(t).
\end{equation}
We want to study the Laplace transform of $\rho_{E_{\pi},F_{\pi}}$. However, for technical reasons, we first need to make a slight modification to $\rho_{E_{\pi},F_{\pi}}$. To this end, we can construct a $C^{k_{2}}$ modification function $\chi_{E_{\pi},F_{\pi}}$ of $\rho_{E_{\pi},F_{\pi}}$ satisfying the following properties:
\begin{equation}\label{3.3.2}
\begin{aligned}
&|\chi_{E_{\pi},F_{\pi}}|_{\infty} \leq ||E_{\pi}||_{\infty}||F_{\pi}||_{\infty};\\
&\chi_{E_{\pi},F_{\pi}}^{(j)}(0) = 0, \quad \text{for each}\ 0 \leq j \leq k_{2};\\
&\chi_{E_{\pi},F_{\pi}}(t) = \rho_{E_{\pi},F_{\pi}}(t), \quad \text{for any}\ t \geq 1; \hbox{ and }\\
&|\chi^{(j)}_{E_{\pi},F_{\pi}}(t)| \leq 2^{j}||E_{\pi}||_{\infty}||F_{\pi}||_{\infty}, \quad \text{for any}\ 0 \leq t \leq 1 \ \text{and each}\ 0 \leq j \leq k_{2}.
\end{aligned}
\end{equation}

        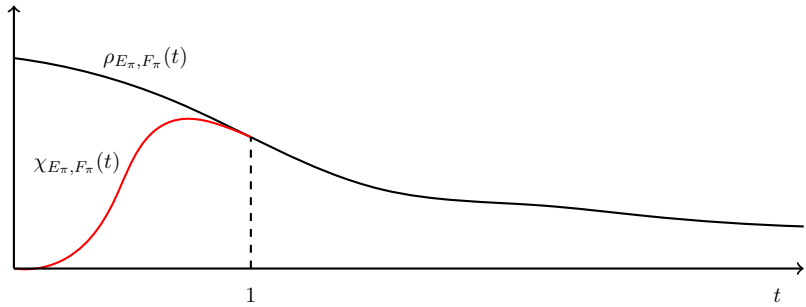
\begin{figure}[h!]
          \centerline{
    \begin{tikzpicture}[thick,scale=0.70, every node/.style={scale=0.75}]
 \draw plot [hobby] coordinates {(0,4) (3,3.2) (7,1.5) (10,1.2) (12,1) (15,0.8)};
  \draw[color=red] plot [hobby] coordinates {(0,0) (1,0.2) (2,1.5) (3,2.8) (4,2.7) (4.5,2.5)};
\draw[->, black] (0,0)-- (15,0);
\draw[->, black] (0,0)-- (0,5);
\draw[-, dashed] (4.5,0)-- (4.5,2.5);
    \node at (14.5, -0.5) {$t$};
      \node at (4.5, -0.5) {$1$};
       \node at (2.5, 4) {$\rho_{E_\pi, F_\pi}(t)$};
              \node at (1.2, 2.0) {$\chi_{E_\pi, F_\pi}(t)$};
\end{tikzpicture}
}
\caption{The correlation function $\rho_{E_\pi, F_\pi}$  is replaced by $\chi_{E_\pi, F_\pi}$.}
\end{figure}

For instance, we can let $\chi_{E_{\pi},F_{\pi}}(t)=0$ for $0\le t\le 1/2$, $\chi_{E_{\pi},F_{\pi}}(t)=\rho_{E_{\pi},F_{\pi}}(t)$ for $t\ge1$, and $\{\chi_{E_{\pi},F_{\pi}}(t)\}_{1/2\le t\le1}$ is a suitable smooth curve.   
Since $\chi_{E_{\pi},F_{\pi}}(t)$ and $\rho_{E_{\pi},F_{\pi}}(t)$ agree for $t > 1$ they clearly
have the same asymptotic. We may sometimes write $\chi_{E_{\pi},F_{\pi}}$ as $\chi_{\pi}$ when there is no confusion.

\subsection{The Laplace transform and transfer operators}\label{subsec 7.2}

Given $s \in \mathbb{C}$, we write $s = a+ib$. Consider the Laplace transform $\widehat{\chi}_{\pi}(s) = \int_{0}^{\infty} e^{-st}\chi_{\pi}(t)dt$. As $\chi_{\pi}$ is bounded, $\widehat{\chi}_{\pi}$ is analytic on $\{s=a+ib: a > 0\}$. It is classical \cite{Pol85} that we can express $\widehat{\chi}_{\pi}(s)$ as a sum of transfer operators over the one-sided space $X^{+}$. However, before delving into that, we need to make some necessary preparations. 

We want to define a transfer  operator associated with the representation $\pi$.

\begin{Definition}
	For any $\pi\in\widehat{G}$ and any $s\in\mathbb{C}$, we define the following transfer operator:
	$$\mathcal{L}_{s,\pi}:F_{\lambda^{1/4}}(X^{+},H_{\pi})\to F_{\lambda^{1/4}}(X^{+},H_{\pi}),\quad \mathcal{L}_{s,\pi}h(x)=\sum_{\sigma(y)=x}e^{\varphi(y)-sr(y)}\pi(\Theta(y))h(y).$$
\end{Definition}

Recall that as a subspace of \(L^{2}(G)\), the metric on \(H_{\pi}\) is the natural \(L^{2}\) norm. For convenience, instead of using \(\|\cdot\|_{L^{2}}\), we use \(\|\cdot\|\) to represent this metric. For \(h\in F_{\lambda^{1/4}}(X^{+},H_{\pi})\), denote \(\|h\|_{\infty}=\sup_{x}\|h(x)\|\) and \(|h|_{\text{Lip}}\) as the Lipschitz constant with respect to the metric \(d_{\lambda^{1/4}}\) on \(X^{+}\) and the metric \(\|\cdot\|\) on \(H_{\pi}\). The proof of the following result is similar to \cite[Proposition 3]{Dol98a} and will be given in \S \ref{Appendix A}.

\begin{Lemma}\label{Lemma 3.3.3}
	There exists $C_{14}>0$ such that for any $\pi\in\widehat{G}$ and any integer $p\ge0$, there exists a family of operators $\{M_{\pi,p}(s):F_{\lambda}(X_{r},H_{\pi})\to F_{\lambda^{1/4}}(X^{+},H_{\pi})\}_{s\in\mathbb{C}}$ such that
	\begin{enumerate}
		\item For any $s=a+ib$ with $a>0$ and any $E,F\in F_{\lambda}^{k}(\widehat{X}_{r})$ with $k>\dim_{G}$, we can write 
		$$\widehat{\chi}_{\pi}(s)=\Phi_{1}(s)+\Phi_{2}(s)+\Phi_{3}(s)+\Phi_{4}(s)$$
		where
		\begin{align*}
			\Phi_{1}(s) &= \int_{0}^{1}e^{-st}(\chi_{\pi}(t)-\rho_{E_{\pi},F_{\pi}}(t))dt,\\
			\Phi_{2}(s) &= \int_{0}^{|r|_{\infty}}\int_{G}\int_{X}\int_{0}^{\max\{0,r-t\}}e^{-st}
			\langle E_{\pi}\circ\phi_{t}\overline{F}_{\pi} \rangle
			dud\mu_{\varphi} d(\hbox{\rm Haar})_{G}(g)dt,\\
			\Phi_{3}(s) &= \sum_{n\ge1}\sum_{q\ge0}\sum_{p<n+q}\langle M_{\pi,p}(s)E_{\pi}], \mathcal{L}_{\overline{s},\pi}^{n+q-p}[M_{\pi,q}(-\overline{s})F_{\pi}\rangle_{\mu\times\text{Haar}_{G}},\\
			\Phi_{4}(s) &= \sum_{n\ge1}\sum_{q\ge0}\sum_{p\ge n+q}
			\langle
			\mathcal{L}_{s,\pi}^{p-n-q}[M_{\pi,p}(s)E_{\pi}], [M_{\pi,q}(-\overline{s})F_{\pi}]
			\rangle_{\mu\times\text{Haar}_{G}},
		\end{align*}
and $	\langle\cdot, \cdot\rangle_{\mu\times\text{Haar}_{G}}$	means the inner product in the Hilbert space $L^{2}(\mu\times\text{Haar}_{G})$.

		\item The map $s\mapsto M_{\pi,p}(s)$ is analytic, and 
		$$||M_{\pi,p}(s)E_{\pi}||_{\infty}\le C_{14}\lambda^{p}e^{|a||r|_{\infty}(p+1)}||E_{\pi}||_{\lambda},$$
		as well as
		$$|M_{\pi,p}(s)E_{\pi}|_{Lip}\le C_{14}(|s|+|\lambda_{\pi}|^{1+m_{G}/2})\lambda^{p/2}e^{|a||r|_{\infty}(p+1)}||E_{\pi}||_{\lambda}.$$
	\end{enumerate}
\end{Lemma}

\subsection{A Dolgopyat type estimate}\label{subsec 7.3}

In view of Lemma \ref{Lemma 3.3.3}, we need to analyze  the operator $\mathcal{L}_{s,\pi}$ acts on $F_{\lambda}(X^{+},H_{\pi})$. Note that, it is sufficient to analysis $\mathcal{L}_{ib,\pi}$, since $\mathcal{L}_{s,\pi}$ can be estimated via a small perturbation of $\mathcal{L}_{ib,\pi}$. To this end, in this subsection, we prove a Dolgopyat type estimate: Proposition \ref{Dolgopyat type estimate for rapid mixing of compact group extension}. The main ingredient in the proof is a cancellation estimate stated in Lemma \ref{Lemma 3.4.4} and will be proved in \S \ref{subsec 7.5}. 

For any $b\in\mathbb{R}$ and each $\pi\in\widehat{G}$, let  \(b_{\pi}=|b|+C_{15}|\lambda_{\pi}|^{1+m_{G}/2}\) for suitable $C_{15} > 0$. Since $\inf_{\pi\not=1}|\lambda_{\pi}|>0$, we choose \(C_{15}>0\) such that \(b_{\pi}\) is large enough for any \((b,\pi)\) with \(\pi\neq 1\). We begin with the following Lasota-Yorke inequality. \footnote{We use the term Lasota-Yorke inequality in accordance with the modern usage. However, earlier manifestations of
such inequalities are due to Doeblin-Fortet and Ionescu Tulcea-Marinescu.}

\begin{Lemma}\label{Lemma 3.4.1}
There exists $C_{16}>0$ such that for any $(b,\pi)$ with $\pi\not=1$, any $h\in F_{\lambda}(X^{+}, H_{\pi})$ and any $n\ge1$,
$$|\mathcal{L}_{ib,\pi}^{n}h|_{Lip}\le C_{16}b_{\pi}||h||_{\infty}+\lambda^{n}|h|_{Lip}.$$
\end{Lemma}
\begin{proof}
	We prove  the lemma   by induction. For the case $n=1$, by a basic triangle inequality, we have for any $x, x^{\prime}\in X^{+}$,
	$$
	\begin{aligned}
		&||\mathcal{L}_{ib,\pi}h(x)-\mathcal{L}_{ib,\pi}h(x^{\prime})||\\
		=&||\sum_{\sigma(y)=x}e^{(\varphi-ibr)(y)}\pi(\Theta(y))h(y)-\sum_{\sigma(y^{\prime})=x^{\prime}}e^{(\varphi-ibr)(y^{\prime})}\pi(\Theta(y^{\prime}))h(y^{\prime})||\\
		\le&||\sum_{y,y^{\prime}}(e^{(\varphi-ibr)(y)}-e^{(\varphi-ibr)(y^{\prime})})\pi(\Theta(y))h(y)||\\
		&+||\sum_{y,y^{\prime}}e^{(\varphi-ibr)(y^{\prime})}(\pi(\Theta(y))h(y)-\pi(\Theta(y^{\prime}))h(y^{\prime}))||\\
		\le&\lambda(2+e^{|\varphi|_{Lip}})(|\varphi|_{Lip}+|b||r|_{Lip}) d_{\lambda}(x,x^{\prime})\sum_{y}e^{\varphi(y)}||h(y)||\\
		&+||\sum_{y,y^{\prime}}e^{(\varphi-ibr)(y^{\prime})}(\pi(\Theta(y))h(y)-\pi(\Theta(y))h(y^{\prime}))||\\
		&+||\sum_{y,y^{\prime}}e^{(\varphi-ibr)(y^{\prime})}(\pi(\Theta(y))h(y^{\prime})-\pi(\Theta(y^{\prime}))h(y^{\prime}))||.
	\end{aligned}
	$$
	Then, since we can assume  $\mathcal{L}_{\varphi}1=1$ and using  Corollary \ref{Corollary 2.7.4}, we can bound the above by
	$$
	\begin{aligned}
		&\lambda(2+e^{|\varphi|_{Lip}})(|\varphi|_{Lip}+|b||r|_{Lip})||h||_{\infty} d_{\lambda}(x,x^{\prime})+|h|_{Lip}\lambda d_{\lambda}(x,x^{\prime})\\
		&+\sum_{y^{\prime}}e^{f(y^{\prime})}C_{3}|\lambda_{\pi}|^{1+m_{G}/2}|\Theta|_{Lip}d_{\lambda}(x,x^{\prime})||h(y^{\prime})||\\
		\le&\lambda(2+e^{|\varphi|_{Lip}})(|\varphi|_{Lip}+|b||r|_{Lip})||h||_{\infty} d_{\lambda}(x,x^{\prime})+|h|_{Lip}\lambda d_{\lambda}(x,x^{\prime})\\
		&+\lambda C_{3}|\lambda_{\pi}|^{1+m_{G}/2}|\Theta|_{Lip}||h||_{\infty}d_{\lambda}(x,x^{\prime}).
	\end{aligned}
	$$
	Thus, we have
	\begin{equation}\label{3.4.1}
		|\mathcal{L}_{ib,\pi}h|_{Lip}\le C_{17}b_{\pi}\lambda||h||_{\infty}+|h|_{Lip}\lambda
	\end{equation}
	for some uniform constant $C_{17}>0$. Inductively, for $n=k$ with $k\ge1$, we assume 
	\begin{equation}\label{3.4.2}
		|\mathcal{L}_{ib,\pi}h|^{k}_{Lip}\le C_{17}b_{\pi}\sum_{j=1}^{k}\lambda^{j}||h||_{\infty}+|h|_{Lip}\lambda^{k}.
	\end{equation}
	Now for $n=k+1$, by \eqref{3.4.1} and \eqref{3.4.2}, 
	$$
	\begin{aligned}
		|\mathcal{L}^{k+1}_{ib,\pi}h|_{Lip}\le& C_{17}b_{\pi}\sum_{j=1}^{k}\lambda^{j}||\mathcal{L}_{ib,\pi}h||_{\infty}+\lambda^{k}|\mathcal{L}_{ib,\pi}h|_{Lip}\\
		\le&C_{17}b_{\pi}\sum_{j=1}^{k}\lambda^{j}||h||_{\infty}+\lambda^{k+1}C_{17}(1+b_{\pi})||h||_{\infty}+\lambda^{k+1}|h|_{Lip}\\
		=&C_{17}b_{\pi}\sum_{j=1}^{k+1}\lambda^{j}||h||_{\infty}+\lambda^{k+1}|h|_{Lip}
	\end{aligned}
	$$
	which proves the induction. Finally, we can set $C_{16}=C_{17}\lambda(1-\lambda)^{-1}$ to complete the proof.
\end{proof}

Fix a $\lambda^{\prime}\in(\lambda,1)$, and then we choose \(C_{18}\geq 1\) such that \(C_{16}/C_{18}+\lambda\le\lambda^{\prime}\). For any \((b,\pi)\) with \(\pi\neq 1\), we introduce a norm on $F_{\lambda}(X^{+}, H_{\pi})$ by \(||h||_{b_{\pi}}=\max\{||h||_{\infty},|h|_{\text{Lip}}/C_{18}b_{\pi}\}\). 

\begin{Corollary}\label{Cor 3.4.2}
	For any \((b,\pi)\) with \(\pi\neq 1\), any \(n\ge1\), and any \(h\in F_{\lambda}(X^{+}, H_{\pi})\), we have \(\frac{|\mathcal{L}^{n}_{ib,\pi}h|_{\text{Lip}}}{C_{18}b_{\pi}}\le\lambda^{\prime}||h||_{b_{\pi}}\) and \(||\mathcal{L}^{n}_{ib,\pi}h||_{b_{\pi}}\le||h||_{b_{\pi}}\).
\end{Corollary}
\begin{proof}
	By Lemma \ref{Lemma 3.4.1}, for any \((b,\pi)\) with \(\pi\neq 1\), any \(n\ge1\), and any \(h\in F_{\lambda}(X^{+}, H_{\pi})\),
	\begin{equation*}
		\dfrac{|\mathcal{L}^{n}_{ib,\pi}h|_{\text{Lip}}}{C_{18}b_{\pi}}\le \dfrac{C_{16}}{C_{18}}||h||_{\infty}+\dfrac{\lambda^{n}}{C_{18}b_{\pi}}|h|_{\text{Lip}}.
	\end{equation*}
	In particular,
	\begin{equation*}
		\dfrac{|\mathcal{L}^{n}_{ib,\pi}h|_{\text{Lip}}}{C_{18}b_{\pi}}\le\bigg(\dfrac{C_{16}}{C_{18}}+\lambda^{n}\bigg)||h||_{b_{\pi}}.
	\end{equation*}
	Note that \(||\mathcal{L}^{n}_{ib,\pi}h||_{\infty}\le||h||_{\infty}\). Thus, provided \(\frac{C_{16}}{C_{18}}+\lambda\le\lambda^{\prime}\), we have \(||\mathcal{L}^{n}_{ib,\pi}h||_{b_{\pi}}\le||h||_{b_{\pi}}\) which completes the proof.
\end{proof}

\begin{Lemma}\label{Lemma 3.4.3}
	For any $(b, \pi)$ with $\pi\not=1$, any $h\in F_{\lambda}(X^{+}, H_{\pi})$ with $|h|_{Lip}\ge 2C_{18}b_{\pi}||h||_{\infty}$ and any $n\ge1$, we have $||\mathcal{L}^{n}_{ib,\pi}h||_{b_{\pi}}\le\lambda^{\prime}||h||_{b_{\pi}}$.
\end{Lemma}
\begin{proof}
Since $\lambda>0$ is close to 1, we can assume $\lambda\ge 1/2$. We have $||\mathcal{L}^{n}_{ib,\pi}h||_{\infty}\le||h||_{\infty}\le|h|_{Lip}/2C_{18}b_{\pi}\le2^{-1}||h||_{b_{\pi}}\le\lambda||h||_{b_{\pi}}$. Thus, by Corollary \ref{Cor 3.4.2}, we have $||\mathcal{L}^{n}_{ib,\pi}h||_{b_{\pi}}\le\lambda^{\prime}||h||_{b_{\pi}}$.
\end{proof}

The Diophantine assumption for the suspension flow $\phi_{t}$ is enough to ensure we get the following cancellation of terms in a transfer operator.

\begin{Lemma}\label{Lemma 3.4.4}
There exist $C_{19}>0, C_{20}>0$ and $C_{21}>0$ such that for any $(b, \pi)$ with $\pi\not=1$ and any $h\in F_{\lambda}(X^{+}, H_{\pi})$ with $|h|_{Lip}\le 2C_{18}b_{\pi}||h||_{\infty}$, there exists a subset $U\subset X^{+}$ with $\mu(U)\ge b_{\pi}^{-C_{19}}$ such that $||\mathcal{L}^{C_{21}\log b_{\pi}}_{ib,\pi}h(x)||\le(1-b_{\pi}^{-C_{20}})||h||_{\infty}$ for any $x\in U$.
\end{Lemma}

The above lemma will be proved in \S \ref{Appendix A}. As a corollary of Lemma \ref{Lemma 3.4.4}, we can obtain a cancellation of the oscillatory integral $\int_{X^{+}}||\mathcal{L}_{ib,\pi}^{C_{21}\log b_{\pi}}h||d\mu$ as follows:
\begin{equation}\label{3.4.3}
	\begin{aligned}
		\int_{X^{+}}||\mathcal{L}_{ib,\pi}^{C_{21}\log b_{\pi}}h||d\mu=&\int_{U}||\mathcal{L}_{ib,\pi}^{C_{21}\log b_{\pi}}h||d\mu+\int_{X^{+}-U}||\mathcal{L}_{ib,\pi}^{C_{21}\log b_{\pi}}h||d\mu\\
		\le&(1-b_{\pi}^{-C_{20}})||h||_{\infty}\mu(U)+||h||_{\infty}\mu(X^{+}-U)\\
		=&(1-b_{\pi}^{-C_{20}}\mu(U))||h||_{\infty}\le(1-b_{\pi}^{-C_{19}-C_{20}})||h||_{\infty}.
	\end{aligned}
\end{equation}
To strengthen the above $L^{1}$ contraction to a $||\cdot||_{\infty}$-contraction, we require the following lemma. 

\begin{Lemma}\label{Lemma 3.4.5}
There exist $C_{22}>0$ and $\delta_{1}\in(0,1)$ such that for any $h\in F_{\lambda}(X^{+})$ and any $k\ge1$,
$$||\mathcal{L}_{\varphi}^{k}h||_{Lip}\le \int_{X^{+}}|h|d\mu+C_{22}\delta_{1}^{k}||h||_{Lip}.$$
\end{Lemma}
\begin{proof}
	This is a directly corollary of the spectral gap of $\mathcal{L}_{\varphi}$ acts on $F_{\lambda}(X^{+})$ \cite{Bal00}.
\end{proof}

\begin{Lemma}\label{Lemma 3.4.6}
	There exist $C_{23}, C_{24}>0$ such that for any $(b, \pi)$ with $\pi\not=1$ and any $h\in F_{\lambda}(X^{+}, H_{\pi})$ with $|h|_{Lip}\le 2C_{18}b_{\pi}||h||_{\infty}$,
	$$||\mathcal{L}_{ib,\pi}^{(C_{24}+C_{21})\log b_{\pi}}h||_{\infty}\le\bigg(1-b_{\pi}^{-C_{23}}\bigg)||h||_{\infty}.$$
\end{Lemma}
\begin{proof}
	By \eqref{3.4.3}, Corollary \ref{Cor 3.4.2} and Lemma \ref{Lemma 3.4.5}, have
	$$
	\begin{aligned}
		&||\mathcal{L}_{ib,\pi}^{(C_{24}+C_{21})\log b_{\pi}}h||_{\infty}\le|\mathcal{L}_{\varphi}^{C_{24}\log b_{\pi}}||\mathcal{L}_{ib,\pi}^{C_{21}\log b_{\pi}}h|||_{\infty}\\
		\le&\int_{X^{+}}||\mathcal{L}_{ib,\pi}^{C_{21}\log b_{\pi}}h||d\mu+C_{22}\delta_{1}^{C_{24}\log b_{\pi}}||\mathcal{L}_{ib,\pi}^{C_{21}\log b_{\pi}}h||_{Lip}\\
		\le&(1-b_{\pi}^{-C_{19}-C_{20}})||h||_{\infty}+2C_{22}\delta^{C_{24}\log b_{\pi}}C_{18}b_{\pi}||h||_{\infty}\\
		\le&(1-b_{\pi}^{-C_{23}})||h||_{\infty}
	\end{aligned}
	$$
	which completes the proof.
\end{proof}

The above $||\cdot||_{\infty}$-contraction implies the following $||\cdot||_{b_{\pi}}$-contraction.

\begin{Corollary}\label{Cor 3.4.7}
	For any $(b, \pi)$ with $\pi\not=1$ and any $h\in F_{\lambda}(X^{+}, H_{\pi})$ with $|h|_{Lip}\le 2C_{18}b_{\pi}||h||_{\infty}$,
	$$||\mathcal{L}_{ib,\pi}^{(C_{24}+C_{21})\log b_{\pi}}h||_{b_{\pi}}\le\bigg(1-b_{\pi}^{-C_{23}}\bigg)||h||_{b_{\pi}}.$$
\end{Corollary}
\begin{proof}
	This comes from Corollary \ref{Cor 3.4.2} and Lemma \ref{Lemma 3.4.6}.
\end{proof}

The core technical result of this paper is the following Dolgopyat type estimate.

\begin{Proposition}\label{Dolgopyat type estimate for rapid mixing of compact group extension}
There exists $C_{25} > 0$ such that for any $1 \neq \pi \in \widehat{G}$, any $b \in \mathbb{R}$, and any $h \in F_{\lambda}(X^{+}, H_{\pi})$, we have $||\mathcal{L}_{ib,\pi}^{n_{b_{\pi}}}h||_{b_{\pi}} \le (1 - b_{\pi}^{-C_{23}}) ||h||_{b_{\pi}}$, where $n_{b_{\pi}} = C_{25} \log b_{\pi}$.
\end{Proposition}
\begin{proof}
		This comes from Lemma \ref{Lemma 3.4.3} and Corollary \ref{Cor 3.4.7}.
\end{proof}

The above estimate on $\mathcal{L}_{ib,\pi}$ implies the following estimate on $\mathcal{L}_{s,\pi}$.

\begin{Corollary}\label{Corollary 3.3.5}
There exist $C_{26} > 0$, $C_{27} > 0$ such that for any $1 \neq \pi \in \widehat{G}$, any $s = a + ib$ with $|a| \le b_{\pi}^{-C_{26}}$, and any $h \in F_{\lambda}(X^{+}, H_{\pi})$, we have $||\mathcal{L}_{s,\pi}^{n_{b_{\pi}}}h||_{b_{\pi}} \le (1 - b_{\pi}^{-C_{27}}) ||h||_{b_{\pi}}$.
\end{Corollary}
\begin{proof}
	By Proposition \ref{Dolgopyat type estimate for rapid mixing of compact group extension}, we have
	\begin{equation}\label{3.3.3}
		||\mathcal{L}_{s,\pi}^{n_{b_{\pi}}}h||_{b_{\pi}}=||\mathcal{L}_{ib,\pi}^{n_{b_{\pi}}}(e^{-ar_{n_{b_{\pi}}}}h)||_{b_{\pi}}\le (1-b_{\pi}^{-C_{23}})||e^{-ar_{n_{b_{\pi}}}}h||_{b_{\pi}}.
	\end{equation}
	On the one hand, when $|a|\le b_{\pi}^{-C_{26}}$,
	\begin{equation}\label{3.3.4}
		||e^{-ar_{n_{b_{\pi}}}}h||_{\infty} \le (1+2b_{\pi}^{-C_{26}+1})||h||_{\infty}.
	\end{equation}
	On the other hand, also provided $|a|\le b_{\pi}^{-C_{26}}$,
	\begin{equation}\label{3.3.5}
		|e^{-ar_{n_{b_{\pi}}}}h|_{\text{Lip}}\le (1+2b_{\pi}^{-C_{26}+1})(b_{\pi}^{-C_{26}+1+2\log\lambda^{-1}}||h||_{\infty}+|h|_{\text{Lip}}).
	\end{equation}
	Substituting \eqref{3.3.4} and \eqref{3.3.5} into \eqref{3.3.3} and choosing $C_{26}>0$ large, we obtain
	\begin{equation*}
		||\mathcal{L}_{s,\pi}^{n_{b_{\pi}}}h||_{b_{\pi}}\le(1-b_{\pi}^{-C_{23}})(1+2b_{\pi}^{-C_{26}+1}) (1+b_{\pi}^{-C_{26}+1+2\log\lambda^{-1}})||h||_{b_{\pi}}\le (1-b_{\pi}^{-C_{27}})||h||_{b_{\pi}},
	\end{equation*}
	for some uniform constant $C_{27} > 0$.
\end{proof}

\subsection{Proof of  Theorem \ref{Rapid mixing for symbolic flow of compact group extension}}\label{subsec 7.4}

With the results we have  assembled we are at last in a position to prove Theorem \ref{Rapid mixing for symbolic flow of compact group extension} (although  the proofs of Lemmas \ref{Lemma 3.3.3} and \ref{Lemma 3.4.4} still have to be given in \S\S \ref{subsec 7.5} and \ref{Appendix A}).

	The analytic continuation of $\widehat{\chi}_{\pi}$ is given by $(1)$ of Lemma \ref{Lemma 3.3.3}. By changing the  summation order, we obtain
	\begin{equation}\label{3.3.6}
		\Phi_{3}(s)=\sum_{p\ge1}\sum_{q\ge0}\sum_{n\ge\max\{p-q,1\}}
			\langle
			[M_{\pi,n+q-p}(s)E_{\pi}], \mathcal{L}_{\overline{s},\pi}^{p}[M_{\pi,q}(-\overline{s})F_{\pi}]
			\rangle
		_{\mu\times\text{Haar}_{G}}.
	\end{equation}
	Similarly, we can also write,
	\begin{equation}\label{3.3.7}
		\Phi_{4}(s)=\sum_{p\ge0}\sum_{q\ge0}\sum_{n\ge1}
			\langle
			\mathcal{L}_{s,\pi}^{p}[M_{\pi,p+n+q}(s)E_{\pi}], [M_{\pi,q}(-\overline{s})F_{\pi}]
		\rangle_{\mu\times\text{Haar}_{G}}.
	\end{equation}
	For any $s=a+ib$ with $|a|\le b_{\pi}^{-C_{26}}$, using 
	Corollaries 
	\ref{Cor 3.4.2} and  \ref{Corollary 3.3.5}, 
	and $(2)$ of Lemma \ref{Lemma 3.3.3}, we can bound \eqref{3.3.6} as follows:
	\begin{equation}\label{3.3.8}
		\begin{aligned}
			|\Phi_{3}(s)| &\le \sum_{p\ge0}\sum_{j=0}^{n_{b_{\pi}}-1}\sum_{q\ge0}\sum_{n\ge\max\{pn_{b_{\pi}}+j-q,1\}} \\
			&\quad ||M_{\pi,n+q-pn_{b_{\pi}}-j}(s)E_{\pi}||_{\infty}2(1-b_{\pi}^{-C_{27}})^{p} ||M_{\pi,q}(-\overline{s})F_{\pi}||_{b_{\pi}} \\
			&\le \sum_{p\ge0}\sum_{q\ge0}\sum_{n\ge0}2n_{b_{\pi}}C_{14}\lambda^{3n/4}||E_{\pi}||_{\lambda}(1-b_{\pi}^{-C_{27}})^{p}C_{14}\lambda^{q/4}||F_{\pi}||_{\lambda} \\
			&\le C_{28}b_{\pi}^{C_{27}+1}||E_{\pi}||_{\lambda}||F_{\pi}||_{\lambda},
		\end{aligned}
	\end{equation}
	where $C_{28}>0$ is a uniform constant. In particular, $\Phi_{3}$ is analytic in the region $s=a+ib$ with $|a|\le b_{\pi}^{-C_{26}}$. Similarly, as shown in \eqref{3.3.7}, one can infer that $\Phi_{4}$ is analytic in the region $s=a+ib$ with $|a|\le b_{\pi}^{-C_{26}}$, and it satisfies
	\begin{equation}\label{3.3.9}
		|\Phi_{4}(s)|\le C_{28}b_{\pi}^{C_{27}+1}||E_{\pi}||_{\lambda}||F_{\pi}||_{\lambda}.
	\end{equation}
	The analysis of $\Phi_{1}$ and $\Phi_{2}$ is straightforward. By their definitions, $\Phi_{1}$ and $\Phi_{2}$ are analytic on $\mathbb{C}$. Moreover, within the region $s=a+ib$ where $|a|\le b_{\pi}^{-C_{26}}$, as indicated by \eqref{3.3.1} and \eqref{3.3.2}, we can bound $|\Phi_{1}|+|\Phi_{2}|\le C_{28}||E_{\pi}||_{\infty}||F_{\pi}||_{\infty}$. Combining this with \eqref{3.3.8} and \eqref{3.3.9}, we conclude that $\widehat{\chi}_{\pi}$ is analytic in the region $s=a+ib$ with $|a|\le b_{\pi}^{-C_{26}}$, and it satisfies
	\begin{equation}\label{3.3.10}
		|\widehat{\chi}_{\pi}(s)|\le\sum_{i=1}^{4}|\Phi_{i}(s)|\le 3C_{28}b_{\pi}^{C_{27}+1}||E_{\pi}||_{\lambda}||F_{\pi}||_{\lambda}.
	\end{equation}
	
	The above bound is excessively large, rendering it impractical for the use of the inverse Laplace transform. To overcome this complication,  by analogy with  part VI in \cite[Section 4]{Dol98a}, we need to differentiate $\chi_{\pi}$. Recall that $E$ and $F$ belong to $F_{\lambda}^{k}(\widehat{X}_{r})$, and $E_{\pi}$ and $F_{\pi}$ are $C^{k_{2}}$ with respect to $u$, where $k=2k_{1}+k_{2}$. Previously, we used $\chi_{\pi}$ to represent $\chi_{E_{\pi},F_{\pi}}$. For any $s=a+ib$ with $a>0$ and any $\pi\in\widehat{G}$, using \eqref{3.3.1} and \eqref{3.3.2}, we have
	$$
	\begin{aligned}
		&\widehat{\chi}_{E_{\pi},F_{\pi}}(s)\\
		=&\int_{0}^{\infty}e^{-st}\chi_{E_{\pi},F_{\pi}}(t)dt\\
		=&\sum_{j=1}^{k_{2}}\dfrac{1}{s^{j}}e^{-st}\chi_{E_{\pi},F_{\pi}}^{(j-1)}(t)\bigg|_{0}^{\infty}+\dfrac{1}{s^{k_{2}}}\int_{0}^{\infty}e^{-st}\chi_{E_{\pi},F_{\pi}}^{(k_{2})}(t)dt\\
		=&\dfrac{1}{s^{k_{2}}}\int_{0}^{\infty}e^{-st}\rho_{E_{\pi},F_{\pi}}^{(k_{2})}(t)dt+\dfrac{1}{s^{k_{2}}}\int_{0}^{\infty}e^{-st}(\chi_{E_{\pi},F_{\pi}}^{(k_{2})}-\rho_{E_{\pi},F_{\pi}}^{(k_{2})})(t)dt\\
		=&\dfrac{1}{s^{k_{2}}}\int_{0}^{\infty}e^{-st}\rho_{\frac{\partial^{k_{2}}E_{\pi}}{\partial u^{k_{2}}},F_{\pi}}(t)dt+\dfrac{1}{s^{k_{2}}}\int_{0}^{1}e^{-st}(\chi_{E_{\pi},F_{\pi}}^{(k_{2})}-\rho_{E_{\pi},F_{\pi}}^{(k_{2})})(t)dt\\
		=&\dfrac{1}{s^{k_{2}}}\int_{0}^{\infty}e^{-st}\chi_{\frac{\partial^{k_{2}}E_{\pi}}{\partial u^{k_{2}}},E_{\pi}}(t)dt+\dfrac{1}{s^{k_{2}}}\int_{0}^{\infty}e^{-st}\bigg(\rho_{\frac{\partial^{k_{2}}E_{\pi}}{\partial u^{k_{2}}},F_{\pi}}-\chi_{\frac{\partial^{k_{2}}E_{\pi}}{\partial u^{k_{2}}},F_{\pi}}\bigg)(t)dt\\
		&+\dfrac{1}{s^{k_{2}}}\int_{0}^{1}e^{-st}(\chi_{E_{\pi},F_{\pi}}^{(k_{2})}-\rho_{E_{\pi},F_{\pi}}^{(k_{2})})(t)dt\\
		=&\dfrac{1}{s^{k_{2}}}\widehat{\chi}_{\frac{\partial^{k_{2}}E_{\pi}}{\partial u^{k_{2}}},F_{\pi}}(s)+\dfrac{1}{s^{k_{2}}}\int_{0}^{1}e^{-st}\bigg(\rho_{\frac{\partial^{k_{2}}E_{\pi}}{\partial u^{k_{2}}},F_{\pi}}-\chi_{\frac{\partial^{k_{2}}E_{\pi}}{\partial u^{k_{2}}},F_{\pi}}\bigg)(t)dt\\
		&+\dfrac{1}{s^{k_{2}}}\int_{0}^{1}e^{-st}(\chi_{E_{\pi},F_{\pi}}^{(k_{2})}-\rho_{E_{\pi},F_{\pi}}^{(k_{2})})(t)dt\\
		=&\dfrac{1}{s^{k_{2}}}\widehat{\chi}_{\frac{\partial^{k_{2}}E_{\pi}}{\partial u^{k_{2}}},F_{\pi}}(s)+\dfrac{1}{s^{k_{2}}}\int_{0}^{1}e^{-st}\bigg(\chi_{E_{\pi},F_{\pi}}^{(k_{2})}-\chi_{\frac{\partial^{k_{2}}E_{\pi}}{\partial u^{k_{2}}},F_{\pi}}\bigg)(t)dt.
	\end{aligned}
	$$
	Therefore, according to \eqref{3.3.2}, \eqref{3.3.10}, and the condition of analyticity in the region $s=a+ib$ with $|a|\le b_{\pi}^{-C_{26}}$, for any $s\neq 0$ we have
	$$|\widehat{\chi}_{E_{\pi},F_{\pi}}(s)|\le 3C_{28}\dfrac{b_{\pi}^{C_{27}+1}}{|s|^{k_{2}}}\bigg|\bigg|\frac{\partial^{k_{2}}E_{\pi}}{\partial u^{k_{2}}}\bigg|\bigg|_{\lambda}||F_{\pi}||_{\lambda}+\dfrac{2^{k_{2}}}{|s|^{k_{2}}}\bigg(\bigg|\bigg|\frac{\partial^{k_{2}}E_{\pi}}{\partial u^{k_{2}}}\bigg|\bigg|_{\infty}+||E_{\pi}||_{\infty}\bigg)||F_{\pi}||_{\infty}.$$
	Recalling that $b_{\pi}=|b|+C_{15}|\lambda_{\pi}|^{1+m_{G}/2}$, and using Lemma \ref{Lemma 3.3.1} in conjunction with the above bound, in the region $s=a+ib$ with $|a|\le b_{\pi}^{-C_{26}}$ and $|b|\ge1$, we obtain
	\begin{equation}\label{3.3.11}
		|\widehat{\chi}_{E_{\pi},F_{\pi}}(s)|\le C_{29}|b|^{-k_{2}+C_{27}+1}|\lambda_{\pi}|^{-4k_{1}+2(C_{27}+1)m_{G}}||E||_{\lambda,k}||F||_{\lambda,k},
	\end{equation}
	and by \eqref{3.3.10}, in the case of $|b|\le1$,
	\begin{equation}\label{3.3.12}
		|\widehat{\chi}_{E_{\pi},F_{\pi}}(s)|\le C_{29}|\lambda_{\pi}|^{-4k_{1}+2(C_{27}+1)m_{G}}||E||_{\lambda,k}||F||_{\lambda,k},
	\end{equation}
	where $C_{29}>0$ is a uniform constant.

	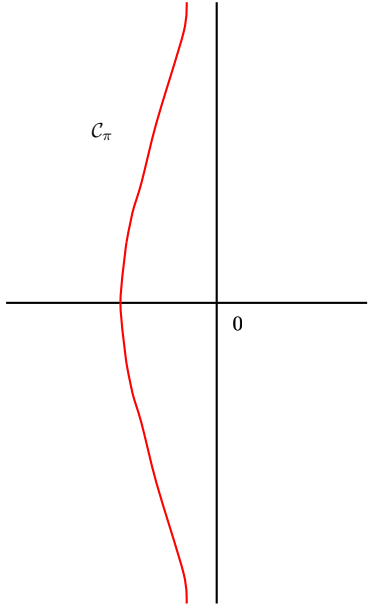
\begin{figure}[h!]
		\centerline{
			\begin{tikzpicture}[thick,scale=0.40, every node/.style={scale=0.7}]
				\draw[-, black] (0,-10)-- (0,10);
				\draw[-, black] (-7,0)-- (5,0);
				\draw[red] plot [smooth] coordinates {(-3,2) (-2.8, 3) (-2.5, 4) (-2, 6)  (-1.1,9)  (-1,10)}; 
				\node at (0.7, -0.7) {$0$};
				\draw[red] plot [smooth] coordinates {(-3,-2) (-2.8, -3) (-2.5,-4)  (-2,-6) (-1.1,-9)  (-1,-10)}; 
				\node at (0.7, -0.7) {$0$};
				\draw[red] plot [smooth] coordinates {(-3,2) (-3.2,0) (-3,-2)}; 
				\node at (-3.8, 5.7) {$\mathcal{C}_\pi$};
			\end{tikzpicture}
		}
		\caption{The curve $\mathcal{C}_\pi$}
	\end{figure}
	

	Now, we apply the inverse Laplace formula to the curve $\mathcal{C}_{\pi}=\{a= -b_{\pi}^{-C_{26}}: b\in\mathbb{R}\}$,  to obtain $\chi_{E_{\pi},F_{\pi}}(t)=\int_{\mathcal{C}_{\pi}}e^{st}\widehat{\chi}_{E_{\pi},F_{\pi}}(s)ds$. For any $n\ge1$, employing the bounds \eqref{3.3.11} and \eqref{3.3.12}, we can choose $k_{1},k_{2}$ to be large, and consequently $k$ to be large such that for some uniform constant $C_{30}>0$, the following holds: $\chi_{E_{\pi},F_{\pi}}(t)\le C_{30}|\lambda_{\pi}|^{-r_{G}-1}||E||_{\lambda,k}||F||_{\lambda,k}t^{-n}$ for any $t>0$ and any $1\not=\pi\in\widehat{G}$. Then, by \eqref{3.3.1} and \eqref{3.3.2}, we have $\rho_{E_{\pi},F_{\pi}}(t)\le (C_{30}+1)|\lambda_{\pi}|^{-r_{G}-1}||E||_{\lambda,k}||F||_{\lambda,k}t^{-n}$ for any $t>0$ and any $1\not=\pi\in\widehat{G}$. By Lemmas \ref{Lemma 2.7.5} and \ref{Lemma 3.3.2}, we obtain $\rho_{E, F}(t)\le \rho_{E_{1},F_{1}}(t)+C_{31}||E||_{\lambda,k}||F||_{\lambda,k}t^{-n}$ for any $t>0$, where $C_{31}=(C_{30}+1)\sum_{\pi\not=1}|\lambda_{\pi}|^{-r_{G}-1}$. The function $\rho_{E_{1},F_{1}}$ is the correlation function for the suspension flow $\psi_{t}$ of the hyperbolic flow $g_{t}$. Since $g_{t}$ is Diophantine or a jointly non-integrable Anosov flow, it is proved in \cite{Dol98b} that $\rho_{E_{1},F_{1}}$ has rapid decay.\footnote{The assumption of \cite{Dol98b} requires two closed orbits with a specific period ratio known as a bad Diophantine number, and this implies a estimate of transfer operators which shows $\rho_{E_{1},F_{1}}$ has rapid decay. Note that, using the same method, it is proved in \cite{Pol01} that one can also choose three closed orbits with the associated point is a bad Diophantine number, i.e. $g_{t}$ is Diophantine, to achieve the same effect by proving the same estimate of transfer operators.} Therefore, the above bound indicates that $\rho_{E, F}(t)$ also has rapid decay, completing the proof of Theorem \ref{Rapid mixing for symbolic flow of compact group extension}.

\subsection{Proof of Lemma \ref{Lemma 3.4.4}}\label{subsec 7.5}

We begin with some simplifying observations.
The first point to observe  is that achieving cancellation at a single point is sufficient. Indeed, according to Lemma \ref{Lemma 3.4.1}, the estimate $|h|_{\text{Lip}} \le 2C_{18}b_{\pi} ||h||_{\infty}$ implies $|\mathcal{L}_{ib,\pi}^{C_{21}\log b_{\pi}}h|_{\text{Lip}} \le 2C_{18}b_{\pi} ||h||_{\infty}$. Therefore, if $||\mathcal{L}_{ib,\pi}^{C_{21}\log b_{\pi}}h(x)|| \le (1-b_{\pi}^{-C_{20}})||h||_{\infty}$, then for any $y \in X^{+}$ with $d_{\lambda}(x,y) \le (4C_{18}b_{\pi}^{C_{20}+1})^{-1}$, we have
\[
\begin{aligned}
||\mathcal{L}_{ib,\pi}^{C_{21}\log b_{\pi}}h(y)|| &\le ||\mathcal{L}_{ib,\pi}^{C_{21}\log b_{\pi}}h(y) - \mathcal{L}_{ib,\pi}^{C_{21}\log b_{\pi}}h(x)|| + (1-b_{\pi}^{-C_{20}})||h||_{\infty}\\
&\le  \dfrac{1}{2}b_{\pi}^{-C_{20}}||h||_{\infty} + (1-b_{\pi}^{-C_{20}})||h||_{\infty} \\
&= (1-2^{-1}b_{\pi}^{-C_{20}})||h||_{\infty}.
\end{aligned}
\]
Using the Gibbs property of $\mu$ (in Lemma \ref{Gibbs property}),  it is  easy to show that the set of these points $y$ has $\mu$-measure $\ge b_{\pi}^{-C_{19}}$ for some uniform constant $C_{19} > 0$. Then  Lemma \ref{Lemma 3.4.4}  would follow.

The second point we want to note is that we can always assume $h$ satisfies $||h(x)|| \ge ||h||_{\infty}3/4$ for any $x \in X^{+}$. Indeed, if not, then for some point $x \in X^{+}$, we would have $||h(x)|| \le 3||h||_{\infty}/4$. Consequently, we would obtain
\[
\begin{aligned}
||\mathcal{L}_{ib,\pi}^{C_{21}\log b_{\pi}}h(\sigma^{C_{21}\log b_{\pi}}(x))||
&\le  \sum_{\substack{\sigma^{C_{21}\log b_{\pi}}(y)=\sigma^{C_{21}\log b_{\pi}}(x)\\y\not=x}}e^{\varphi_{C_{21}\log b_{\pi}}(y)}||h||_{\infty} + e^{\varphi_{C_{21}\log b_{\pi}}(x)}3||h||_{\infty}/4\\
&\le  (1-e^{-C_{21}\log b_{\pi}|\varphi|_{\infty}}/4)||h||_{\infty} \le (1-b_{\pi}^{-C_{20}})||h||_{\infty}.
\end{aligned}
\]
Then the required result would follow  from the argument presented in the previous paragraph.

We will  use the following elementary  inequality to cancel the terms in the  transfer operator.

\begin{Lemma}\label{Lemma 3.5.1}
Let $H$ be a Hilbert space, and $0\not=v_{1},v_{2}\in H$. If $||\frac{v_{1}}{||v_{1}||}-\frac{v_{2}}{||v_{2}||}||\ge\varepsilon$ and $||v_{1}||\le||v_{2}||$, then $||v_{1}+v_{2}||\le (1-\varepsilon^{2}/4)||v_{1}||+||v_{2}||$. 
\end{Lemma}

Denote $n_{b_{\pi}} = C_{21}\log b_{\pi}$ and $\eta = h/||h||$. Note that $||\eta|| \equiv 1$. Given that $|h|_{\text{Lip}} \le 2C_{18}b_{\pi}||h||_{\infty}$ and $||h(x)|| \ge 3||h||_{\infty}/4$ for any $x \in X^{+}$, it is straightforward to show that $|\eta|_{\text{Lip}} \le 2C_{18}b_{\pi}$. Our objective is to demonstrate the existence of $x \in X^{+}$ and $y_{1}, y_{2} \in \sigma^{-n_{b_{\pi}}}(x)$ such that
\begin{equation}\label{3.5.1}
	||e^{ibr_{n_{b_{\pi}}}(y_{1})}\pi(\Theta_{n_{b_{\pi}}}(y_{1}))\eta(y_{1}) - e^{ibr_{n_{b_{\pi}}}(y_{2})}\pi(\Theta_{n_{b_{\pi}}}(y_{2}))\eta(y_{2})|| \ge b_{\pi}^{-C_{32}},
\end{equation}
for some uniform constant $C_{32} > 0$. 
Then, by Lemma \ref{Lemma 3.5.1},
\[
\begin{aligned}
&||\mathcal{L}_{ib,\pi}^{n_{b_{\pi}}}h(x)||\\
\le & \sum_{\sigma^{n_{b_{\pi}}}(y)=x; y \neq y_{1},y_{2}}e^{\varphi_{n_{b_{\pi}}}(y)}||h(y)||+ ||e^{\varphi_{n_{b_{\pi}}}(y_{1})}e^{ibr_{n_{b_{\pi}}}(y_{1})}\pi(\Theta_{n_{b_{\pi}}}(y_{1}))h(y_{1})\\ &+e^{\varphi_{n_{b_{\pi}}}(y_{2})}e^{ibr_{n_{b_{\pi}}}(y_{2})}\pi(\Theta_{n_{b_{\pi}}}(y_{2}))h(y_{2})|| \\
\le & \sum_{\sigma^{n_{b_{\pi}}}(y)=x; y \neq y_{1},y_{2}}e^{\varphi_{n_{b_{\pi}}}(y)}||h(y)|| + (1-b_{\pi}^{-2C_{32}}/4)e^{\varphi_{n_{b_{\pi}}}(y_{1})}||h(y_{1})||\\
 &+e^{\varphi_{n_{b_{\pi}}}(y_{2})}||h(y_{2})|| \\
\le & (1 - e^{-n_{b_{\pi}}|\varphi|_{\infty}}b_{\pi}^{-2C_{32}})||h||_{\infty} \le (1-b_{\pi}^{-C_{20}})||h||_{\infty},
\end{aligned}
\]
for some uniform constant $C_{20} > 0$. Then Lemma \ref{Lemma 3.4.4} would follow from the argument presented at the beginning of this subsection.

By the assumption of Theorem \ref{Rapid mixing for symbolic flow of compact group extension}, there exist $w \in X$, $n_{0}\in \mathbb{N}$, and $p_{0}\in \mathbb{N}^{+}$ such that the subset $\Gamma(w,n_{0},p_{0})$ of $G$ is Diophantine. Let $x_{w} = \Pi_{+}\sigma^{-(n_{b_{\pi}}-n_{0})}(w) \in X^{+}$ and 
$$I_{w} =e^{ibr_{n_{b_{\pi}}-n_{0}}(x_{w})}\pi(\Theta_{n_{b_{\pi}}-n_{0}}(x_{w}))\eta(x_{w}) \in H_{\pi}.$$ 
Note that $||I_{w}||=1$. By the definition of $\Gamma(w,n_{0},p_{0})$ being Diophantine (Definition \ref{Definition of Diophantine condition on Brin group}), there exists $g_{W} \in \Gamma(w,n_{0},p_{0})$ with a closed chain $W=\{w_{j}\}_{j=0}^{p}$ at $w$ with $p \le p_{0}$, such that 
\begin{equation}\label{3.5.2}
||I_{w} - \pi(g_{W})I_{w}|| \ge \dfrac{\delta_{3}}{|\lambda_{\pi}|^{C_{33}}},
\end{equation}
for some uniform constants $\delta_{3}>0$ and $C_{33}>0$.
For this closed chain $W=\{w_{j}\}_{j=0}^{p}$, for each $0 \le j \le p$, we define $x_{j} = \Pi_{+}\sigma^{-(n_{b_{\pi}}-n_{0})}(w_{j}) \in X^{+}$ and
$$I_{j}=e^{ibr_{n_{b_{\pi}}-n_{0}}(x_{j})}\pi(\Theta_{n_{b_{\pi}}-n_{0}}(x_{j}))\eta(x_{j}) \in H_{\pi}.$$
Note that, as $W$ is a closed chain at $w$, we have $x_{p} = x_{0} = x_{w}$ and $I_{p} = I_{0} = I_{w}$. By Definition \ref{Definition of Brin group in the symbolic level}, we have $\sum_{j=0}^{p-1}\Delta^{\delta_{w_{j},w_{j+1}}}(w_{j},w_{j+1})=0$ and $g_{W}=\prod_{j=0}^{p-1}\Theta^{\delta_{w_{j},w_{j+1}}}(w_{j},w_{j+1})$,\footnote{Recalling, $\prod_{j=0}^{p-1}\Theta^{\delta_{w_{j},w_{j+1}}}(w_{j},w_{j+1})=\Theta^{\delta_{w_{p-1},w_{p}}}(w_{p-1},w_{p})\cdots\Theta^{\delta_{w_{0},w_{1}}}(w_{0},w_{1})$, and $\delta_{x,y}=s$ if $y\in W^{s}(x)$; otherwise, $\delta_{x,y}=u$.}. Using this, we can express
$$
\begin{aligned}
& I_{w} - \pi(g_{W})I_{w}\\
=&\sum_{j=0}^{p-1}e^{-ib\sum_{n=1}^{p-1-j}\Delta^{\delta_{w_{j+n},w_{j+n+1}}}(w_{j+n},w_{j+1+n})}\pi\bigg(\prod_{n=1}^{p-1-j}\Theta^{\delta_{w_{j+n},w_{j+n+1}}}(w_{j+n},w_{j+1+n})\bigg)\\
& (I_{j+1} - e^{-ib\Delta^{\delta_{w_{j},w_{j+1}}}(w_{j},w_{j+1})}\pi(\Theta^{\delta_{w_{j},w_{j+1}}}(w_{j},w_{j+1}))I_{j}).
\end{aligned}
$$
Then, by \eqref{3.5.2} and $p\le p_{0}$, there exists $0\le j\le p-1$ such that
\begin{equation}\label{3.5.3}
\begin{aligned}
&||I_{j+1}-e^{-ib\Delta^{\delta_{w_{j},w_{j+1}}}(w_{j},w_{j+1})}\pi(\Theta^{\delta_{w_{j},w_{j+1}}}(w_{j},w_{j+1}))I_{j}||\\
\ge&\dfrac{\delta_{3}}{p_{0}|\lambda_{\pi}|^{C_{33}}}\ge\dfrac{\delta_{3}}{p_{0}b_{\pi}^{C_{33}/(1+m_{G}/2)}}.
\end{aligned}
\end{equation}

\begin{Claim}
We must have $w_{j+1}\in W^{s}_{n_{0}}(w_{j})$, i.e. $w_{j+1}\not\in W^{u}_{n_{0}}(w_{j})$.
\end{Claim}
\begin{proof}[Proof of Claim]
If $w_{j+1}\in W^{u}_{n_{0}}(w_{j})$, then $(x_{j+1})_{i}=(x_{j})_{i}$ for any $i\le n_{b_{\pi}}-2n_{0}$. Since $|\eta|_{\text{Lip}}\le 2C_{18}b_{\pi}$, we have
\begin{equation}\label{3.5.4}
||\eta(x_{j+1})-\eta(x_{j})||\le 2C_{18}b_{\pi}\lambda^{n_{b_{\pi}}-2n_{0}}\le \dfrac{1}{4}\dfrac{\delta_{3}}{p_{0}b_{\pi}^{C_{33}/(1+m_{G}/2)}}
\end{equation}
provided $C_{21}>0$ is large enough where $n_{b_{\pi}}=C_{21}\log b_{\pi}$. Meanwhile, by Lemmas \ref{Lemma 6.2.2} and \ref{Lemma 6.3.4}, 
\begin{equation}\label{3.5.5}
\begin{aligned}
&|\Delta^{u}(w_{j},w_{j+1})-[r_{n_{b_{\pi}}-n_{0}}(x_{j})-r_{n_{b_{\pi}}-n_{0}}(x_{j+1})]|\\
=&|\Delta^{u}(w_{j},w_{j+1})-[r_{n_{b_{\pi}}-n_{0}}(\sigma^{-n_{b_{\pi}}+n_{0}}(w_{j}))-r_{n_{b_{\pi}}-n_{0}}(\sigma^{-n_{b_{\pi}}+n_{0}}(w_{j+1}))]|\\
\le &C_{10}\lambda^{n_{b_{\pi}}-n_{0}}\le\dfrac{1}{4}\dfrac{\delta_{3}}{p_{0}b_{\pi}^{C_{33}/(1+m_{G}/2)+1}},
\end{aligned}
\end{equation}
and
\begin{equation}\label{3.5.6}
\begin{aligned}
&d_{G}(\Theta^{u}(w_{j},w_{j+1}),\Theta_{n_{b_{\pi}}-n_{0}}(x_{j+1})\Theta_{n_{b_{\pi}}-n_{0}}(x_{j})^{-1})\\
=&d_{G}(\Theta^{u}(w_{j},w_{j+1}),\Theta_{n_{b_{\pi}}-n_{0}}(\sigma^{-n_{b_{\pi}}+n_{0}}(w_{j+1}))\Theta_{n_{b_{\pi}}-n_{0}}(\sigma^{-n_{b_{\pi}}+n_{0}}(w_{j}))^{-1})\\
\le&C_{12}\lambda^{n_{b_{\pi}}-n_{0}}\le\dfrac{1}{4}\dfrac{\delta_{3}}{p_{0}b_{\pi}^{C_{33}/(1+m_{G}/2)+1}},
\end{aligned}	\end{equation}
provided $C_{21}>0$ is large enough.

Now, we can relate the bounds \eqref{3.5.5} and \eqref{3.5.6} to the estimate below.
$$
\begin{aligned}
&||\eta(x_{j+1})-\eta(x_{j})||\\
=&||e^{-ibr_{n_{b_{\pi}}-n_{0}}(x_{j+1})}\pi(\Theta_{n_{b_{\pi}}-n_{0}}(x_{j+1})^{-1})I_{j+1}-e^{-ibr_{n_{b_{\pi}}-n_{0}}(x_{j})}\pi(\Theta_{n_{b_{\pi}}-n_{0}}(x_{j})^{-1})I_{j}||\\
=&||I_{j+1}-e^{-ib[r_{n_{b_{\pi}}-n_{0}}(x_{j})-r_{n_{b_{\pi}}-n_{0}}(x_{j+1})]}\pi(\Theta_{n_{b_{\pi}}-n_{0}}(x_{j+1})\Theta_{n_{b_{\pi}}-n_{0}}(x_{j})^{-1})I_{j}||\\
\ge&||I_{j+1}-e^{-ib\Delta^{u}(w_{j},w_{j+1})}\pi(\Theta^{u}(w_{j},w_{j+1}))I_{j}||\\
&-C_{15}|\lambda_{\pi}|^{1+m_{G}/2}d_{G}(\Theta^{u}(w_{j},w_{j+1}),\Theta_{n_{b_{\pi}}-n_{0}}(x_{j+1})\Theta_{n_{b_{\pi}}-n_{0}}(x_{j})^{-1})\\
&-|b||\Delta^{u}(w_{j},w_{j+1})-[r_{n_{b_{\pi}}-n_{0}}(x_{j})-r_{n_{b_{\pi}}-n_{0}}(x_{j+1})]|\\
\ge&||I_{j+1}-e^{-ib\Delta^{u}(w_{j},w_{j+1})}\pi(\Theta^{u}(w_{j},w_{j+1}))I_{j}||-\dfrac{1}{4}\dfrac{\delta_{3}}{p_{0}b_{\pi}^{C_{33}/(1+m_{G}/2)}}.
\end{aligned}
$$
Together with \eqref{3.5.4}, we have
$$||I_{j+1}-e^{-ib\Delta^{u}(w_{j},w_{j+1})}\pi(\Theta^{u}(w_{j},w_{j+1}))I_{j}||\le\dfrac{1}{2}\dfrac{\delta_{3}}{p_{0}b_{\pi}^{C_{33}/(1+m_{G}/2)}}$$
which contradicts \eqref{3.5.3} and thus completes the proof of the claim.
\end{proof}

By the aforementioned claim and estimate \eqref{3.5.3}, we obtain
\begin{equation}\label{3.5.7}
||I_{j+1}-e^{-ib\Delta^{s}(w_{j},w_{j+1})}\pi(\Theta^{s}(w_{j},w_{j+1}))I_{j}||\ge\dfrac{\delta_{3}}{p_{0}b_{\pi}^{C_{33}/(1+m_{G}/2)}}.
\end{equation}
Since $r$ and $\Theta$ depend only on future coordinates, by definition, we have $\Delta^{s}(w_{j},w_{j+1})=r_{n_{0}}(\Pi_{+}(w_{j+1}))-r_{n_{0}}(\Pi_{+}(w_{j}))$ and $\Theta^{s}(w_{j},w_{j+1})=\Theta_{n_{0}}(\Pi_{+}(w_{j+1}))^{-1}\Theta_{n_{0}}(\Pi_{+}(w_{j}))$. In particular,
$$
\begin{aligned}
&||I_{j+1}-e^{-ib\Delta^{s}(w_{j},w_{j+1})}\pi(\Theta^{s}(w_{j},w_{j+1}))I_{j}||\\
=&||I_{j+1}-e^{-ib[r_{n_{0}}(\Pi_{+}(w_{j+1}))-r_{n_{0}}(\Pi_{+}(w_{j}))]}\pi(\Theta_{n_{0}}(\Pi_{+}(w_{j+1}))^{-1}\Theta_{n_{0}}(\Pi_{+}(w_{j})))I_{j}||\\
=&||e^{ibr_{n_{b_{\pi}}}(x_{j+1})}\pi(\Theta_{n_{b_{\pi}}}(x_{j+1}))\eta(x_{j+1})-e^{ibr_{n_{b_{\pi}}}(x_{j})}\pi(\Theta_{n_{b_{\pi}}}(x_{j}))\eta(x_{j})||,
\end{aligned}
$$
which proves \eqref{3.5.1} by \eqref{3.5.7} since $\sigma^{n_{b_{\pi}}}(x_{j+1})=\sigma^{n_{b_{\pi}}}(x_{j})$, thus completing the proof of Lemma \ref{Lemma 3.4.4}.

\subsection{Proof of Lemma \ref{Lemma 3.3.3}}\label{Appendix A}

The proof is similar to that of \cite[Proposition 3]{Dol98a}. Recalling that $E, F\in F_{\lambda}^{k}(\widehat{X}_{r})$, we can regard $E:X_{r}\to C^{2k_{1}}(G)$ by defining $E(x,u)(\cdot)=E(x,\cdot,u)$. We can express $E(x,u)=\sum_{\pi\in\widehat{G}}E_{\pi}(x,u)$ and $F(x,u)=\sum_{\pi\in\widehat{G}}F_{\pi}(x,u)$, where $E_{\pi}(x,u)$ and $F_{\pi}(x,u)\in H_{\pi}$. By utilizing Lemma \ref{Lemma 3.3.1}, we can confirm that $E_{\pi}$ and $F_{\pi}\in F_{\lambda}(X_{r}, H_{\pi})$.

Let's denote
$$E_{\pi,s}(x):=\int_{0}^{r(x)}e^{-su}E_{\pi}(x,u)du\quad\text{and}\quad  F_{\pi,s}(x):=\int_{0}^{r(x)}e^{-su}F_{\pi}(x,u)du.$$
Clearly, $E_{\pi,s}$ and $F_{\pi,s}\in F_{\lambda}(X, H_{\pi})$. Through a straightforward calculation, we obtain the following estimate:
\begin{equation}\label{A1}
	||E_{\pi,s}||_{Lip}\le C_{34}e^{|a||r|_{\infty}}||E_{\pi}||_{\lambda}\quad\text{and}\quad||F_{\pi,s}||_{Lip}\le C_{34}e^{|a||r|_{\infty}}||F_{\pi}||_{\lambda}
\end{equation}
for some uniform constant $C_{34}>0$ where $s=a+ib$.

\begin{Lemma}\label{Lemma A.1}
For any $s=a+ib$ with $a>0$, we have
$$\widehat{\chi}_{\pi}(s)=\Phi_{1}(s)+\Phi_{2}(s)+\sum_{n=1}^{\infty}
\langle
e^{-sr_{n}}\pi(\Theta_{n})^{-1}E_{\pi,s}\circ\sigma^{n},F_{\pi,-\overline{s}}
\rangle_{\mu_{\varphi}\times\text{Haar}_{G}}$$
where $\Phi_{1}(s)=\int_{0}^{1}e^{-st}(\chi_{\pi}(t)-\rho_{E_{\pi},F_{\pi}}(t))dt$, $\Phi_{2}(s)=\int_{0}^{|r|_{\infty}}e^{-st}\Delta_{\pi}(t)dt$ and $	\Delta_{\pi}(t)=\int_{G}\int_{X}\int_{0}^{\max\{0,r-t\}}E_{\pi}\circ\phi_{t}\overline{F}_{\pi} dud\mu_{\varphi}dg$.
\end{Lemma}
\begin{proof}
By definition,
\begin{equation}\label{A2}
\begin{aligned}
&\widehat{\chi}_{\pi}(s)=\int_{0}^{\infty}e^{-st}\chi_{\pi}(t)dt\\
=&\int_{0}^{\infty}e^{-st}(\chi_{\pi}(t)-\rho_{E_{\pi},F_{\pi}}(t))dt+\int_{0}^{\infty}e^{-st}\rho_{E_{\pi},F_{\pi}}(t)dt\\
=&\int_{0}^{\infty}e^{-st}(\chi_{\pi}(t)-\rho_{E_{\pi},F_{\pi}}(t))dt+\int_{0}^{\infty}e^{-st}\Delta_{\pi}(t)dt\\
&+\int_{0}^{\infty}\int_{X}\int_{G}\int_{\max\{0,r(x)-t\}}^{r(x)}e^{-st}E_{\pi}(x,u+t)\overline{F}_{\pi}(x,u)dudgd\mu_{\varphi}(x)dt.
\end{aligned}
\end{equation}
By \eqref{3.3.2}, we have
\begin{equation}\label{A3}
\int_{0}^{\infty}e^{-st}(\chi_{\pi}(t)-\rho_{E_{\pi},F_{\pi}}(t))dt=\int_{0}^{1}e^{-st}(\chi_{\pi}(t)-\rho_{E_{\pi},F_{\pi}}(t))dt.
\end{equation}
By the definition of $\Delta_{\pi}(t)$, we have $\Delta_{\pi}(t)=0$ when $t\ge|r|_{\infty}$, and therefore
\begin{equation}\label{A4}
\int_{0}^{\infty}e^{-st}\Delta_{\pi}(t)dt=\int_{0}^{|r|_{\infty}}e^{-st}\Delta_{\pi}(t)dt.
\end{equation}
Finally, by $\pi=L|_{H_{\pi}}$ where $L$ is the left regular representation, we have
\begin{equation*}
		\begin{aligned}
			&\int_{0}^{\infty}\int_{X}\int_{G}\int_{\max\{0,r(x)-t\}}^{r(x)}e^{-st}E_{\pi}(x,u+t)\overline{F}_{\pi}(x,u)dudgd\mu_{\varphi}(x)dt\\
			=&\int_{X}\int_{G}\int_{0}^{r(x)}\int_{t\ge r(x)-u}e^{-st}E_{\pi}(x,u+t)\overline{F}_{\pi}(x,u)dtdudgd\mu_{\varphi}(x)\\
			=&\int_{X}\int_{G}\int_{t^{\prime}\ge r(x)}e^{-st^{\prime}}E_{\pi}(x,t^{\prime})dt^{\prime}\int_{0}^{r(x)}e^{su}\overline{F}_{\pi}(x,u)dudgd\mu_{\varphi}(x)\\
			=&\int_{X}\int_{G}\sum_{n=1}^{\infty}\bigg[\int_{r_{n}(x)}^{r_{n+1}(x)}e^{-st^{\prime}}E_{\pi}(x,t^{\prime})dt^{\prime}\bigg]\bigg[\int_{0}^{r(x)}e^{su}\overline{F}_{\pi}(x,u)du\bigg]dgd\mu_{\varphi}(x)\\
			=&\sum_{n=1}^{\infty}
			\langle
			e^{-sr_{n}}\pi(\Theta_{n})^{-1}E_{\pi,s}\circ\sigma^{n},F_{\pi,-\overline{s}}
			\rangle_{\mu_{\varphi}\times\text{Haar}_{G}}.
		\end{aligned}
\end{equation*}
Now we can substitute the above equality, \eqref{A3} and \eqref{A4} into \eqref{A2} to complete the proof.
\end{proof}

Next, we want to express the series on the right side of Lemma \ref{Lemma A.1} as a sum of integrals over the one-sided space $X^{+}$. To achieve this, we rely on the following variant   on a result by Sinai, also presented in \cite[Proposition 10]{Dol98a}. While the proof in \cite{Dol98a} is formulated for $h\in F_{\lambda}(X, \mathbb{C})$, it is straightforward to observe that the same reasoning applies to $h\in F_{\lambda}(X, H_{\pi})$ as well.

\begin{Lemma}\label{Lemma A.2}
For any $h\in F_{\lambda}(X, H_{\pi})$, we can write $h=\sum_{p\ge0}h_{p}$ with the following properties:
\begin{enumerate}
\item For each non-negative integer $p$, we have $h_{p}\in F_{p}(X, H_{\pi})$.
\item For each non-negative integer $p$, we have $||h_{p}||_{\infty}\leq ||h||_{Lip}\lambda^{p}$.
\end{enumerate}
\end{Lemma}

According to the lemma above, we can express
$$E_{\pi, s}=\sum_{p\ge0}E_{\pi,s,p}\quad\text{and}\quad F_{\pi,-\overline{s}}=\sum_{p\ge0}F_{\pi, -\overline{s},p}$$
where $E_{\pi,s,p}$ and $F_{\pi,-\overline{s},p}\in F_{q}(X, H_{\pi})$. Since $E_{\pi,s,p}\circ\sigma^{p}$ and $F_{\pi,-\overline{s},p}\circ\sigma^{p}\in F_{2p}(X^{+}, H_{\pi})\subset F_{\lambda^{1/4}}(X^{+}, H_{\pi})$, by $(2)$ of Lemma \ref{Lemma A.2}, we have
\begin{equation}\label{A5}
||E_{\pi,s,p}\circ\sigma^{p}||_{Lip}\le 3||E_{\pi,s}||_{Lip}\lambda^{p/2}\quad\text{and}\quad||F_{\pi,-\overline{s},p}\circ\sigma^{p}||_{Lip}\le 3||F_{\pi,-\overline{s}}||_{Lip}\lambda^{p/2}.
\end{equation}
Using the decompositions of $E_{\pi,s}$ and $F_{\pi,-\overline{s}}$, and since $\sigma^{*}\mu_{\varphi}=\mu_{\varphi}$, we obtain
$$
\begin{aligned}
&
\langle 
e^{-sr_{n}}\pi(\Theta_{n})^{-1}E_{\pi,s}\circ\sigma^{n}, F_{\pi,-\overline{s}}
\rangle_{\mu_{\varphi}\times\text{Haar}_{G}}\\
=&\sum_{q\ge0}\sum_{p\ge0}
\langle
e^{-sr_{n}}\pi(\Theta_{n})^{-1}E_{\pi,s,p}\circ\sigma^{n},F_{\pi,-\overline{s},q}
\rangle_{\mu_{\varphi}\times\text{Haar}_{G}}\\
=&\sum_{q\ge0}\sum_{p<n+q}
\langle
e^{-sr_{n}\circ\sigma^{q}}\pi(\Theta_{n})^{-1}\circ\sigma^{q}[E_{\pi,s,p}\circ\sigma^{p}]\circ\sigma^{n+q-p}, [F_{\pi,-\overline{s},q}\circ\sigma^{q}]
\rangle_{\mu_{\varphi}\times\text{Haar}_{G}}\\
&+\sum_{q\ge0}\sum_{p\ge n+q}
\langle
e^{-sr_{n}\circ\sigma^{p-n}}\pi(\Theta_{n})^{-1}\circ\sigma^{p-n}[E_{\pi,s,p}\circ\sigma^{p}], [F_{\pi,-\overline{s},q}\circ\sigma^{q}]\circ\sigma^{p-n-q}
\rangle_{\mu_{\varphi}\times\text{Haar}_{G}}\\
=&\sum_{q\ge0}\sum_{p<n+q}\\
&
\langle
e^{-sr_{n+q-p}}\pi(\Theta_{n+q-p})^{-1}[e^{-sr_{p}}\pi(\Theta_{p})^{-1}E_{\pi,s,p}\circ\sigma^{p}]\circ\sigma^{n+q-p}, [e^{\overline{s}r_{q}}\pi(\Theta_{q})^{-1}F_{\pi,-\overline{s},q}\circ\sigma^{q}]
\rangle_{\mu_{\varphi}\times\text{Haar}_{G}}\\
&+\sum_{q\ge0}\sum_{p\ge n+q}\\
&
\langle
e^{-sr_{p-n-q}}[e^{-sr_{p}}\pi(\Theta_{p})^{-1}E_{\pi,s,p}\circ\sigma^{p}], \pi(\Theta_{p-n-q})^{-1}[e^{\overline{s}r_{q}}\pi(\Theta_{q})^{-1}F_{\pi,-\overline{s},q}\circ\sigma^{q}]\circ\sigma^{p-n-q}
\rangle_{\mu_{\varphi}\times\text{Haar}_{G}}.
\end{aligned}
$$
Thus, we define an operator
$$M_{\pi,p}(s):F_{\lambda}(X_{r}, H_{\pi})\to F_{\lambda^{1/4}}(X^{+}, H_{\pi})
\hbox{ by }
 M_{\pi,p}(s)E_{\pi}=e^{-sr_{p}}\pi(\Theta_{p})^{-1}E_{\pi,s,p}\circ\sigma^{p}.$$
Since $\mathcal{L}_{\varphi}^{*}\mu=\mu$ and $\Pi_{+}^{*}\mu_{\varphi}=\mu$, the above identity gives us
$$
\begin{aligned}
&
\langle
e^{-sr_{n}}\pi(\Theta_{n})E_{\pi,s}\circ\sigma^{n}, F_{\pi,-\overline{s}}
\rangle_{\mu_{\varphi}\times\text{Haar}_{G}}\\
=&\sum_{q\ge0}\sum_{p<n+q}
\langle
[M_{\pi,p}(s)E_{\pi}], \mathcal{L}_{\overline{s},\pi}^{n+q-p}[M_{\pi,q}(-\overline{s})F_{\pi}]
\rangle_{\mu\times\text{Haar}_{G}}\\
&+\sum_{q\ge0}\sum_{p\ge n+q}
\langle
\mathcal{L}_{s,\pi}^{p-n-q}[M_{\pi,p}(s)E_{\pi}], [M_{\pi,q}(-\overline{s})F_{\pi}]\circ\sigma^{p-n-q}
\rangle_{\mu\times\text{Haar}_{G}}.
\end{aligned}
$$
We can substitute the above formula into Lemma \ref{Lemma A.1} to complete the proof of $(1)$ of Lemma \ref{Lemma 3.3.3}.

By definition, applying a basic triangle inequality yields
$$
\begin{aligned}
&||M_{\pi,p}(s)E_{\pi}(x)-M_{\pi,p}(s)E_{\pi}(y)||\\
\le&|e^{-sr_{p}(x)}-e^{-sr_{p}(y)}|\ ||E_{\pi,s,p}\circ\sigma^{p}||_{\infty}+e^{|a||r|_{\infty}n}||\pi(\Theta_{p}(x)-\pi(\Theta_{p}(y)))||\ ||E_{\pi,s,p}\circ\sigma^{p}||_{\infty}\\
&+e^{|a||r|_{\infty}n}||E_{\pi,s,p}\circ\sigma^{p}(x)-E_{\pi,s,p}\circ\sigma^{p}(y)||.
\end{aligned}
$$
Then, by \eqref{A1}, \eqref{A5}, and Corollary \ref{Corollary 2.7.4}, we obtain
$$||M_{\pi,p}(s)E_{\pi}||_{\infty}\le C_{35}\lambda^{p}e^{|a||r|_{\infty}(n+1)}||E_{\pi}||_{\infty}$$
and
$$|M_{\pi,p}(s)E_{\pi}|_{Lip}\le C_{35}(|s|+|\lambda_{\pi}|^{1+m_{G}/2})\lambda^{p/2}e^{|a||r|_{\infty}(n+1)}||E_{\pi}||_{\lambda},$$
for some uniform constant $C_{35}>0$ which completes the proof of $(2)$ of Lemma \ref{Lemma 3.3.3}.

\section{Superpolynomial equidistribution}\label{sec 8}

In this section, we prove Theorem \ref{Superpolynomial equidistribution for compact group extension of hyperbolic flow}. This is a second application of the Dolgopyat type estimate: Proposition \ref{Dolgopyat type estimate for rapid mixing of compact group extension} or Corollary \ref{Corollary 3.3.5}.

The proof follows a similar structure to that of \cite{Pol01} and \cite{Pol08}. Fix a Markov partition of the underlying hyperbolic flow \(g_{t}\). Let \(\psi_{t}\) be the symbolic flow of \(g_{t}\), and let \(\phi_{t}\) be the symbolic flow of \(f_{t}\), constructed in \S \ref{subsec 6.5}. Denote by \(\tau^{\prime}\) a closed orbit or prime closed orbit of \(\psi_{t}\). Similarly, we use \(\ell_{\tau^{\prime}}\) and \([\tau^{\prime}]\) to denote the period and the holonomy class of \(\tau^{\prime}\) respectively. Denote by \(V(T)\) the collection of prime closed orbits \(\tau^{\prime}\) of \(\psi_{t}\) with \(\ell_{\tau^{\prime}}\le T\). We have the following analogous result for \(\phi_{t}\).

\begin{Proposition}\label{Prop 3.6.1}
For any $n\ge1$ there exist $C_{36}, C_{37}>0$ such that for any $1\not=\pi\in\widehat{G}$ and any $T>0$,
$$\dfrac{1}{\# V(T)}\sum_{\tau^{\prime}\in V(T)}\xi_{\pi}([\tau^{\prime}])\le C_{36}|\lambda_{\pi}|^{C_{37}}T^{-n}$$
where $\xi_{\pi}$ is the character of $\pi$.
\end{Proposition}

We postpone the proof of this proposition until later in the section.  

It is not difficult to prove that Proposition \ref{Prop 3.6.1} implies Theorem \ref{Superpolynomial equidistribution for compact group extension of hyperbolic flow}. To this end, we need the following classical result of Bowen \cite{Bow73}.

\begin{Lemma}\label{Lemma 3.6.2}
For each prime closed orbit $\tau$ of $g_{t}$ which does not intersect the boundary $\partial R$, there exists a prime closed orbit $\tau^{\prime}$ of $\psi_{t}$ such that $\ell_{\tau}=\ell_{\tau^{\prime}}$ and $[\tau]=[\tau^{\prime}]$. 
Let $E(T)$, $T > 0$, denote the number of closed orbits $\tau$ which do intersect $\partial R$.  
Then there exist $C_{38}>0$ and $\delta_{2}>0$ such that for any $T>0$, 
$$|\#\pi(T)-\#V(T)|\le C_{38}e^{(h_{top}-\delta_{2})T} \hbox{ and } \#E(T) \le C_{38}e^{(h_{top}-\delta_{2})T}$$
where $h_{top}$ is the topological entropy of $g_{t}$.
\end{Lemma}

\begin{proof}[\textbf{Proof of Theorem \ref{Superpolynomial equidistribution for compact group extension of hyperbolic flow} assuming Proposition \ref{Prop 3.6.1}}]	
By Lemma \ref{Lemma 2.7.9}, for any \(F\in C^{2k}(G)\), we have \(F=\sum_{\pi\in\widehat{G}}f_{\pi}\xi_{\pi}\), and thus
\[
\dfrac{1}{\# V(T)}\sum_{\tau^{\prime}\in V(T)}F([\tau^{\prime}])-\int_{G}Fd(\text{\rm Haar})_{G}(g)=\sum_{\pi\not=1}\dfrac{1}{\# V(T)}\sum_{\tau^{\prime}\in V(T)}f_{\pi}\xi_{\pi}([\tau^{\prime}]).
\]
Then, by Proposition \ref{Prop 3.6.1}, 
\[
\bigg|\dfrac{1}{\# V(T)}\sum_{\tau^{\prime}\in V(T)}F([\tau^{\prime}])-\int_{G}Fd(\text{\rm Haar})_{G}(g)\bigg|\le \sum_{\pi\not=1}|f_{\pi}|C_{36}|\lambda_{\pi}|^{C_{37}}T^{-n}.
\]
By Lemma \ref{Lemma 2.7.9}, we then have
\[
\bigg|\dfrac{1}{\# V(T)}\sum_{\tau^{\prime}\in V(T)}F([\tau^{\prime}])-\int_{G}Fd(\text{\rm Haar})_{G}(g)\bigg|\le C_{3}C_{36}||F||_{C^{2k}}T^{-n}\sum_{\pi\not=1}|\lambda_{\pi}|^{-2k+2m_{G}+C_{37}}.
\]
If we choose \(k\) large such that \(2k-2m_{G}-C_{37}>r_{G}\), then by Lemma \ref{Lemma 2.7.5}, there exists \(C_{39}>0\) such that for any \(F\in C^{2k}(G)\),
\begin{equation}\label{3.6.1}
\bigg|\dfrac{1}{\# V(T)}\sum_{\tau^{\prime}\in V(T)}F([\tau^{\prime}])-\int_{G}Fd(\text{\rm Haar})_{G}(g)\bigg|\le C_{39}||F||_{C^{2k}}T^{-n},
\end{equation}
which is the superpolynomial equidistribution for \(\phi_{t}\).

Now, we can prove \eqref{3.6.1} holds for $f_{t}$ as well. To this end, by \eqref{3.6.1}, Lemma \ref{Lemma 3.6.2} and a basic triangle inequality, 
\begin{equation}\label{3.6.2}
\begin{aligned}
&\bigg|\dfrac{1}{\#\pi(T)}\sum_{\tau\in\pi(T)}F([\tau])-\dfrac{1}{\# V(T)}\sum_{\tau^{\prime}\in V(T)}F([\tau^{\prime}])\bigg|\\
\le&\dfrac{1}{\#\pi(T)}\bigg|\sum_{\tau\in\pi(T)}F([\tau])-\sum_{\tau^{\prime}\in V(T)}F([\tau^{\prime}])\bigg|+|F|_{\infty}\bigg|\dfrac{\# V(T)}{\#\pi(T)}-1\bigg|\\
\le&|F|_{\infty}\dfrac{\# E(T)}{\#\pi(T)}
 +|F|_{\infty}\bigg|\dfrac{\# V(T)}{\#\pi(T)}-1\bigg|\\
 \le&|F|_{\infty}C_{38}e^{-\delta_{2}T}
 +|F|_{\infty}\bigg|\dfrac{\# V(T)}{\#\pi(T)}-1\bigg|.
\end{aligned}
\end{equation}	
By Lemma \ref{Lemma 3.6.2} and the prime orbit theorem \cite{Par90}, we have
\begin{equation}\label{3.6.3}
	\bigg|\dfrac{\# V(T)}{\#\pi(T)}-1\bigg|\le C_{40}e^{-\delta_{2} T}T
\end{equation}
for some uniform constant $C_{40}>0$. Thus, we can combine
\eqref{3.6.2} and \eqref{3.6.3} to complete the proof.
\end{proof}

In the reminder of this section, we prove Proposition \ref{Prop 3.6.1}. Given $1\not=\pi\in\widehat{G}$, we define a $L$-function by 
$$L_{\pi}:\mathbb{C}\to\mathbb{C},\quad L_{\pi}(s)=\prod_{\text{prime }\tau^{\prime}}\det(I-\pi([\tau^{\prime}])e^{-sh_{top}\ell_{\tau^{\prime}}})^{-1}.$$
The $L$-function $L_{\pi}$ is analytic for $s=a+ib$ with $a>1$. The (prime) closed orbits $\tau^{\prime}$ of $\psi_{t}$ correspond one-one to (prime) periodic orbits $\{\sigma^{n}(x)=x\}$ of $\sigma:X^{+}\to X^{+}$. The period $\ell_{\tau^{\prime}}$ of $\tau^{\prime}$ is equals to $r_{n}(x)$, and the holonomy class $[\tau^{\prime}]$ of $\tau$ is equals to $[\Theta_{n}(x)]$ \cite{Par90}. Using this relationship, it is not difficult to obtain that for any $\pi\not=1$ and any $s=a+ib$ with $a>1$, 
\begin{equation}\label{3.6.4}
\begin{aligned}
L_{\pi}(s)=&\prod_{\text{prime }\tau^{\prime}}e^{\sum_{n=1}^{\infty}\frac{1}{n}\text{Tr}(\pi([\tau^{\prime}])^{n})e^{-sh_{top}n\ell_{\tau^{\prime}}}}\\
=&e^{\sum_{\text{prime }\tau^{\prime}}\sum_{n=1}^{\infty}\frac{1}{n}\xi_{\pi}([\tau^{\prime}]^{n})e^{-sh_{top}n\ell_{\tau^{\prime}}}}=e^{\sum_{n=1}^{\infty}\frac{1}{n}Z_{\pi,n}(s)}
\end{aligned}
\end{equation}
where $Z_{\pi,n}(s)=\sum_{\sigma^{n}(x)=x}\xi_{\pi}(\Theta_{n}(x))e^{-sh_{top}r_{n}(x)}$.

\begin{Lemma}\label{Lemma 3.6.3}
There exists $C_{41}>0$ such that for any $\pi\not=1$, we have $L_{\pi}$ is analytic and non-zero in the region of $s=a+ib$ with $a> 1-b_{\pi}^{-C_{41}}$ and $|b| \geq 1$. Furthermore, in this region, we have 
$$\bigg|\dfrac{L_{\pi}^{\prime}(s)}{L_{\pi}(s)}\bigg|\le C_{41}b_{\pi}^{C_{41}}.$$
\end{Lemma}
\begin{proof}
The proof is very similar to the case of zeta functions \cite{Pol01}. The analytical extension of $L_{\pi}$ is given by \eqref{3.6.4}. Then, using the 
estimate in  Corollary \ref{Corollary 3.3.5}, in this region, similarly to \cite[Proposition 3]{Pol01}, one can bound
$$|Z_{\pi,n}(s)|\le C_{42}b_{\pi}(1-b_{\pi}^{C_{42}})^{n}$$
for some uniform constants $C_{42}>0$ which implies the first part. We can then derive the bound on the logarithmic derivative of $L_{\pi}$, using an analogous method to that used for zeta functions \cite[Proposition 4]{Pol01}.
\end{proof}

Recalling, $\ell_{\tau^{\prime}}$ is the period of $\tau^{\prime}$. We will use $\Lambda_{\tau^{\prime}}$ to denote the least period of $\tau^{\prime}$. For any $k\ge1$, we can employ Lemma \ref{Lemma 3.6.3} to obtain an estimate of the following function
$$N_{\pi,k}:[e,\infty)\to\mathbb{C},\quad N_{\pi,k}(x)=\sum_{e^{h_{top}\ell_{\tau^{\prime}}}\le x}\xi_{\pi}([\tau^{\prime}])\Lambda_{\tau^{\prime}}(x-e^{\ell_{\tau^{\prime}}})^{k}$$
where the sum covers all closed orbits of $\psi_{t}$. For any $x\ge e$, we define a function
$$M_{x,k}:\mathbb{C}\to\mathbb{C},\quad M_{x,k}(s)=\dfrac{x^{s+k}}{\prod_{j=0}^{k}(s+j)}.$$
It is easy to see $M_{x,k}$ is analytic in the region of $s=a+ib$ with $a>0$. By the definition of $L_{\pi}$, we have the following, see also \cite{Pol08}.

\begin{Lemma}\label{Lemma 3.6.4}
For any $k\in\mathbb{N}^{+}$, any $x\ge e$ and any $d>1$, we have
$$N_{\pi,k}(x)=\dfrac{k!}{ih_{top}}\int_{d-i\infty}^{d+i\infty}\dfrac{L_{\pi}^{\prime}(s)}{L_{\pi}(s)}M_{x,k}(s)ds.$$
\end{Lemma}

By Lemma \ref{Lemma 3.6.3}, $L_{\pi}$ is analytic and non-zero in the region of $s=a+ib$ with $a>1-b_{\pi}^{-C_{41}}$. In particular, we have $\frac{L_{\pi}^{\prime}}{L_{\pi}}M_{x,k}$ is analytic in this region as well. Thus, we can move the curve of integration in Lemma \ref{Lemma 3.6.4} to $\mathcal{C}_{\pi}=\{a= 1-b_{\pi}^{-C_{41}}:b\in\mathbb{R}\}$  to obtain the following result.

\begin{Lemma}\label{Lemma 3.6.5}
For any $n\ge1$ there exists $C_{43}>0$ such that for any $x>e$,
$$|N_{\pi, \lfloor C_{43}\rfloor}(x)|\le C_{43}|\lambda_{\pi}|^{C_{43}}\dfrac{x^{\lfloor C_{43}\rfloor+1}}{(\log x)^{n}}.$$
\end{Lemma}
\begin{proof}
We can bound
$$
\begin{aligned}
&|N_{\pi,k}(x)|\\
=&\dfrac{k!}{h_{top}}\bigg|\int_{\mathcal{C}_{\pi}}\dfrac{L_{\pi}^{\prime}(s)}{L_{\pi}(s)}M_{x,k}(s)ds\bigg|\\
\le&C_{41}x^{k+1}\int_{|b|\ge1}b_{\pi}^{C_{41}}\dfrac{1}{|b|^{k+1}}x^{-b_{\pi}^{-C_{41}}}db+C_{41}2x^{k+1}\int_{|b|\le1}b_{\pi}^{C_{41}}x^{-b_{\pi}^{-C_{41}}}db\\
\le&C_{44}\dfrac{x^{k+1}}{(\log x)^{n}}\int_{|b|\ge1}b_{\pi}^{C_{45}}\dfrac{1}{|b|^{k+1}}db+C_{44}2\dfrac{x^{k+1}}{(\log x)^{n}}\int_{|b|\le1}b_{\pi}^{C_{45}}db
\end{aligned}
$$
for some uniform constant $C_{44}, C_{45}>0$. Since $b_{\pi}=|b|+C_{15}|\lambda_{\pi}|^{1+m_{G}/2}$, we can then choose $k=\lfloor C_{43}\rfloor$ for some large enough $C_{43}>0$ to complete the proof.
\end{proof}

\begin{proof}[\textbf{Proof of Proposition \ref{Prop 3.6.1}}]
This is a standard argument, as discussed in \cite{Pol01} or \cite{Pol08}. We define 
$$\Psi_{\pi}(T):=\sum_{\tau^{\prime}\in V(T)}\xi_{\pi}([\tau^{\prime}])\ell_{\tau^{\prime}}\quad\text{and}\quad\Phi_{\pi}(T):=\sum_{\tau^{\prime}\in V(T)}\xi_{\pi}([\tau^{\prime}]).$$
By the argument in \cite{Pol08}, the estimate of $N_{\pi,\lfloor C_{43}\rfloor}$ in Lemma \ref{Lemma 3.6.5} implies the following estimate on $\Psi_{\pi}(T)$:
\begin{equation}\label{3.6.5}
	\Psi_{\pi}(T)=O\bigg(|\lambda_{\pi}|^{C_{43}}\dfrac{e^{h_{top}T}}{T^{n/2}}\bigg).
\end{equation}
The relationship between $\Psi_{\pi}(T)$ and $\Phi_{\pi}(T)$ is as follows:
$$\Phi_{\pi}(T)=\int^{T}_{1}\dfrac{1}{u}d\Psi_{\pi}(u)+O(|\xi_{\pi}|_{\infty}).$$
By the above estimate and \eqref{3.6.5}, we have
$$
\begin{aligned}
	\Phi_{\pi}(T)=&\dfrac{\Psi_{\pi}(u)}{u}\bigg|_{u=1}^{T}+\int^{T}_{1}\Psi_{\pi}(u)\dfrac{1}{u^{2}}du+O(|\xi_{\pi}|_{\infty})\\
	=&O\bigg(|\lambda_{\pi}|^{C_{43}}\dfrac{e^{h_{top}T}}{T^{n/2+1}}\bigg)+O\bigg(|\lambda_{\pi}|^{C_{43}}\dfrac{e^{h_{top}T}}{T^{n/2+2}}\bigg)\int^{T}_{1} \frac{1}{u^2}du+O(|\xi_{\pi}|_{\infty})\\
	=&O\bigg(|\xi_{\pi}|_{\infty}|\lambda_{\pi}|^{C_{43}}\dfrac{e^{h_{top}T}}{T^{n/2+1}}\bigg).
\end{aligned}
$$
Now, by the prime orbit theorem \cite{Par90}, we have
$$\dfrac{1}{\# V(T)}\sum_{\tau^{\prime}\in V(T)}\xi_{\pi}([\tau^{\prime}])=\dfrac{1}{\# V(T)}\Phi_{\pi}(T)=O(|\xi_{\pi}|_{\infty}|\lambda_{\pi}|^{C_{43}}T^{-n/2})$$
which completes the proof by Lemma \ref{Lemma 2.7.10}.
\end{proof}




\Addresses

 \end{document}